\documentclass[12pt, reqno]{amsart}

\usepackage[margin=1.3 in]{geometry}

\usepackage[utf8]{inputenc}
\usepackage[T1]{fontenc}
\usepackage[english]{babel}

\usepackage{graphicx}
\usepackage{tikz-cd}
\usepackage{import}
\usepackage{wrapfig}
\usepackage{tabularx}

\usepackage{amsmath, amsthm, amssymb, amsfonts}
\usepackage{enumitem}
\usepackage{mathrsfs}

\usepackage{xcolor}
\definecolor{antiquefuchsia}{rgb}{0.57, 0.36, 0.51}
\definecolor{azure}{rgb}{0.0, 0.5, 1.0}

\usepackage[colorlinks = true, linkcolor = azure, citecolor = antiquefuchsia]{hyperref}

\usepackage{comment}
\usepackage{todonotes}
\theoremstyle{plain}
\newtheorem{theorem}{Theorem}
\newtheorem{proposition}[theorem]{Proposition}
\newtheorem{corollary}[theorem]{Corollary}
\newtheorem{lemma}[theorem]{Lemma}

\theoremstyle{definition}
\newtheorem{definition}[theorem]{Definition}
\newtheorem{example}[theorem]{Example}

\newtheorem{remark}[theorem]{Remark}

\numberwithin{theorem}{section}
\numberwithin{equation}{section}

\newcommand{\Z}{\mathbb{Z}} % integers
\newcommand{\N}{\mathbb{N}} % integers
\newcommand{\R}{\mathbb{R}} % real numbers
\newcommand{\C}{\mathbb{C}} % complex numbers
\newcommand{\DD}{\mathbb{D}} % unit disk
\newcommand{\HH}{\mathbb{H}}
\renewcommand{\O}{\mathcal{O}}
\newcommand{\he}{\text{\rm Her}} 
\newcommand{\M}{\mathfrak{M}}
\newcommand{\V}{\mathbb{V}}

\newcommand{\1}{{\bf 1}}
\newcommand{\extend}{\bullet}
\renewcommand{\d}{\mathrm{d}} % differential
\newcommand{\g}{\mathfrak{g}} % generic Lie-algebra
\newcommand{\gl}{\mathfrak{gl}} % Lie-algebra of Gl(m)
\renewcommand{\u}{\mathfrak{u}} % Lie-algebra of U(m)
\newcommand{\so}{\mathfrak{so}} % Lie-algebra of SO(m)

\DeclareMathOperator{\End}{End} % Endomorphisms
\DeclareMathOperator{\supp}{supp}

\newcommand{\A}{\mathbb{A}}
\newcommand{\Id}{\mathrm{Id}}
\newcommand{\B}{\mathbb{B}}
\newcommand{\BB}{\mathfrak{B}}
\renewcommand{\H}{\mathcal{H}}
\newcommand{\I}{\mathbb{I}}
\renewcommand{\L}{\mathcal{L}}
\newcommand{\U}{\mathbb{U}}
\newcommand{\T}{\mathcal{T}}
\newcommand{\D}{\mathscr{D}} 
\newcommand{\G}{\mathbb{G}}
\renewcommand{\S}{\mathcal{S}} 
\newcommand{\even}{\mathrm{ev}}
\newcommand{\interior}{\mathrm{int}}

\newcommand{\act}{\triangleleft}
\newcommand{\lact}{\triangleright}

\DeclareMathOperator{\spn}{span}
\usepackage{physics}

\title[The TOG-Principle for simple surfaces]{The Transport Oka-Grauert Principle for simple surfaces}

\author[J. Bohr]{Jan Bohr}
\address{ 
Mathematical Institute of the University of Bonn,
Endenicher Allee 60, 53115 Bonn, Germany}
\email {bohr@math.uni-bonn.de}

\author[G.P. Paternain]{Gabriel P. Paternain}
\address{ Department of Pure Mathematics and Mathematical Statistics,
University of Cambridge,
Cambridge CB3 0WB, UK}
\email {g.p.paternain@dpmms.cam.ac.uk}

\date{\today}

\begin{document}
\maketitle

%\tableofcontents

\begin{abstract}
This article considers the attenuated transport equation on Riemannian surfaces in the light of a novel twistor correspondence under which matrix attenuations correspond to holomorphic vector bundles on a complex surface. The main result is a transport version of the classical Oka-Grauert principle and states that the twistor space of a simple surface supports no nontrivial holomorphic vector bundles. This solves an open problem on the existence of matrix holomorphic integrating factors on simple surfaces and is applied to give a range characterisation for the non-Abelian X-ray transform.

The main theorem is proved using the inverse function theorem of Nash and Moser and the required tame estimates are obtained from recent results on the injectivity of attenuated X-ray transforms and microlocal analysis of the associated normal operators.

%We show that some group actions are transitive. This allows to draw nice diagrams.
\end{abstract}

\section{Introduction}

Inverse problems play a central role in different parts of analysis and geometry. In these problems, there is often an underlying PDE of transport type involving the geodesic vector field of a Riemannian manifold that drives the behaviour of various X-ray transforms.
In recent years, a series of papers has culminated in general injectivity results (modulo gauge transformations) for 
a fundamental class of {\it nonlinear} X-ray transforms on simple Riemannian surfaces.
One goal of this paper is to give a characterisation of the {\it range} for this class of transforms
 via a theory of `holomorphic integrating factors’. The result is reminiscent of the Ward correspondence for anti-self-dual Yang-Mills fields, but without solitonic degrees of freedom.  The range characterisation turns out to be equivalent, via a novel twistor correspondence, to a non-existence theorem for holomorphic vector bundles on certain complex surfaces, resembling the classical Oka-Grauert theorem. Remarkably, the proof of this complex geometric result uses essentially both the theory of transport equations and microlocal analysis.

%Inverse problems play a central role in different parts of analysis and geometry. A particular inverse problem that has received significant attention in the recent years is the inversion of the geodesic X-ray transform on Riemannian manifolds, as well as its attenuated and non-Abelian variants. In two dimensions and in the absence of trapped sets and conjugate points (e.g., on {\it simple surfaces}), the understanding of injectivity properties of these transforms is now fairly complete. %Key ideas underpinning these results include the reformulation as transport problem admitting suitable Pestov-type energy identities, as well as the usage of Fourier analysis %connecting to complex analysis by expanding functions in terms of a Fourier series
%along the fibres of the unit tangent bundle.

%One goal of this paper is to give a characterisation of the {\it range} of a significant class of {\it nonlinear} X-ray transforms. The result is reminiscent of the Ward correspondence for anti-self-dual Yang-Mills fields in \cite{Mas06},
%with solitonic degrees of freedom given -- via a novel twistor-type correspondence -- by the moduli space of holomorphic vector bundles on an associated complex surface. We show that this moduli space is trivial for simple surfaces, a result that resembles the classical Oka-Grauert principle from complex geometry. Strikingly, this complex geometric result is proved using 
%in an essential way the theory of transport equations and microlocal analysis.

We now describe the setting of the paper in more detail. 
Let $(M,g)$ be a compact Riemannian surface with smooth boundary $\partial M$. Let $SM=\{(x,v)\in TM: g(v,v)=1\}$ be the unit tangent bundle and $X$ the geodesic vector field on $SM$. This paper addresses three aspects related to the transport equation 
\begin{equation}\label{ate}
(X+\A)R=0\text{ on } SM
\end{equation}
with matrix attenuations  $\A\in C^\infty(SM,\C^{n\times n})$: 1) The existence of special solutions to \eqref{ate}, called (matrix-)holomorphic integrating factors.  2) A twistor correspondence between attenuations $\A$ and holomorphic vector bundles on a complex surface. 3) A range characterisation for the non-Abelian X-ray transform, which arises from boundary measurements of solutions to \eqref{ate}.

%complexification of \eqref{ate} which recasts attenuations as holomorphic vector bundles on a complex surface. Third, a characterisation of the range of the {\it non-Abelian X-ray transform}, which arises from boundary measurements of solutions to \eqref{ate}.

These considerations are closely related and are motivated by an inverse problem that we now describe. Assume that $\partial M$ is strictly convex and that $M$ is non-trapping, i.e.~all geodesics in $M$ reach $\partial M$ in finite time.  We denote with $\nu$ the inward pointing unit normal to $\partial M$ and partition the boundary of $SM$ into $\partial SM= \partial_+SM \cup \partial_-SM$, where
 $$\partial_\pm SM=\{(x,v)\in SM: x\in \partial M,\;\; \pm g( \nu(x),v) \ge 0\}.$$ Then, by standard ODE theory, equation \eqref{ate} admits a unique continuous solution $R=R^0:SM\rightarrow GL(n,\C)$, differentiable along the geodesic flow, with $R^0 = \Id$ on $\partial_-SM$ and we define the {\it scattering data} of $\A\in C^\infty(SM,\C^{n\times n})$ by 
\begin{equation}
C_\A := R^0\vert_{\partial_+SM}\in C^\infty(\partial_+SM,GL(n,\C)).
\end{equation}
The nonlinear map $\A \mapsto C_\A$ is called the {\it non-Abelian X-ray transform}.
The inverse problem of recovering an attenuation $\A$ from measurements of its scattering data $C_\A$ has been 
subject of a number of recent papers \cite{PSU12,PaSa20,MNP19} (with earlier in work \cite{Ver91,Sha00, FiUh01, Nov02,Esk04}) and the question of injectivity is now well understood in the following setting: Let $G\subset GL(n,\C)$ be a Lie group with Lie algebra $\g$ and suppose that $\A$ is given in terms of a 1-form $A\in \Omega^1(M,\g)$ and a matrix field $\Phi\in C^\infty(M,\g)$ as
\begin{equation}
\A(x,v) =A_x(v) + \Phi(x).
\end{equation}
We then write $\A=(A,\Phi)$ (referred to as $\g$-pair) and $C_\A=C_{A,\Phi}$ and note that the scattering data of a $\g$-pair is a $G$-valued function. Two special cases are of particular importance: If $\Phi=0$, then $C_A=C_{A,0}$ describes parallel transport of the connection that $A$ induces on the trivial bundle $M\times \C^n$. If $A=0$ and $\g=\so(3)$, then $C_\Phi = C_{0,\Phi}$ arises as measurement data in a novel imaging method called {\it Polarimetric Neutron Tomography} \cite{DL_20,MNP19,nature1,nature2}.

A surface $(M,g)$ is called {\it simple}, if $\partial M$ is strictly convex and $M$ is non-trapping and free of conjugate points. 

\begin{theorem}[Paternain, Salo, Uhlmann -- 2012 \& 2020]\label{PSU} Let $(M,g)$ be a simple surface and $G=U(n)$ \cite{PSU12} or $G=GL(n,\C)$ \cite{PaSa20}. Suppose that two $\g$-pairs $(A,\Phi)$ and $(B,\Psi)$ have the same scattering data, $C_{A,\Phi}=C_{B,\Phi}\in C^\infty(\partial_+SM,G)$. Then
\begin{equation}
(B,\Psi) = (A,\Phi)\act\varphi := (\varphi^{-1} \d \varphi + \varphi^{-1} A \varphi,\varphi^{-1} \Phi \varphi)
\end{equation}
for some gauge $\varphi\in C^\infty(M,G)$ with $\varphi = \Id$ on $\partial M$.\qed
\end{theorem}

Here $U(n)$ is the unitary group, with Lie algebra $\u(n)=\{T\in \C^{n\times n}:T^*=-T\}$ consisting of skew-Hermitian matrices. On manifolds of dimension $\ge 3$ a similar result was obtained in \cite{PSUZ19}, using the groundbreaking techniques of Uhlmann and Vasy \cite{UhVa16} that also underpin the recent solution of the boundary rigidity problem \cite{SUV21}.
For a more detailed account on the history and applications of the non-Abelian X-ray transform we refer to \cite{PaSa20,Nov19} as well as the recent monograph \cite{PSU20}.

%The proof of this theorem first reduces the problem to a proving injectivity (up to gauge) of the linear {\it attenuated X-ray transform}  $I_{A,\Phi}$ via a pseudolinearisation identity. For $\u(n)$-pairs, injectivity of $I_{A,\Phi}$ (up to gauge) was established in \cite{PSU12}, following the scheme of proof of the well known Kodaira vanishing theorem from complex geometry, in particular involving a Pestov type energy identity. For general $\gl(n,\C)$-pairs, where the resulting energy identity contains uncontrollable terms, injectivity was established in \cite{PaSa20} using factorisation theorems from loop groups.

\subsection{Holomorphic integrating factors} Our first contribution concerns {\it \mbox{(matrix-)} holomorphic integrating factors (HIF)}, which were initially sought after as a tool to prove Theorem \ref{PSU} and are now used %as key ingredient 
to obtain the range characterisations in \textsection \ref{introrange}.

To define holomorphic integrating factors we use the fact that every smooth function $F:SM\rightarrow \C^{n\times n}$ has a unique decomposition $F=\sum_{k\in \Z}F_k$ in terms of its vertical Fourier modes (see \textsection\ref{prelim} for more details). Then $F$ is called {\it fibrewise holomorphic} iff $F_k=0$ for $k<0$ and we define
\begin{equation}\label{defG}
\G =\{F\in C^\infty(SM,GL(n,\C)):F\text{ and } F^{-1} \text{ are fibrewise holomorphic}\}.
\end{equation}
\begin{definition} %A holomorphic integrating factor for $\A \in C^\infty(SM,\C^{n\times n})$ is a function $F\in \G$ which satisfies the equation $(X+\A)F=0$ on $SM$.
A function $F\in \G$ is called a holomorphic integrating factor for the attenuation $\A\in C^\infty(SM,\C^{n\times n})$, if it satisfies the equation $(X+\A)F=0$ on $SM$.
\end{definition}

If an attenuation $\A$ admits holomorphic integrating factors, then necessarily its Fourier modes vanish for $k<-1$, which is to say that $\A$ is a member of the set
\begin{equation}\label{defmho}
\mho=\{\A\in C^\infty(SM,\C^{n\times n}): \A_k=0 \text{ for } k<-1\}.
\end{equation}
Note that $\mho$ in particular contains all  $\gl(n,\C)$-pairs $\A=(A,\Phi)$, which have nonzero Fourier modes only for $\vert k \vert \le 1$.
We prove the following result:

\begin{theorem}\label{mainhif}
On a simple surface $(M,g)$ every attenuation $\A\in \mho$ admits holomorphic integrating factors.
\end{theorem}

In the Abelian case ($n=1$) this theorem was established in \cite{SaUh11} and has since become an indispensable tool in the treatment of attenuated and tensor tomography. The non-Abelian case is much harder and has so far only been addressed in a Euclidean setting. There a weak form of HIF (with $F$ only being continuous)  was constructed by Novikov \cite{Nov02} for $\A$ sufficiently small and, without smallness assumption, by Eskin and Ralston  \cite{EsRa04}. The question of whether smooth matrix HIF exist on simple surfaces has since been open and we can now give an affirmative answer.%, relying on the injectivity result in \cite{PSU12}.

The idea behind the proof of Theorem \ref{mainhif} is conceptually quite simple. The set $\G$ from \eqref{defG} forms a group and acts on $\mho$ from the right by
\begin{equation}\label{rule}
\A\act F = F^{-1} XF + F^{-1} \A F,
\end{equation}
such that the orbit of $0\in \mho$ contains precisely those attenuations $\A$ that admit holomorphic integrating factors.
Theorem \ref{mainhif} can thus be reformulated as transitivity of this group action. Using results on the attenuated X-ray transform from \cite{PSU12,PaSa20} and microlocal analysis of the associated normals operators,
we show that the derivative of $F\mapsto \A\act F$ at $\Id \in \G$ is surjective for all $\A\in \mho$. After establishing appropriate tame estimates, we use this together with the inverse function theorem of Nash and Moser to show that all orbits of $\G$ are open. As $\mho$ is connected, the action must be transitive.

\subsection{Twistor correspondence} \label{twistorintro} The second purpose of this article is to promote a novel viewpoint on transport equations as in \eqref{ate} by relating them to holomorphic vector bundles on a {\it twistor space} $Z$ associated to $(M,g)$. This is inspired by Penrose's twistor programme \cite{Pen77} and the paradigm that solutions to integrable systems should be parametrised by complex geometric objects \cite{AHS78,LeB80, Hit80,MaWo96}.

The twistor space $Z$ can be constructed for any oriented surface $(M,g)$ and, as a smooth manifold, equals the unit disk bundle
\begin{equation*}
Z=\{(x,v)\in TM:g(v,v)\le 1\}.
\end{equation*}
We equip $Z$ with a complex structure that turns $Z^{\interior}$ into a classical complex surface, and that degenerates at $SM\subset \partial Z$.
%The twistor space $Z$ can be constructed for any surface $(M,g)$ and is a real $4$-manifold with $SM\subset \partial Z$, such that the interior $Z^{\interior}$ is a complex surface, %with Cauchy Riemann equations degenerating to a transport equation on $SM\subset\partial Z$. with complex structure degenerating at $\partial Z$. 
Postponing precise definitions to \textsection\ref{twistor}, we note that standard constructions from complex geometry can be carried out `smooth up to the boundary', in particular there is a natural moduli space
\begin{equation}
\M=\M_n(Z)=\{\text{Holomorphic rank } n \text{ vector bundles on } Z\}/\sim,
\end{equation}
where $\sim$ denotes isomorphism of holomorphic vector bundles. We establish several correspondence principles (see Propositions \ref{tcA} and \ref{tcB}) which relate the complex geometry on $Z$ to transport problems on $SM$. In particular, if $(M,g)$ is diffeomorphic to a disk (e.g.,~when it is simple), we prove that there is an isomorphism
%which allow to view certain fibrewise holomorphic functions as true holomorphic functions and establish a natural isomorphism 
 %valid whenever $(M,g)$ is homeomorphic to a disk, which states that
\begin{equation}\label{corrM}
\M \cong \mho / \G,
\end{equation}
where the right hand side is the quotient space under the action defined in  \eqref{rule}. In light of \eqref{corrM}, we may reformulate Theorem \ref{mainhif} as follows (see also Theorem \ref{tog} where the result is stated in context):

\begin{theorem}[Transport Oka-Grauert principle]\label{maintog} Let $Z$ be the twistor space of a simple surface $(M,g)$. Then $\M_n(Z)=0$, that is, $Z$ supports no nontrivial holomorphic vector bundles.
\end{theorem}

This result is reminiscent of the Oka-Grauert principle in complex geometry (cf.~\cite{Gra58,For11}), which states that on a Stein manifold the classification of continuous and holomorphic vector bundles coincide. This, amongst other similarities elaborated on in \textsection \ref{twistor}, suggests the following slogan:
\begin{equation*}
\text{\it Twistor spaces of simple surfaces behave like (contractible) Stein surfaces.}
\end{equation*}
It is tempting to try and prove Theorem  \ref{maintog} by complex geometric methods, thus deriving Theorem  \ref{mainhif} as corollary. However, there are several obstacles to this: First one would need to show that $Z^\interior$ is indeed a Stein surface -- this is easily seen if $(M,g)$ is flat (see Lemma \ref{flatstein}), but remains challenging for other geometries. Second, one has to deal with the degeneracy of the complex structure at $\partial Z$, which is a highly nontrivial task. The work of 
Eskin and Ralston \cite{EsRa04} can be interpreted as such  a `desingularisation', similar to the one performed, albeit in a different setting, by LeBrun and Mason in \cite{LeMa02}. We discuss this approach in more detail in  \textsection \ref{relatedwork}.

For general simple surfaces it seems to be preferable to prove Theorem \ref{maintog} using transport techniques, requiring however, the injectivity result in \cite{PSU12} {\it a priori}. It is curious to note that the techniques in \cite{PSU12} were in turn inspired by the Kodaira vanishing theorem from complex geometry. 

In \cite{BLP23}, we investigate the behaviour of holomorphic vector bundles in a non-simple setting, specifically for the twistor space $Z$ of a closed Anosov surface of genus $\mathsf g$. There we demonstrate that $\M_1(Z,Z^\interior)$, the moduli space of holomorphic line bundles on $Z$ up to equivalence in the interior, is isomorphic to $\C^\mathsf{g}/\Z^{2\mathsf{g}}\times \C$.

\subsection{Range characterisation}\label{introrange} Finally, we provide a characterisation of the range of the non-Abelian X-ray transform $\A\mapsto C_\A$ in terms of boundary objects. This is inspired by the range characterisations for the linear X-ray transform by Pestov and Uhlmann \cite{PeUh04} and the subsequent work for attenuated X-ray transforms in \cite{PSU15,AMU18}.

Our main result concerns the range of $(A,\Phi)\mapsto C_{A,\Phi}$ for $\u(n)$-pairs and is formulated in terms of a `boundary operator'
\begin{equation}
P:C_\alpha^\infty(\partial_+SM,\he^+_n) \rightarrow C^\infty(\partial_+SM,U(n)),
\end{equation}
where $\he^+_n\subset \C^{n\times n}$ denotes Hermitian positive definite  matrices. Postponing precise definitions to \textsection \ref{range}, we note that the domain of $P$ and the operator $P$ itself are defined in terms of the following objects:
\begin{itemize}
\item The {\it scattering relation} $\alpha:\partial_+ SM\rightarrow \partial_- SM$ of $(M,g)$, sending starting point and direction of a geodesic to end point and direction.
\item A nonlinear type of Hilbert transform 
\begin{equation}\H^+:C^\infty(\partial SM,\he^+_n)\rightarrow C^\infty(\partial SM,GL(n,\C)),
\end{equation}
defined in terms of the Birkhoff factorisation in loop groups \cite{PS86}, see \textsection\ref{hilbertsection}.
\end{itemize}

\begin{theorem}[Range characterisation for $\u(n)$-pairs]\label{mainrange}
Suppose that ($M,g)$ is a simple surface (or more generally, that $\M=0$). Then an element $q\in C^\infty(\partial_+SM,U(n))$ lies in the range of $\{\u(n)\text{-pairs}\}\ni (A,\Phi)\mapsto C_{A,\Phi}$ if and only if 
\begin{equation*}
\hspace{3em} q=h \cdot Pw \cdot (h^{-1}\circ \alpha)
\end{equation*}
for some $w\in C_\alpha^\infty(\partial_+SM,\he^+_n)$ and a contractible map $h\in C^\infty(\partial M,U(n))$ (i.e. the element induced in $\pi_{1}(U(n))$ is trivial.)
\end{theorem}

Together with Theorem \ref{PSU}, we now have a complete understanding of injectivity and range properties of the (nonlinear) non-Abelian X-ray transform $(A,\Phi)\mapsto C_{A,\Phi}$ on simple surfaces.

The theorem is restated as Theorem \ref{rangepairs}, where it is complemented by a number of further characterisations concerning in particular the range of $\Phi \mapsto C_\Phi$
and $A\mapsto C_A$, as well as the case of $\gl(n,\C)$-valued attenuations. Also a characterisation in the non-simple case is discussed.  For precise statements we refer to \textsection \ref{range}.

Let us illustrate the idea behind our range characterisations with the case of the transform $C^\infty(M,\u(n))\ni \Phi\mapsto C_\Phi$ (cf.\,Theorem \ref{rangefields} below). To produce an element in the range, we start with some function $w\in C_\alpha^\infty(\partial_+SM,\he^+_n)$. This can be extended to a smooth first integral $w^\sharp:SM\rightarrow \he^+_n$, constant along the geodesic flow. %uniquely defined by $Xw^\sharp =0$ on $SM$.
By Birkhoff's factorisation theorem, $w^\sharp =F^*F$ for some $F\in \G$. We now make the assumption that the $0$th Fourier mode of $F$ satisfies $F_0=\Id$, in which case the factorisation is unique. Consider $q:=F( F^{-1}\circ \alpha)\vert_{\partial_+SM} \in C^\infty(\partial_+SM,U(n))$. Then $q$ is given solely in terms of boundary data and in fact equals $q=C_\Phi$, the scattering data of the matrix field 
\begin{equation}\label{hiddenhif} \Phi = -(XF)F^{-1}\in C^\infty(M,\u(n)).
\end{equation} In particular, $q$ lies in the range of $\Phi\mapsto C_\Phi$. We prove that on a simple surface all elements in the range arise in this way by showing that {\it every} matrix field $\Phi$ is of the form \eqref{hiddenhif}. This in turn is a consequence of Theorem \ref{mainhif}.

Theorem \ref{mainrange} bears a striking resemblance with the Ward correspondence for the anti-self-dual Yang-Mills (ASDYM) equation %in signature $2+2$, described 
by Mason in \cite{Mas06}: there a one-to-one correspondence is set up between solutions to the ASDYM equation on $\tilde{\mathbb{M}} = S^2\times S^2$ (with split signature) on the one hand and pairs $(E,H)$ on the other hand, where $E$ is a holomorphic vector bundle on a complex twistor space associated with  $\tilde{\mathbb{M}}$ and $H$ is a Hermitian metric on $E$, restricted to a real subspace. The two `parameters' $E$ and $H$ are also referred to as {\it solitonic} and {\it radiative/dispersive} degrees of freedom, respectively. Back to Theorem \ref{mainrange}, and ignoring the gauge $h$, we see that the range  of the non-Abelian X-ray transform is also parametrised by a Hermitian metric, given by $w\in C^{\infty}_{\alpha}(\partial_+SM, \he^+_n)$. Notably, there are no solitonic degrees of freedom, which is in line with the Transport Oka-Grauert principle in Theorem \ref{maintog}.

\begin{comment} Following the same recipe, we obtain several more range characterisations. In particular we consider the range of $\mathcal{A}\ni \A\mapsto C_\A$ for $\mathcal{A}$ including the following classes of attenuations:
\begin{equation}
\mathcal{A} = \mho, ~\{\u(n)\text{-Higgs fields}\}, ~\{\u(n)\text{-connections}\},~ \{\gl(n,\C)\text{-pairs}\},~\text{etc.}
\end{equation}
 Finally, we give a (less concrete) range characterisation for $\u(n)$-pairs when $\M\neq 0$, i.e.~in the presence of solitonic degrees of freedom. Full statements are given in \textsection \ref{range}.
\end{comment}

At last, let us mention a potential application of our range characterisations. In the context of Polarimetric Neutron Tomography, it has been of recent interest to rigorously study statistical algorithms for recovering a matrix field $\Phi$ from noisy measurements of $C_\Phi$ \cite{MNP19,MNP20}. In particular it was shown in \cite{BoNi21}, that if $M$ is the Euclidean unit disk, then $\Phi$ can be recovered by a statistical algorithm {\it in polynomial time}, provided there is a suitable initialiser. Knowing the range of $\Phi\mapsto C_\Phi$ is a possible starting point to construct a computable intitialiser -- we hope to address this in furture work.

\subsection*{Acknowledgements} We would like to thank 
Maciej Dunajski, Claude LeBrun, Thomas Mettler, Fran\c{c}ois Monard, Richard Nickl and Ivan Smith for their helpful comments. We are also very grateful to the referee for corrections and suggestions that improved the presentation. JB was supported by the EPSRC Centre for Doctoral Training and the Munro-Greaves Bursary of Queens' College Cambridge.
GPP was supported by EPSRC grant EP/R001898/1.

\section{Preliminaries}\label{prelim}
Here we provide some well-known background material which may be found in \cite{GK,SiTh76}; for a recent presentation and its relevance to geometric inverse problems in two dimensions we refer to \cite{PSU20}.
Throughout, $(M,g)$ is a compact, oriented two dimensional Riemannian manifold with smooth and possibly empty boundary
$\partial M$.

The unit sphere bundle $SM$ is a compact $3$-manifold with boundary $\partial SM=\{(x,v)\in SM:x\in \partial M\}$, containing  $\partial_0SM:=\partial_+SM\cap \partial_-SM$ as submanifold. %The geodesic flow on $SM$ is denoted by $\varphi_t$, with infinitesimal generator being the geodesic vector field $X$ and with exit time  $\tau:SM\rightarrow [0,\infty]$ given by $\tau(x,v)=\inf\{t\ge 0:\varphi_t(x,v)\in \partial SM\}$. 
The geodesic vector field $X$ is the infinitesimal generator of the geodesic flow $\varphi_t$ on $SM$ and for $(x,v)\in SM$ we denote with $\tau(x,v)\in[0,\infty]$ the first time $t\mapsto \varphi_t(x,v)$ exits $SM$.
%The exit time $\tau:SM\rightarrow [0,\infty]$ is defined by $\tau(x,v)=\inf\{t\ge 0:\varphi_t(x,v)\in \partial SM\}$. 
The vertical vector field $V$ is defined as the infinitesimal generator of the circle action that the orientation of $M$ induces on the fibres of $SM$. The pair $X,V$ can be completed to a global frame of $T(SM)$ by considering the vector field $X_\perp:=[X,V]$. There are two further structure equations given by $[V,X_\perp]=X$ and $[X,X_\perp]=-KV$, where $K$ is the Gaussian curvature of $M$. The Sasaki metric on $SM$ is the unique Riemannian metric for which $\{X,X_\perp,V\}$ is an orthonormal frame and the volume for for this metric is denoted by $\d\Sigma^3$. The induced area form on $\partial SM$ is denoted by $\d\Sigma^2$.

 %Let $X$ denote the geodesic vector field, i.e.\,the infinitesimal generator of the geodesic flow $\varphi_t$ acting on the unit
%circle bundle $SM$. The latter is a compact 3-manifold with boundary given by $\partial SM=\{(x,v)\in SM:x\in \partial M\}$, containing the submanifold $\partial_0SM=\partial_+SM\cap \partial_-SM$.
%Since $M$ is assumed oriented there is a circle action on the fibers of $SM$ with infinitesimal generator $V$ called the {\it vertical vector field}. It is possible to complete the pair $X,V$ to a global frame
%of $T(SM)$ by considering the vector field $X_{\perp}:=[X,V]$. There are two additional structure equations given by $X=[V,X_{\perp}]$ and $[X,X_{\perp}]=-KV$
%where $K$ is the Gaussian curvature of the surface. Using this frame we can define a Riemannian metric on $SM$ -- called the {\it Sasaki metric} -- by declaring $\{X,X_{\perp},V\}$ to be an orthonormal basis and the volume form of this metric will be denoted by $\d\Sigma^3$. The fact that $\{ X, X_{\perp}, V \}$ are orthonormal together with the commutator formulas implies that the Lie derivative of $\d\Sigma^3$ along the three vector fields vanishes. The induced area form on $\partial SM$ will be denoted by $\d \Sigma^2$.

If $x = (x_1, x_2)$ are isothermal coordinates in $(M,g)$ so that the metric has the form $g = e^{2\lambda(x)} \,dx^2$ and if  $\theta$ is the angle between $v$ and $\partial_{x_1}$, then in the $(x,\theta)$ coordinates in $SM$ the vector fields have the following explicit formulas:
\begin{comment}
\begin{align}
X &= e^{-\lambda}\left(\cos\theta \partial_{x_1}+
\sin\theta \partial_{x_2}+
\left(-\partial_{x_1}\lambda \sin\theta+\partial_{x_2}\lambda\cos\theta\right)\partial_\theta\right), \label{defX}\\
X_{\perp} &= -e^{-\lambda}\left(-\sin\theta\partial_{x_1}+
\cos\theta \partial_{x_2}-
\left(\partial_{x_1}\lambda \cos\theta+\partial_{x_2}\lambda \sin\theta\right)\partial_\theta\right), \\
V &= \partial_\theta.
\end{align}
\end{comment}
\begin{align}
X &= e^{-\lambda}\left(\cos\theta\frac{\partial}{\partial x_{1}}+
\sin\theta\frac{\partial}{\partial x_{2}}+
\left(-\frac{\partial \lambda}{\partial x_{1}}\sin\theta+\frac{\partial\lambda}{\partial x_{2}}\cos\theta\right)\frac{\partial}{\partial \theta}\right), \label{defX}\\
X_{\perp} &= -e^{-\lambda}\left(-\sin\theta\frac{\partial}{\partial x_{1}}+
\cos\theta\frac{\partial}{\partial x_{2}}-
\left(\frac{\partial \lambda}{\partial x_{1}}\cos\theta+\frac{\partial \lambda}{\partial x_{2}}\sin\theta\right)\frac{\partial}{\partial \theta}\right), \\
V &= \frac{\partial}{\partial\theta}.
\end{align}

The space $L^2(SM,\C^n)$ is defined in terms of the measure $\d \Sigma^3$ and the standard Hermitian inner product on $\C^n$. There is an orthogonal decomposition $L^{2}(SM,\C^n)=\oplus_{k\in\mathbb Z}H_{k}$, where $H_k$ is the eigenspace of $-iV$ corresponding to the eigenvalue $k$. A function $u\in L^{2}(SM,\C^n)$ has a Fourier series expansion
\begin{equation} u=\sum_{k=-\infty}^{\infty}u_{k},\end{equation}
where $u_{k}\in H_k$. For $k\in \Z$ and $I\subset \Z$ we define
\begin{equation}\label{defOmega}
\Omega_{k}=C^{\infty}(SM,\C^n)\cap H_{k}\quad \text{and}\quad \oplus_{k\in I}\Omega_k =  C^\infty(SM,\C^n)\cap \left( \oplus_{k\in I}H_k \right) .
\end{equation}

\begin{comment}
Given functions $u,v:SM\to \C^n$ we consider the
inner product
\begin{equation}(u,v) =\int_{SM}\langle u,v\rangle_{\mathbb \C^n}\,d\Sigma^3.\end{equation}
%Since $X, X_{\perp}, V$ are volume preserving we have $(Vu,v) = -(u,Vv)$ for any $u, v$ in $C^{\infty}(SM,\C^n)$, and if additionally $u|_{\partial(SM)} = 0$ or $v|_{\partial(SM)} = 0$ then we also have $(Xu,v) = -(u,Xv)$ and $(X_{\perp} u, v) = -(u, X_{\perp} v)$.
The space $L^{2}(SM,\C^n)$ decomposes orthogonally
as a direct sum
$L^{2}(SM,\C^n)=\bigoplus_{k\in\mathbb Z}H_{k}$
where $H_k$ is the eigenspace of $-iV$ corresponding to the eigenvalue $k$.
A function $u\in L^{2}(SM,\C^n)$ has a Fourier series expansion
\begin{equation}u=\sum_{k=-\infty}^{\infty}u_{k},\end{equation}
where $u_{k}\in H_k$. Let $\Omega_{k}=C^{\infty}(SM,\C^n)\cap H_{k}$.
\end{comment}

\begin{definition} Let $u\in L^2(SM,\C^n)$.
\begin{enumerate}[label=(\roman*)]
\item  $u$ is called {\it fibrewise holomorphic}, iff $u_k=0$ for $k<0$. Similarly, $u$ is called {\it fibrewise anti-holomorphic}, iff $u_k=0$ for $k>0$.
\item $u$ is called {\it even} iff $u_k=0$ for $k\in 2\Z+1$, or equivalently iff $u(x,-v)=u(x,v)$ for all $(x,v)\in SM$. Similarly, $u$ is called {\it odd} iff $u_k=0$ for $k\in 2\Z$, or equivalently iff $u(x,-v)=-u(x,v)$ for all $(x,v)\in SM$.
\end{enumerate}
 %A function $u:SM\to\C^n$ is said to be (fibre-wise) holomorphic 
 %if $u_{k}=0$ for all $k<0$. Similarly, $u$ is said to be (fibre-wise) antiholomorphic if $u_{k}=0$ for all $k>0$.
\end{definition}

We tacitly use these definitions also on $\partial SM$, noting that functions $u\in L^2(\partial SM,\C^n)$ have an analogous decomposition $u=\sum_{k\in \Z}u_k$ into Fourier modes.

As in \cite{GK} we introduce the first order operators
\begin{equation}\label{defeta}
\eta_{+},\eta_{-}:C^{\infty}(SM,\C^n)\to
C^{\infty}(SM,\C^n)
\end{equation}
\begin{equation}\eta_{+}:=(X+iX_{\perp})/2,\;\;\;\;\;\;\eta_{-}:=(X-iX_{\perp})/2.\end{equation}
Clearly $X=\eta_{+}+\eta_{-}$. 
We have
\begin{equation}\eta_{+}:\Omega_{m}\to \Omega_{m+1},\;\;\;\;\eta_{-}:\Omega_{m}\to \Omega_{m-1},\;\;\;\;(\eta_{+})^{*}=-\eta_{-},\quad [\eta_\pm,V]=\mp i\eta_\pm
.\end{equation}
In particular, $X$ has the following important mapping property
\begin{equation}X:\oplus_{k\geq 0}\Omega_{k}\to \oplus_{k\geq -1}\Omega_{k}.\end{equation}
We will often use all of the above for smooth functions taking values in complex $n\times n$ matrices, which we indistinctly denote by $\C^{n\times n}$ or $\gl(n,\C)$, if we wish to think of them as Lie algebra of $GL(n,\C)$. (We also use the notation $\Omega_k$ in the matrix valued case.)

%The  and we will not make any distinction in the notation as it will become clear from the context. 

\subsection{Factorisation theorems}\label{factorisation}

For the range characterisations below it will be important to factor $GL(n,\C)$-valued maps on $SM$ in terms of the group $\G$ from \eqref{defG}. %, as described in Theorem \ref{thmfac} below. 
This requires a bundle-version of two well-known 
factorisation theorems for loop groups that we now recall, following the notation and presentation in \cite[\textsection 8]{PS86}. %, including the classical Birkhoff factorisation theorem in the Hermitian case.

Let us denote by $LGL_{n}(\C)$ the set of all smooth maps $\gamma:S^{1}\to GL(n,\C)$. The set has a natural structure of an infinite dimensional Lie group as explained in \cite[Section 3.2]{PS86}. This group contains several subgroups which are relevant for us. We shall denote by $L^{+}GL_{n}(\C)$ %and $L^-GL_n(\C)$ 
the subgroup consisting of those
loops $\gamma$ which are boundary values of holomorphic maps %and anti-holomorphic maps
\[\gamma:\{z\in\C:\,\,|z|<1\}\to GL(n,\C).\]
%respectively.
We let $\Omega U_{n}$ denote the set of smooth loops $\gamma:S^{1}\to U(n)$ such that $\gamma(1)=\text{Id}$.
The first result we shall use is Theorems 8.1.1 in \cite{PS86}: %and 8.1.2 in \cite{PS86}, the latter being the celebrated Birkhoff factorisation theorem:

 %result that is relevant to us is Theorem 8.1.1 in \cite{PS86}:

\begin{theorem} Any loop $\gamma\in LGL_{n}(\C)$ can be factored uniquely as
$ \gamma=\gamma_{u}\cdot \gamma_{+},$
with $\gamma_{u}\in\Omega U_{n}$ and $\gamma_{+}\in L^{+}GL_{n}(\C)$. In fact, the product map  
\begin{equation}\Omega U_{n}\times L^{+}GL_{n}(\C)\to LGL_{n}(\C)\end{equation}
%is a diffeomorphism.
is a diffeomorphism.
\label{thm:PS}
\end{theorem}

The second result we shall need is the celebrated Birkhoff factorisation theorem (cf.\,\cite[Theorem 8.1.1]{PS86}), stating that loops $\gamma\in LGL_n(\C)$ can be factored as $\gamma=\gamma_- \cdot \Delta \cdot \gamma_+$, where $\gamma_-^*,\gamma_+\in L^+GL_n(\C)$ and $\Delta$ is a group homomorphism from $S^1$ into the diagonal matrices in $GL(n,\C)$. In fact, we require only a version for loops with values in the space of positive definite Hermitian matrices, denoted
\begin{equation}
\he_n^+=\{H\in \C^{n \times n}: \xi^* H \xi >0 \text{ for all } \xi \in \C^n\backslash 0\}.
\end{equation}
In this case, $\Delta$ always equals $\Id$ and the statement is equivalent to the preceding theorem. We postpone a precise formulation to Theorem \ref{thmfac} below.

\begin{comment}
\begin{theorem} Any loop $\gamma\in LGL_{n}(\C)$ can be factorized
$\gamma=\gamma_{-}\cdot \lambda\cdot \gamma_{+},$
where $\gamma_{-}\in L^{-}GL_{n}(\C)$, $\gamma_{+}\in L^{+}GL_{n}(\C)$, and $\lambda$ is a loop which is a group homomorphism from $S^1$ into diagonal matrices in $GL_{n}(\C)$. Loops for which $\lambda=\Id$ form a dense open subset of the identity component of $LGL_{n}(\C)$ and the multiplication map
\begin{equation} L_{\Id}^{-}\times L^{+}GL_n(\C)\to LGL_{n}(\C),\end{equation}
where $L_{\Id}^-=\{\gamma_{-}\in L^{-}GL_n(\C):\gamma_{-}(0)=\Id\}$, is a diffeomorphism onto this subset.
\label{thm:PSB}
\end{theorem}
\end{comment}

Consider now a compact non-trapping surface $(M,g)$ with strictly convex boundary. It is well known that
such surfaces are diffeomorphic to a disc (cf.~\cite{PSU20}) and thus there exists a section $\1:M\rightarrow SM$ which trivialises the bundle $SM$ to $M\times S^1$. One can then perform loop group factorisations fibrewise to obtain:

\begin{theorem}\label{thmfac} Let $(M,g)$ be a non-trapping surface with strictly convex boundary.
\begin{enumerate}[label=(\roman*)]
\item \label{fac2} Any $R\in C^\infty(SM,GL(n,\C))$ can be factored as $R=UF$ (or $R=FU$) where $F\in \G$ and $U\in C^\infty(SM,U(n))$. If $R$ is even, then also $U$ and $F$ are even. Moreover, $F$ is unique up to left (or right) multiplication by a function in $C^\infty(M,U(n))$.
\item \label{fac1} Any $H\in C^\infty(SM,\he^+_n)$ can be factored as $H=F^*F$ with $F\in \G$. If $H$ is even, then also $F$ is  even. Moreover, $F$ is unique up to left multiplication by a function in $C^\infty(M,U(n))$.
\end{enumerate}
%Statements \ref{fac1} and \ref{fac2} remain true at the boundary, that is, after substituting $\G$ by $\mathbb{H}=\{f = F\vert_{\partial SM}: F\in \G\}$ and sub
\end{theorem}

\begin{proof}
Part \ref{fac2}, modulo the statement on even functions, follows from Theorem \ref{thm:PS}, applied to the loop $R(x,\cdot)$ for each $x\in M$. 
Normalising such that $U(x,\1(x))=\Id$, the resulting factors $U$ and $F$ are smooth on $SM$ -- we refer to Theorem 4.2 in \cite{PaSa20} and its proof for more details. Now suppose that $R=UF$ is even. Denoting $a:SM\rightarrow SM$ the antipodal map, defined by $a(x,v)=(x,-v)$, we then have $UF = R = R\circ a= (U\circ a)(F\circ a)$. As the factorisation is unique up to gauge, there exists a function $h\in C^\infty(M,U(n))$ with $U=(U\circ a)h$ and $F=h^*(F\circ a)$. Consequently $U$ and $F$ must be even.

For \ref{fac1} note that any $H\in C^\infty(SM,\he^+_n)$ admits a square root, i.e.~there exists an $R\in C^\infty(SM,GL(n,\C))$ with $H=R^*R$. Using \ref{fac2}, we may decompose $R=UF$, with $U$ unitary and $F\in \G$ and thus $H=F^*U^*UF=F^*F$, as desired. The uniqueness claim follows from the observation that if we had two factorizations $F^*F=\tilde{F}^*\tilde{F}$, then $(\tilde{F}^*)^{-1}F^*=\tilde{F}F^{-1}$. It follows that $\tilde{F}F^{-1}$ is both fibrewise holomorphic and antiholomorphic and thus equal to $h\in C^{\infty}(M,U(n))$.

 %Conversely, if $R\in C^\infty(SM,GL(n,\C))$, then $H:=R^*R$ takes values in $\he^+_n$ and, using  part \ref{fac1}, factors as $H=F^*F$ with $F\in \G$. It is easy to see that $U:=RF^{-1}$ is unitary, such that \ref{fac2} follows.
\end{proof}

\begin{remark}\label{boundaryfac} 
There is a boundary version of the theorem in terms of the group $\mathbb{H}=\{f=F\vert_{\partial SM}: F\in \G\} = \{f\in C_\Id^\infty(\partial SM,GL(n,\C)): f\text{ is fibrewise holomorphic}\}$. Here and below the subscript $\Id$ 
refers to maps that are homotopic to $\Id$. %Any smooth map $r:\partial SM\rightarrow GL(n,\C)$ induces a homomorphism $r_*:\Z\times \Z\rightarrow \Z$ between fundamental groups and we put $w(r)=r_*(1,0)\in \Z$. Writing $C^\infty_w(\partial SM,\cdot)$ for maps with $w(r)=0$ we have
 Indeed, all of the following maps are surjective and injective up to a gauge in $C^\infty_\Id(\partial M,U(n))$:
\begin{align}
C^\infty_\extend  (\partial SM,U(n))\times \HH \rightarrow C^\infty_\extend(\partial SM,GL(n,\C)),&\quad (u,f)\mapsto uf\label{bfac2} \\
\HH\times C^\infty_\extend(\partial SM,U(n)) \rightarrow C^\infty_\extend(\partial SM,GL(n,\C)),&\quad (u,f)\mapsto fu \label{bfac3}\\
\HH \rightarrow  C^\infty(\partial SM,\he_+^n),&\quad f\mapsto f^*f \label{bfac1}
\end{align} 
Here $C_\extend^\infty$ stands for smooth maps $r$ which have a $GL(n,\C)$-valued extension to all of $SM$, or equivalently for which the induced homomorphism $r_*:\Z\times \Z\rightarrow \Z$ between fundamental groups satisfies $r_*(1,0)=0$. 
(Since $\partial SM=\partial M\times S^{1}$ is a torus and $GL(n,\C)$ has fundamental group $\mathbb{Z}$,
the map $r$ induces a homomorphism $r_{*}:\mathbb{Z}\times\mathbb{Z}\to \mathbb{Z}$.)
To show \eqref{bfac2}-\eqref{bfac1}, we extend $r\in C_\extend^\infty(\partial SM,GL(n,\C))$ to a function $R\in C^\infty(SM,GL(n,\C))$ and apply Theorem \ref{thmfac} to $R$ in order to find appropriate factors for $r$. We emphasise however, that the factors $u$ and $f$ in the previous display can be found pointwise for every $x\in \partial M$ by solving a Birkhoff factorisation problem in the fibre $S_xM$.
\end{remark}

\section{Matrix holomorphic integrating factors} \label{s_hif}

In this section we prove Theorem \ref{mainhif} on the existence of matrix holomorphic integrating factors on simple surfaces. Recall from the discussion below the theorem that this is equivalent to proving that the group $\G$ from \eqref{defG} acts transitively on $\mho$ from \eqref{defmho} via the rule \eqref{rule}. To show transitivity, we use the Nash-Moser inverse function theorem in the form of Theorem 2.4.1 in \cite[\textsection III]{Ham82}, which requires that:
\begin{enumerate}[label=\alph*)]
\item \label{tamea} $\G$ is a tame Fr{\'e}chet Lie group, $\mho$ is a connected, tame Fr{\'e}chet manifold and the action of $\G$ on $\mho$ is smooth tame;
\item \label{tameb} for all $\A\in \mho$, the derivative of $F\mapsto A\act F$ at $\Id$ has a tame right inverse.
\end{enumerate}
Here tameness is understood with respect to the grading  $\left(\Vert \cdot \Vert_{H^s}:s=0,1,\dots \right)$ by Sobolev norms and condition \ref{tamea} is satisfied in view of standard estimates; for more details we refer to Appendix \textsection \ref{tamesetting}.  The key condition is \ref{tameb} and we claim that the derivative in question is given by
\begin{equation}\label{action2}
\T_\A: T_\Id \G \rightarrow \mho,\quad \T_\A(H) = XH + [\A,H],
\end{equation}
with $T_\Id \G = \{G\in C^\infty(SM,\C^{n\times  n}): G \text{ fibrewise holomorphic}\}$ and $[\cdot,\cdot]$ denoting the commutator. To see this, fix $\A
\in \mho$ and consider 
$F_t=\Id + t H\in \G$ for $H\in T_\Id \G$ and small $t\in \R$. Let $s\ge 0$, then for $\vert t \vert$ sufficiently small,  the Neumann series $\sum_{k\ge 0} (-tH)^k$ converges in the Sobolev space $H^s(SM)$ and one computes that
\begin{equation}
\A\act F_t= t XH + \A + t \A H - t H \A + o_{\Vert \cdot \Vert_{H^s}}(1),\quad \text{ as  }t\rightarrow 0
%\A + t \left( XH + [\A,H] \right) +t  \sum_{k \ge 1} (-tH)^k XH + \sum_{k \ge 2} (-tH)^k(\Id + tH),
\end{equation}
which yields the formula in \eqref{action2}.

The proof of Theorem \ref{mainhif} is complete, if we show that in the simple case the map $\T_\A$ in \eqref{action2} has a tame right inverse for all $\A \in \mho$. This is implied by the following proposition, which is formulated in terms of $\C^n$-valued functions -- the required right inverse for $\T_\A$ is obtained by going `one level higher', i.e.~viewing $\hat \A = [\A,\cdot]$ as attenuation with values in $\End(\C^{n\times  n})\cong \C^{n^2\times n^2}$, acting on $\C^{n^2}$-valued functions.

\begin{proposition} \label{xaonto}
 Let $(M,g)$ be a simple surface and $\A\in\mho$. Then the map
$
 (X+\A):\oplus_{k\ge 0}\Omega_k \rightarrow \oplus_{k\ge -1}\Omega_k$
 is onto and admits a right inverse $L_\A:\oplus_{k\ge -1}\Omega_k \rightarrow \oplus_{k\ge 0}\Omega_k$ obeying the tame estimate
 \begin{equation}
 \Vert L_\A f \Vert_{H^s} \lesssim \Vert f \Vert_{H^{s+1}}\quad f\in \oplus_{k\ge -1}\Omega_{k}, s\ge 0,
 \end{equation}
 where $\lesssim$ means up to a constant that depends only on $(M,g)$, $\A$ and $s$.
\end{proposition}

The proposition relies on a number of lemmas that we will discuss first.
The first lemma, modulo the tame estimates, appears as Proposition 4.5 in \cite{AiAs15} and relies on the fact that the attenuated X-ray transform $I_{A,\Phi}$ is injective on $\Omega_0\oplus \Omega_1$. Recall that  $I_\A=I_{A,\Phi}:C^\infty(SM,\C^n)\rightarrow C^\infty(\partial_+SM,\C^n)$ (for a $\u(n)$-pair $\A=(A,\Phi)$) is defined by $I_\A f = u^f\vert_{\partial_+ SM}$, where $u^f:SM\rightarrow \C^n$ is the unique continuous solution (differentiable along the geodesic flow) of $(X+\A)u^f=-f$ on $SM$ and $u^f=0$ on $\partial_-SM$. 
The tame estimates can be traced back to mapping properties of the associated normal operator.

\begin{lemma} \label{lemA} Let $(M,g)$ be simple and $\A=(A,\Phi)$ a skew-Hermitian pair. Then for any $f_m+f_{m+1}\in \Omega_m\oplus \Omega_{m+1}$ there is a solution $u\in C^\infty(SM,\C^n)$ to 
\begin{equation} \label{lemA1}
(X+\A)u = 0 \quad \text{ and } \quad u_m=f_m,~u_{m+1}=f_{m+1}.
\end{equation}
The solution operator $S_{m,\A}:\Omega_m\oplus \Omega_{m+1}\rightarrow C^\infty(SM,\C^n)$, sending $f_m+f_{m+1}$ to $u=S_{m,\A}(f_m+f_{m+1})$, may be chosen to satisfy the tame estimates
\begin{equation} \label{lemA2}
\Vert S_{m,\A}(f_{m}+f_{m+1}) \Vert_{H^s} \lesssim \Vert f_{m} + f_{m+1} \Vert_{H^{s+1}},\quad f_m+f_{m+1} \in \Omega_m\oplus \Omega_{m+1}, s\ge 0,
\end{equation}
where $\lesssim$ means up to a constant that depends only on $(M,g)$,  $\A,m$ and $s$. 
\end{lemma}

\begin{proof}
First consider the case $m=0$. Write
$
I^{0,1}_{\A}:\Omega_0\oplus\Omega_1 \rightarrow C^\infty(\partial_+SM,\C^n)
$
for the attenuated X-ray transform, restricted to $\Omega_0\oplus \Omega_1$.  This transform is injective, as the natural gauge from \cite[Theorem~1.3]{PSU12} is fixed on $\Omega_0\oplus \Omega_1$.  Indeed, if $I^{0,1}_{\A}(f)=0$ for $f\in\Omega_0\oplus\Omega_{1}$, then there is a smooth $p:M\to \C^n$ with $p|_{\partial M}=0$ such that $f=\Phi\,p+(X+A)p$. Since $f_{-1}=0$ we see that $(\eta_{-}+A_{-1})p=0$ and this gives $p=0$ via Lemma \ref{ogdisk} since any holomorphic function that vanishes on the boundary must be identically zero.

By means of Santal{\'o}'s formula the $L^2$-adjoint $(I_\A^{0,1})^*$ with respect to the measure $\langle \nu(x),v\rangle \d \Sigma^2(x,v)$ on $\partial_+SM$ can be characterised by the equivalence
\begin{equation} \label{lemA3}
f_0 + f_1 = (I^{0,1}_\A)^* h \quad \Longleftrightarrow \quad f_0=(h^\sharp)_0,~f_1 = (h^\sharp)_1.
\end{equation}
This is valid for all $h\in \mathcal{S}_\A^\infty(\partial_+SM,\C^n)$, the set of functions $h\in C^\infty(\partial_+SM,\C^n)$ for which the solution $h^\sharp$ to $(X+\A)h^\sharp = 0 $ with $h^\sharp\vert_{\partial_+SM} = h$ is smooth on all of $SM$ (cf.\,\cite[\textsection 5]{PSU15}, adding a matrix field is unproblematic). The first statement of the lemma is then the assertion that $(I_\A^{0,1})^*:\mathcal{S}_\A^\infty(\partial_+SM,\C^n)\rightarrow \Omega_0\oplus\Omega_1$ is onto. Indeed, if $f_0+f_1 = (I_\A^{0,1})^* h$, then $u=h^\sharp$ solves \eqref{lemA1}.

This assertion, together with the tame estimates,  is proved with help of the associated normal operator, which was shown to be elliptic in \cite{AiAs15}. More precisely, if we make the identification $\Omega_0\oplus \Omega_1 \cong C^\infty(M,\C^{2n})$ (which is possible after trivialising $SM$), then 
 %Denote its $L^2$-adjoint (with respect to the symplectic measure $\langle \nu(x),v\rangle \d \Sigma^2$) by $(I^{0,1}_\A)^*$ and identify $\Omega_0\oplus \Omega_1$ with $C^\infty(M,\C^n\oplus \C^n)\cong C^\infty(M,\C^{2n})$. Then the normal operator
\begin{equation}
(I_\A^{0,1})^*I_\A^{0,1}:C_c^\infty(M^\interior,\C^{2n})\rightarrow C^\infty(M^\interior,\C^{2n})
\end{equation}
is an elliptic pseudodifferential operator of order $-1$ \cite[\textsection 4.2]{AiAs15}.  %To make available standard mapping properties of pseudodifferential operators, we move our discussion to a closed ambient space.\\
To proceed, embed $M$ into a closed surface $(N,g)$ and cover $N$ by open subsets $U_1,\dots,U_m$ such that $M\subset U_1$,  and $M_j = \bar U_j$ is a simple surface for all $j$.  Let $\psi_1,\dots, \psi_m\in C^\infty(N,\R)$ be such that $\psi_1\equiv 1$ on $M$, $\supp \psi_j\subset M_j$ and $\sum_{j=1}^m\psi_j^2=1$ on $N$.
Further, extend $\A=(A,\Phi)$ to a pair $\A_1=(A_1,\Phi_1)$ with compact support in $SM^{\interior}_1$ and set  $\A_2=\dots=\A_m = 0$. Following the template from \cite[\textsection 8.2]{PSU20}, we see that
\begin{equation}
P= \sum_{j=1}^m \psi_j (I_{\A_j}^{0,1})^*I_{\A_j}^{0,1} \psi_j: C^\infty(N,\C^{2n})\rightarrow C^\infty(N,\C^{2n})
\end{equation}
is an elliptic pseudodifferential operator of order $-1$ on $N$,  which is self-adjoint and thus has Fredholm index zero. As each of the operators $I_{\A_j}^{0,1}$ is injective, also $P$ is injective and thus it defines a homeomorphism $P:C^\infty(N,\C^{2n})\rightarrow C^\infty(N,\C^{2n})$. For $f=f_0+f_1\in \Omega_0\oplus\Omega_1\cong C^\infty(M,\C^{2n})$ we can now define 
\begin{equation} \label{lemA4}
S_{0,\A}(f_0+f_1) = (I_{\A_1}^{0,1}\psi_1 P^{-1}(Ef))^{\sharp_1}\vert_{SM},
\end{equation}
where $E:C^\infty(M,\C^{2n})\rightarrow C^\infty(N,\C^{2n})$ is an extension operator and the map $(\cdot)^{\sharp_1}:\mathcal{S}_{\A_1}^\infty(\partial_+SM_1,\C^n)\rightarrow C^\infty(SM_1,\C^{n})$ is defined similar as above, this time with respect to $\A_1$. First note that $u=S_{0,A}(f_0+f_1)$ indeed solves \eqref{lemA1}: Write $h_1 = I_{\A_1}^{0,1}\psi_1 P^{-1}(Ef)$ and $h = u \vert_{\partial_+SM}$, then $h_1^{\sharp_1}\vert_{SM}=h^\sharp$ and thus
\begin{equation}
(I_\A^{0,1})^*h= (I_{\A_1}^{0,1})^*h_1 =  \psi_1  (I_{\A_1}^{0,1})^* I_{\A_1}^{0,1} \psi_1 \left( P^{-1}(Ef)\right)= f \quad \text{ on } M,
\end{equation}
where we used the characterisation \eqref{lemA3} and the fact that $\psi_1\equiv 1$ on $M$, while all other $\psi_j's$ vanish. Consequently, $S_{0,\A}$ is the desired solution operator and it remains to check the tame estimates.

We check tameness of the operators in \eqref{lemA4} separately. First note that the extension $E$, multiplication by $\psi_1$ and application of $I^{0,1}_\A$ satisfy the appropriate tame estimates in a Sobolev scale {\it without loss of derivatives.} For $E$ this is the content of Seeley's classical article \cite{See64} and for $I_\A^{0,1}$ this is a standard forward estimate \cite[Theorem~4.2.1]{Sha94}.  Further we have
\begin{equation}
\Vert P^{-1} g \Vert_{H^s(N)} \lesssim \Vert g \Vert_{H^{s+1}(N)},\quad g\in C^\infty(N,\C^{2n}),s\ge 0,
\end{equation}
which follows from $P:H^{s}(N,\C^{2n})\rightarrow  H^{s+1}(N,\C^{2n})$ being injective with closed range. Next,  note that $\supp I^{0,1}_{\A_1}(\psi_1 g) \subset K$ for all $g\in C^\infty(M_1,\C^{2n})$ and a fixed compact set $K\subset \partial_+SM_1$ with $K\cap \partial_0SM_1=\emptyset$.  In order to obtain the tame estimates for $S_{0,\A}$, it thus remains to show
\begin{equation} \label{lemA5}
\Vert h^{\sharp_1} \Vert_{H^s(SM_1)} \lesssim  \Vert h \Vert_{H^s(\partial_+SM_1)} \quad  h\in C_K^\infty(\partial_{+}SM_1),  s\ge 0
\end{equation}
where the subscript indicates that $\supp h \subset K$.  Let $R:SM_1\rightarrow U(n)$ be a smooth solution to $(X+\A_1)R=0$ with $R = \Id$ on $\partial_+SM_1$ (this exists, because $\A_1$ has compact support). Further, write $\psi:SM_1\to \partial_+SM_1$ for the foot-point projection,  sending $(x,v)$ to the unique point on $\partial_+SM_1$ on the same geodesic.  Then $h^{\sharp_1} = R \cdot \psi^*h$ and we conclude \eqref{lemA5} from the following mapping properties: Multiplication by $R$ is bounded $H^s(SM_1,\C^n)\rightarrow H^s(SM_1,\C^n)$ and pull-back by $\psi$ is bounded
$H_K^s(\partial_+SM_1,\C^n)\rightarrow H^s(SM_1,\C^n)$ (again subscript $K$ indicates a support restriction). This concludes the Lemma for $m=0$.

For general $m\in \Z$ the operator $S_{m,\A}$ is obtained by conjugation with $e^{im\theta}$, where the angle $\theta$ is chosen with respect to a trivialisation of $SM$.
Indeed, given $f_m+f_{m+1}\in \Omega_m\oplus \Omega_{m+1}$, let $\tilde f_0 = e^{-im\theta} f_m$,  $\tilde f_1 =e^{-im\theta} \tilde f_{m+1}$ and $\tilde \A = \A + e^{-im\theta} X(e^{im\theta})$. Put $\tilde u=S_{0,\tilde \A}(\tilde f_0 + \tilde f_1)$, then $S_{m,\A}(f_m+f_{m+1}):= u=e^{im\theta} \tilde u$ satisfies 
\begin{equation}
\begin{split}
X u &= (Xe^{im\theta}) \tilde u - e^{im\theta} (\tilde \A \tilde u) = - \A u\\
 u_m &=e^{im\theta} \tilde u_0 = f_m, \quad \text{ and } \quad u_{m+1} = e^{i(m+1)\theta} \tilde u_1 = f_{m+1},
\end{split}
\end{equation}
as desired. The tame estimates follow immediately from the case $m=0$ and the proof is complete.
\end{proof}

%Recall that if $A$ is a connection on $M$, then $(X+A)=\mu_++\mu_-$ where $\mu_\pm = \eta_\pm + A_{\pm 1}$ and $\eta_\pm$ are the Guillemin-Kazhdan operators.  The next lemma provides us with time right-inverse for $\mu_\pm$, and is essentially proved by complex analysis methods.

The next lemma is essentially a result in complex analysis:

\begin{lemma} \label{lemB}
Let $(M,g)$ be non-trapping with strictly convex boundary and $A\in \Omega^1(M,\gl(n,\C))$ a matrix valued $1$-form. Then $\mu_\pm=\eta_\pm + A_{\pm 1}: \Omega_{m}\rightarrow \Omega_{m\pm 1}$ is onto and admits a right inverse $T_{A,\pm,m}:\Omega_{m\pm 1}\rightarrow \Omega_m$ obeying the tame estimates
\begin{equation}
\Vert T_{A,\pm,m} q \Vert_{H^{s+1}} \lesssim \Vert q \Vert_{H^s}, \quad q\in \Omega_m, s\ge 0,
\end{equation}
where $\lesssim$ means up to a constant that depends only on $(M,g)$, $A,m$ and $s$.
\end{lemma}

\begin{proof}
We only consider $\mu_-$, the result for $\mu_+$ follows by a similar method. Fix global isothermal coordinates,  such that elements in $\Omega_m$ are given by $he^{im\theta}$ for a function $h\in C^\infty(\DD,\C^n)$ (in particular $A_{\pm 1} = e^{\pm i \theta} a_{\pm }$) and the metric is $g=e^{2\lambda} g_{\mathrm{Eucl.}}$ for some conformal factor $\lambda(z)$.  Here $\DD\subset \C$ is the closed unit disk. Define $\alpha\in C^\infty(\DD,\C^{n\times n})$ by $\alpha(z) = e^{\lambda(z)} a_{-}(z)$. Then a computation similar to \cite[Lemma~6.1.8]{PSU20} yields
\begin{equation}\label{lemB2}
\mu_-\left( h e^{im \theta}\right) = e^{-(m+1)\lambda}  \bar \partial_\alpha (h e^{m\lambda}) \cdot e^{i(m-1)\theta}, \qquad h\in C^\infty(\DD)
\end{equation}
where $\bar \partial_\alpha:C^\infty(\DD,\C^n)\rightarrow C^\infty(\DD,\C^n)$ is defined by $\bar \partial_\alpha u  = \partial_{\bar z} u+ \alpha u$. As multiplication operators are tame (without loss of derivatives), it suffices to construct a tame right inverse for $\bar \partial_\alpha$.
By Lemma \ref{ogdisk} there is a solution $R\in C^\infty(\DD,GL(n,\C))$ to $\bar \partial_\alpha R = 0$. Then
$
R^{-1}\bar \partial_{\alpha} (R u) = \bar \partial_0 u$ for all $u \in C^\infty(\DD)
$
and we have reduced the problem to finding a tame right inverse for $\bar \partial_0\equiv \partial_{\bar z}$. 

It is a basic result in complex analysis that the equation $\partial_{\bar z} u = h$ over $\C$, given some $h\in C_c^\infty(\C,\C^n)$, is solved by
\begin{equation}
u(z)=Ph(z ) = - \frac{1}{2\pi i} \int_\C \frac{h(\zeta)}{\zeta - z} \d \bar \zeta \wedge \d \zeta,\quad z\in \C.
\end{equation}
A right inverse for $\bar \partial_0$ on $\DD$ is thus given by
\begin{equation}
T:C^\infty(\DD,\C^n)\rightarrow C^\infty(\DD,\C^n),\quad Tf = P(Ef)\vert_{\DD},
\end{equation}
where $E:C^\infty(\DD,\C^n)\rightarrow C^\infty(\C,\C^n)$ is a Seeley extension operator, say chosen such that $\supp Ef\subset 2\DD$ for all $f\in C^\infty(\DD,\C^n)$.  Let $\chi\in C_c^\infty(\C,\R)$ with $\chi \equiv 1$ on $\DD$, then for all $s\ge 0$ we have
\begin{equation}
\Vert Tf \Vert_{H^{s+1}(\DD)} \le \Vert \chi P (Ef) \Vert_{H^{s+1}(\C)} \lesssim \Vert E f \Vert_{H^{s}(2\DD)} \lesssim \Vert f \Vert_{H^s(\DD)}, 
\end{equation}
where we have used that $P$ is a pseudodifferential operator of order $-1$ and thus it has the mapping property $H^{s}_{c}(\C)\rightarrow H^{s+1}_\mathrm{loc}(\C)$.

The right-inverse $T_{A,-,m}$ of $\mu_-:\Omega_m\rightarrow \Omega_{m-1}$ is obtained from $T$ by conjugating with $R$ and multiplying with scalar factors as indicated in \eqref{lemB2}. In particular,  the tame estimate \eqref{lemB2} follows from to the previous display and the proof is complete.
\end{proof}

The final ingredient is a non-holomorphic version of Proposition \ref{xaonto} and follows from well-known solvability results and estimates concerning the attenuated transport equation over {\it smooth} functions.

\begin{lemma} \label{lemC}
Let $(M,g)$ be non-trapping with strictly convex boundary and suppose $\A\in C^\infty(SM,\u(n))$. Then $X+\A:C^\infty(SM,\C^{n})\rightarrow C^\infty(SM,\C^{n})$ has a right inverse $U_\A$ that obeys the tame estimates
\begin{equation}
\Vert U_\A f \Vert_{H^s(SM)} \lesssim \Vert f \Vert_{H^s(SM)}\quad s\ge 0, f\in C^\infty(SM,\C^{n}),
\end{equation}
where $\lesssim$ means that the inequality holds up to a multiplicative constant that depends only on $(M,g)$, $\A$ and $s$. 
\end{lemma}

\begin{proof} First assume that both $\A$ and $f$ have compact support in $SM^{\interior}$.  Then the unique  continuous solution $g:SM\rightarrow \C^{n}$ to 
\begin{equation}
(X+ \A)g = f\quad \text{ on } SM\quad \text{ and } \quad g = 0 \text{ on } \partial_-SM
\end{equation}
vanishes near the glancing region $\partial_0SM$ and consequently is smooth on $SM$.
%Indeed,  we have 
%$
%g(x,v) = - R_\A \int_0^{\tau(x,v)} (R_\A^{-1} f)(\varphi_t(x,v)) \d t,
%$
%where $R_\A:SM\rightarrow GL(n,\C)$ is smooth and solves $X_\A R_\A = 0$, such that smoothness of $g$ follows from smoothness of $\tau$ away from the glancing region $\partial_0SM$. 
Following \cite[Lemma 5.12]{MNP19}, we have
\begin{equation} \label{lemC2}
\Vert g \Vert_{L^2} \le \tau_\infty \cdot \Vert f \Vert_{L^2},
\end{equation}
where $\tau_\infty = \sup_{SM} \tau$.
Let $P$ be a differential operator on $SM$ of order $m\ge 0$ and with constant coefficients with respect to the commuting frame $\{X,P_T,P_V\}$ from \cite[Lemma 5.1]{MNP19}. Then $\tilde g= Pg$ is the unique solution of
\begin{equation}
(X+\A) \tilde g = Pf + [\A,P]g\quad \text{ on } SM\quad \text{ and } \quad \tilde g = 0 \text{ on } \partial_-SM,
\end{equation}
where $[\cdot,\cdot]$ denotes the commutator. As $\tilde f = Pf + [\A,P]g$ has compact support, we may apply \eqref{lemC2} to obtain
\begin{equation}
\begin{split}
\Vert P g \Vert_{L^2}  \le \tau_\infty\left(\Vert Pf \Vert_{L^2} + \Vert [\A,P]g \Vert_{L^2}\right) 
\lesssim \Vert f \Vert_{H^m} + \Vert g \Vert_{H^{m-1}},
\end{split}
\end{equation}
where we used that $[\A,P]$ is a differential operator of order $m-1$. The $H^m$-norm of $g$ can be bounded in terms of $\Vert P g \Vert_{L^2}$, if $P$ is taken uniformly elliptic.  By induction (and an interpolation argument to pass to non-integral regularities) it then follows that
$
\Vert g \Vert_{H^s} \lesssim \Vert f \Vert_{H^s}
$ for all $s\ge 0$, with implicit constant only depending on $\tau_\infty$, $s$ and $\A$.

The right inverse $U_\A$ for general $\A$ and $f$ can be obtained by a standard extension trick: Embed $M$ in the interior of a slightly large manifold $(M_1,g)$ which is also non-trapping and has strictly convex boundary and extend $\A$ to a smooth attenuation $\A_1:SM_1\rightarrow \u(n)$ with compact support in $SM_1^{\interior}$.
 Let $E:C^\infty(SM,\C^{n})\rightarrow C^\infty(SM,\C^{n})$ be a Seeley extension operator and define $U_\A f = g_1 \vert_{SM}$, where $g_1:SM_1\rightarrow \C^{n}$ is the unique solution to $(X + \A_1) g_1 = E f $ on $SM_1$ with $g_1 = 0 $ on $\partial_-SM_1$. Then by the previous considerations we have
 \begin{equation}
 \Vert U_\A f \Vert_{H^s(SM)} \le \Vert g_1 \Vert_{H^s(SM_1)} \lesssim \Vert E f \Vert_{H^s(SM_1)} \lesssim \Vert f \Vert_{H^s(SM)},
 \end{equation}
 which proves the Lemma.
\end{proof}

\begin{proof}[Proof of Proposition \ref{xaonto}]
We first give the proof for a skew-Hermitian pair $\A=(A,\Phi)$. Given $f\in \oplus_{k\ge -1}\Omega_k$,  we use Lemma \ref{lemC} to obtain a smooth solution $u$ to $(X+\A)u=f$ and consider $\tilde u = u_0 + u_1 + \dots\in \oplus_{k\ge 0}\Omega_k$. Then\begin{equation}
(X+\A) \tilde u = f - \mu_+ u_{-1} - \left(\Phi u_{-1} + \mu_+ u_{-2} \right) =: f - q_{0} - q_{-1}.
\end{equation}
By Lemma \ref{lemB} we may solve the equations $\mu_- g_0 = q_{-1}$ and $\mu_+ g_{-1} = - q_0$ with  $g_{m}\in \Omega_{m}$ ($m=-1,0$) and by Lemma \ref{lemA} there exists a smooth solution $v$ to $(X+\A)v=0$ with $v_{-1}=g_{-1}$ and $v_{0}=g_0$. Let $\tilde v = v_0 + v_1 + \dots\in \oplus_{k\ge 0}\Omega_k$, then
\begin{equation}
(X+\A) \tilde v = \mu_-v_0 - \mu_+v_{-1} = q_{-1} + q_0.
\end{equation}
In particular $L_\A f := \tilde u + \tilde v \in \oplus_{k\ge 0}\Omega_k$ defines a preimage of $f\in \oplus_{k\ge -1}\Omega_k$ under $(X+\A)$, which implies surjectivity. For the tame estimates note that
\begin{equation}
\begin{split}
\tilde u &= P_{\ge 0}\circ U_\A f\\
\tilde v &= P_{\ge 0}  \circ  S_{\A,-1} \circ (T_{A,-,0}, - T_{A,+,{-1}})\circ Q\circ U_\A f,
\end{split}
\end{equation}
where $S,T,U$ are as in the lemmas above, $P_{\ge 0}:C^\infty(SM,\C^{n\times n})\rightarrow \oplus_{k\ge 0}\Omega_k$ is the $L^2$-orthogonal projection and $Q:C^\infty(SM,\C^{m\times m})\rightarrow \Omega_{-1} \oplus \Omega_0$ is defined by
\begin{equation}
Q u = (\Phi u_{-1} + \mu_+u_{-2}) \oplus \mu_+u_{-1}.
\end{equation}
Each of these linear operators was shown to satisfy a tame estimate $\Vert \bullet  \cdot \Vert_{H^s} \lesssim \Vert \cdot \Vert_{H^{s+d}}$ of degree $d\in \R$,  which can be read off the preceding lemmas and Lemma \ref{tameproj}. Combined, we see that $L_\A$ is tame of degree $1$, as desired.

Next assume that $\A\in \mho$ is a general attenuation.  By Lemma 5.2 in \cite{PaSa20} there exists $F\in \G$ such that $\B=\A\act F\in \mho$ defines a skew-Hermitian pair $\B=(B,\Psi)$. Our previous considerations thus yields a tame right inverse $L_\B$ to $X+\B$.  It is easy to check that $L_\A = FL_\B F^{-1}$ gives a right inverse for $X+\A$, which inherits the tameness from $L_\B$. This concludes the proof.\end{proof}

By the same methods we obtain the following variant of Theorem \ref{mainhif}:

\begin{proposition}\label{hifeven} Let $(M,g)$ be a simple surface. Then any odd attenuation $\A\in \mho$ admits  {\it even} holomorphic integrating factors.
\end{proposition}

\begin{proof} As for Theorem \ref{mainhif}, the proposition is equivalent to $\G_\even=\{F\in \G: F\text{ even}\}$ acting transitively on $\mho_\text{odd}=\{\A \in \mho: \A \text{ odd}\}$. Using Nash-Moser's theorem, it suffices to prove that for all $\A\in \mho_\text{odd}$ also the map 
$
(X+\A):\oplus_{k\ge 0} \Omega_{2k} \rightarrow \oplus_{k\ge -1} \Omega_{2k+1} $ has a tame right inverse. In terms of $L_\A$ from Proposition \ref{xaonto} and the projection $P^\even : \oplus_{k\ge 0} \Omega_k \rightarrow \oplus_{k\ge 0}\Omega_{2k}$ onto even parts we define $L^\even_\A:\oplus_{k\ge -1} \Omega_{2k+1} \rightarrow \oplus_{k\ge 0}\Omega_{2k}$ by $L^\even_\A f = P^\even L_\A f$. This is a tame map, as $L_\A$ and $P^\even$ are tame (see Lemma \ref{tameproj}) and provides the desired right inverse.
\end{proof}

\section{Twistor correspondence}\label{twistor}

Let $(M,g)$ be an oriented Riemannian surface with smooth, possibly empty boundary $\partial M$. We construct a {\it twistor space} $Z$ associated to $M$, which provides a natural habitat to complexify transport problems on $SM$.

\subsection{A complex surface} \label{cxsurface} The twistor space of $(M,g)$ is a degenerate complex surface with underlying smooth manifold $
Z=\{(x,v)\in TM:g(v,v)\le 1\}.
$
The complex structure on $Z$ is described in terms of a complex distribution
\begin{equation*}
\D \subset T_\C Z\equiv TZ\otimes\C,
\end{equation*}
to be thought of as the $(0,1)$-bundle; the structure degenerates on $SM\subset \partial Z$ in the sense that there $\D\cap \bar \D \neq 0$. 
The construction of $\D$ is carried out in the following lemma, precisely in equation \eqref{donz}; we then review some standard notions from complex geometry in the degenerate context.

It is most convenient to describe the geometry of $Z$ in terms of the fibration
\begin{equation}
p\colon SM\times \DD\rightarrow Z,\quad (x,v,\omega)\mapsto (x,v\omega),
\end{equation}
where $\DD = \{\omega \in \C\colon \vert \omega \vert \le 1\}$ is the complex unit disk and the product $v\omega\in T_xM$ is explained by the complex structure that $g$ and the orientation induce on $M$. Note that $p$ is a principal $S^1$-bundle for the diagonal action $(x,v,\omega)\act e^{it} = (x,ve^{it},e^{-it}\omega)$, which has infinitesimal generator
\begin{equation}
\V(x,v,\omega) = V(x,v) + i( \bar \omega \partial_{\bar \omega}- \omega \partial_\omega), \quad (x,v,\omega)\in SM\times \DD.
\end{equation}
In particular, $SM\times \DD/S^1\cong Z$ as smooth manifolds with corners, with boundary hypersurfaces
$p(SM\times \partial \DD)\equiv SM$ and  $p(\partial SM\times \DD)$. We equip $T_\C(SM\times \DD)$  with the natural Hermitian structure given in terms of the Sasaki metric on $SM$ and the Euclidean metric on $\DD$.%; further a {\it complex distribution} on $SM\times \DD$ is a subbundle of $T_\C(SM\times \DD)$.

\begin{lemma}[Complex structure on $Z$]\label{cxstructure} \,
\begin{enumerate}[label =(\roman*)]
\item \label{cxstructure1} The following commutator relations hold on $SM\times \DD$:
\begin{equation*}
[\omega^2 \eta_+ + \eta_-,\V] = i (\omega^2 \eta_+ + \eta_-)\quad \text{ and } \quad [\partial_{\bar \omega},\V] = i \partial_{\bar \omega}.%,\quad \text{ and } \quad [\omega^2 \eta_+ + \eta_-,\partial_{\bar \omega}]= 0
\end{equation*}
In particular, the complex distribution $\widetilde \D = \spn_\C \{\omega^2 \eta_+ + \eta_-,\partial_{\bar \omega}\}$ on $SM\times \DD$ is $S^1$-invariant and descends to a distribution on $Z$, denoted
\begin{equation}\label{donz}
\D = p_*(\widetilde \D).
\end{equation}
\item \label{cxstructure2} The Gram matrix of $\{\omega^2 \eta_+ + \eta_-,\bar \omega^2 \eta_- + \eta_+,\V, \partial_{\bar \omega},\partial_\omega\}$ at $(x,v,\omega)\in SM\times \DD$, denoted $G(x,v,\omega)\in \C^{5\times 5}$, satisfies
\begin{equation}\label{blowup1}
\det G(x,v,\omega) = (1 -\vert \omega \vert^4)^2/4.
\end{equation}
\item \label{cxstructure3} The distribution $\D$ from \eqref{donz} is involutive and satisfies
\begin{equation*}
\D \cap \bar \D = \begin{cases}
0 & \text{on } Z\backslash SM\\
\spn_\C X & \text{on }  SM.
\end{cases} 
\end{equation*}
In particular, $Z^\interior$ has a complex structure for which $\D = T^{0,1}Z^{\interior}$.
\end{enumerate}
\end{lemma}

\begin{proof}
For \ref{cxstructure1} we use the structure equation  $[\eta_\pm,V]=\mp i \eta_\pm$ to the effect that
\begin{equation*}
[\omega^2 \eta_+ + \eta_-,\V] = \omega^2 [\eta_+,V] + [\eta_-,V] + [\omega^2 \eta_+,-i \omega \partial_\omega] = -i \omega^2 \eta_+ + i\eta_- + 2i \omega^2 \eta_+,
\end{equation*}
which gives the first relation; the second one is obvious. To check $S^1$-invariance, denote by $\xi$ either of the two vector fields $\omega^2 \eta_+ + \eta_-$ or $\partial_{\bar \omega}$ and define, for $t\in \R$,
\begin{equation*}
\xi_t(x,v,\omega) = \d \varphi_{-t}^\V\left(\xi( \varphi^\V_t(x,v,\omega))\right).
\end{equation*}
The Lie derivative of $\xi$ along $\V$ equals $\L_V\xi=-[\xi,\V]=-i\xi$. Hence  $({\d }/{\d t})\xi_t = -i\xi_t$ for all $t\in \R$, which means that $\xi_t=\exp(-it)\xi_0$ and thus the complex line bundle spanned by $\xi$ is $S^1$-invariant.

For \ref{cxstructure2} one checks, e.g., that
\begin{equation*}
\langle \omega^2\eta_+ + \eta_-, \omega^2\eta_+ + \eta_- \rangle =  \vert \omega \vert^4\cdot   \langle \eta_+,\eta_+\rangle + \langle \eta_-,\eta_-\rangle =\left(  \vert \omega \vert^4 + 1\right)/2,
\end{equation*}
where we used that $\sqrt 2 \eta_+$ and $\sqrt 2 \eta_-$ are orthonormal. Proceeding similarly with the other combinations, one sees that $G(x,v,\omega)$ is a block matrix with blocks
\begin{equation}\label{grammatrix}
\left[ \begin{matrix}
(1 + \vert \omega \vert^4)/2 & \omega^2\\
\bar \omega^2 & (1+\vert \omega \vert^4)/2
\end{matrix}\right]
\quad \text{ and } \quad 
\left[
\begin{matrix}
1 + 2\vert \omega \vert^2 & i \bar \omega & -i \omega\\
-i \omega & 1 & 0 \\
i \bar\omega & 0 & 1
\end{matrix}
\right],
\end{equation}
and the expression for $\det G(x,v,\omega)$ follows by a simple computation.

For \ref{cxstructure3}, note that $\omega^2 \eta_+ + \eta_-$ and $\partial_{\bar \omega}$ commute on $SM\times \DD$, hence $\widetilde \D$ and consequently $\D$ are involutive. On $p^{-1}(Z\backslash SM)=\{\vert \omega \vert < 1\}$ we have $\widetilde \D\cap \overline{\widetilde \D}=0$ due to  \ref{cxstructure2}, which implies that $\D\cap \bar \D=0$ away from $SM$. Further,
\begin{equation*}
X = \eta_+ + \eta_- \in \widetilde \D\cap \overline {\widetilde \D}\quad \text{ on } \{\omega = 1\}
\end{equation*}
and as $p\vert_{\{\omega = 1\}}\colon SM\times \{1\}\rightarrow SM$ is the identity, this  implies $X\in \D \cap \bar \D$ on $SM$. The dimension of $\D\cap \bar \D$ at $[(x,v,\omega)]$ equals the deficiency of $G(x,v,\omega)$, which is $1$ by \eqref{grammatrix}, hence $\D \cap \bar \D$ is indeed spanned by $X$ on $SM$. Finally, note that in the interior of $Z$ we have $T_\C Z = \D \oplus \bar \D$ and this induces a unique complex structure $J$ as follows: for $(x,v)\in Z^\interior$ define $J_{(x,v)}\colon (T_\C Z)_{(x,v)} \rightarrow (T_\C Z)_{(x,v)}$ by $J_{(x,v)}(w_1\oplus w_2) = -i w_1 + i w_2$, where $w_1\oplus w_2$ is the unique decomposition into $\D$ and $\bar \D$--components. It is straightforward to verify that $J$ preserves the {\it real} tangent space $TZ^\interior$
and by construction it is an almost complex structure with $T^{0,1}Z^\interior =\D$. Involutivity of $\D$ is equivalent to the formal integrability of $J$ and thus the Newlander-Nirenberg theorem implies that $J$ is a complex structure.
\end{proof}

The preceding lemma shows that $Z^{\interior}$ is a complex surface in the classical sense, but with complex structure degenerating at $SM$. Nevertheless the $\bar \partial$-complex of $Z^\interior$ can be extended to all of $Z$ in a way that is smooth up the boundary. We will built this extended $\bar \partial$-complex from scratch and show {\it a posteriori} that  it coincides with the standard one in the interior. On an open set $U\subset Z$, we define
\begin{equation}\label{dbar}
\Omega^0(U) \xrightarrow{\bar \partial} \Omega^{0,1}(U) \xrightarrow{\bar \partial} \Omega^{0,2}(U)
\end{equation}
as follows: note that the spaces $C^\infty(U)$ and $C^\infty(p^{-1}(U))$ are well defined also when $U\cap \partial Z\neq \emptyset$, and contain $\C$-valued functions, smooth up to the boundary. Then
\begin{eqnarray*}
\Omega^0(U) &:= &   \{h\in C^\infty(p^{-1}(U)): \V h = 0\} 
\cong C^\infty(U) \\ 
\Omega^{0,1}(U)&:= &\{(h_1,h_2)\in C^\infty(p^{-1}(U))^2: \V h_j+ i h_j = 0 ~ (j=1,2)\}\\
\Omega^{0,2}(U) & := & \{h\in C^\infty(p^{-1}(U)): \V h + 2i h = 0  \}
\end{eqnarray*}
and we define
\begin{equation}\label{dbar2}
\bar \partial h :=\left((\omega^2 \eta_+ + \eta_- )h, \partial_{\bar \omega} h\right)\quad \text{ and } \quad \bar \partial (h_1,h_2) := (\omega^2 \eta_+ + \eta_-) h_2 - \partial_{\bar \omega} h_1,
\end{equation}
noting that $\bar \partial$ has the mapping properties indicated in \eqref{dbar} in view of part \ref{cxstructure1} of the preceding lemma. See Lemma \ref{lemcoo2} for a description of $\bar \partial$ in coordinates. If $U\cap \partial Z = \emptyset$, then we recover the usual $\bar \partial$-complex of the complex surface $Z^{\interior}$, via isomorphisms
\begin{equation}\label{isotostandard}
\Omega^{0,q}(U)\cong \{\alpha\in \Omega^q(U):\alpha \vert_{\bar \D} = 0\},\quad q=1,2,
\end{equation}
exhibited in the following lemma.

\begin{lemma}[Comparison with standard $\bar \partial$-complex]\label{lemisotoast} Let $U\subset Z^\interior$ be open and consider on $p^{-1}(U)\subset \{\vert \omega \vert <1 \} $ the complex $1$-forms 
\begin{equation}
\tau= \frac{1}{1- \vert \omega \vert^4} \left(\eta_-^\vee - \bar \omega^2 \eta_+^\vee\right)\quad \text{ and } \quad \gamma = \d \bar \omega - i \bar \omega V^\vee, 
\end{equation}
where $\{\eta_+^\vee,\eta_-^\vee,V^\vee\}$ is the coframe on $SM$ that is dual to $\{\eta_+,\eta_-,V\}$.  Then:
\begin{enumerate}[label=(\roman*)]
\item \label{isotoast1} The following duality relations hold true:
\begin{equation*}
\left.\begin{array}{lcr}
\tau(\omega^2\eta_++\eta_-) =\gamma(\partial_{\bar \omega})&\equiv& 1\\
\tau(\partial_{\bar \omega}) =\gamma(\omega^2\eta_++\eta_-)&\equiv& 0
\end{array}\right.
\quad \text{ and } \quad \tau,\gamma = 0 \text{ on } \overline{ \widetilde \D}\oplus \spn \V
\end{equation*}
\item \label{isotoast2} For $(h_1,h_2)\in \Omega^{0,1}(U)$ and $h\in \Omega^{0,2}(U)$ the differential forms
$h_1\tau + h_2 \gamma$ and $h\tau \wedge \gamma$
are $S^1$-invariant and the maps
\begin{equation} \label{isotoast3}
(h_1,h_2)\mapsto p_*(h_1\tau + h_2 \gamma)\quad \text{ and } \quad h\mapsto p_*(h\tau \wedge \gamma)
\end{equation}
yield isomorphisms as in \eqref{isotostandard} for $q=1$ and $q=2$, respectively.
\item \label{isotoast4} The isomorphisms from \ref{isotoast2} intertwine the $\bar \partial$-operators from \eqref{dbar2} with the standard $\bar \partial$-operators of the complex surface $U$.
\end{enumerate} 
\end{lemma}

\begin{proof}
The proof of \ref{isotoast1} is a simple computation that we omit. As a consequence $\tau(\xi)$ is constant for $\xi\in \{\omega^2 \eta_+ + \eta_-,\bar \omega^2 \eta_- + \eta_+,\V, \partial_{\bar \omega},\partial_\omega\}$ and, taking Lie derivatives, we see that
\begin{equation*}
0 = \L_\V(\tau(\xi)) = \L_\V\tau(\xi) + \tau(\L_\V\xi)=(\L_\V\tau - i\tau)(\xi),
\end{equation*}
where in the last step we used Lemma \ref{cxstructure}\ref{cxstructure1}, noting that while e.g., $\L_\V(\partial_{\omega})=+i\partial_\omega$, the equality still holds true as $\tau(\partial_\omega)=0$. By Lemma \ref{cxstructure}\ref{cxstructure2} such $\xi$'s form a frame over $p^{-1}(U)$ and thus $\L_\V\tau =i\tau$. Arguing similarly, also $\L_\V\gamma=i\gamma$ follows.
This implies, e.g., that
\begin{equation}\label{isotoast5}
\L_\V(h_1\tau)= \V h_1 \tau + h_1 \L_\V \tau  = - i h_1 \tau + i h_1  \tau = 0
\end{equation}
and overall we obtain the desired $S^1$-invariance. Hence 
$
\alpha=p_*(h_1\tau + h_2 \gamma)$ and $\alpha' =  p_*(h\tau \wedge \gamma)$ are well defined differential forms on $U$. Using  part \ref{isotoast1}  we see that $\alpha,\alpha' = 0$ on $\bar \D$ such that \eqref{isotoast3} indeed defines a map as in \eqref{isotostandard}.

 We obtain inverse maps as follows: Given $\alpha \in 
\Omega^1(U)$, we can express its lift $p^*\alpha$ in terms of the $1$-forms 
$\{\tau,\gamma,\bar \tau,\bar \gamma, \V^\vee \}$ (with $\V^\vee$ defined similarly to $V^\vee$), which frame $T^*_\C (p^{-1}(U))$ by part \ref{isotoast1}. If $\alpha \vert_{\bar \D} = 0$, then only the $\tau$- and $\gamma$-coefficients of $p^*\alpha$ are nonzero, which is to say that $p^*\alpha = h_1 \tau + h_2 \gamma$ for some $h_1,h_2\in C^\infty(p^{-1}(U))$. One computes that
\begin{equation}
0=\L_\V(p^*\alpha) = (\V h_1 + ih_1) \tau + (\V h_2 + i h_2) \gamma,
\end{equation}
hence $(h_1,h_2)\in \Omega^{0,1}(U)$ and we have found the desired preimage of $\alpha$. The argument for $q=2$ is completely analogous.

For \ref{isotoast4} consider $f\in \Omega^0(U)$ with lift $h=p^*f$. Then $\bar \partial f \in \Omega^1(U)$ (in the classical sense) is uniquely defined by $\bar \partial f = \d f$ on $\D$ and $\bar \partial f = 0$ on $\bar \D$, hence
\begin{equation}
p^*(\bar \partial f) = \begin{cases}
p^*(\d f) & \text{ on } \widetilde \D\\
0& \text{ on } \overline{\widetilde \D} \oplus \spn \V.
\end{cases}
\end{equation}
Thus $p^*(\bar\partial f) = h_1 \tau + h_2 \gamma$, where $h_1=(p^*\d f)(\omega^2 \eta_+ + \eta_-)=(\omega^2 \eta_++\eta_-)h$ and $h_2 = p^*\d f(\partial_{\bar \omega})= \partial_{\bar \omega} h$ -- this gives the desired intertwining property on $\Omega^0(U)$. Similarly one shows that if $\alpha \in \Omega^1(U)$ with $\alpha \vert_{
\bar \D} = 0$ and lift $p^*\alpha = h_1 \tau + h_2 \gamma$, then $\bar \partial \alpha\in \Omega^2(U)$ (in the classical sense) satisfies $p^*(\bar \partial \alpha) = h \tau \wedge \gamma$ with \begin{equation*}
h=p^*(\d \alpha)(\xi,\partial_{\bar \omega}) = \d(h_1 \tau)(\xi,\partial_{\bar \omega}) + \d (h_2 \gamma)(\xi,\partial_{\bar \omega}),
\end{equation*}
where $\xi = \omega^2\eta_++\eta_-$. The right hand side is easily computed in view of part \ref{isotoast1} and $[\xi,\partial_{\bar \omega}]=0$ and one obtains $h=\xi h_2 - \partial_{\bar \omega}h_1$, as desired.
\end{proof}

\begin{definition} A function $f\in \Omega^0(U)$ on an open set $U\subset Z$ is called {\it holomorphic}, if $\bar \partial f =0\in \Omega^{0,1}(U)$. We then write $f\in \O(U)$. 
\end{definition}

We emphasise that holomorphic functions on $Z$ are -- by definition -- smooth up the boundary. We can now draw the first connection to transport problems on $SM$.

\begin{proposition}[Twistor correspondence A] \label{tcA} There is a one-to-one correspondence between holomorphic functions on $Z$ and fibrewise holomorphic first integrals on $SM$, implemented by the map
\begin{equation}\label{correspondenceA}
\O(Z) \xrightarrow{\sim} \{u\in \oplus_{k\ge 0}\Omega_k:Xu= 0 \}\subset C^\infty(SM),\quad f\mapsto f\vert_{SM}.
\end{equation}
\end{proposition}

\begin{proof}
Given $f\in \O(Z)$, denote $h = p^* f\in C^\infty(SM\times \DD)$. For fixed $(x,v)\in SM$, the function $h(x,v,\cdot)$ is holomorphic in $\{\vert \omega \vert <1\}$ and thus expands as %admits a power series expansion 
\begin{equation}\label{powerseries}
h(x,v,\omega) = \sum_{k\ge 0} \omega^k u_k(x,v),
\end{equation}
with coefficients $u_k(x,v)\in \C$. By Cauchy's integral formula and the $\V$-invariance,
\begin{equation}
u_k(x,v) = \frac1{2\pi i}\int_{\vert \zeta \vert =1} \frac{h(x,v,\zeta)}{\zeta^{k+1}} \d\zeta = \frac{1}{2\pi}\int_0^{2\pi} f(x,e^{i\theta}) e^{-ik\theta}\d\theta.
\end{equation}
This shows that  $u_k(x,v)$ depends smoothly on $(x,v)\in SM$ and moreover, that it is the $k$th Fourier mode of the fibrewise holomorphic function $u=f\vert_{SM}\in C^\infty(SM)$. Using the identity $X=\eta_++\eta_-$ and, again, holomorphicity of $f$ we see that
\begin{equation}
Xu = (\omega^2 \eta_+ + \eta_-) h \vert_{\omega =1} = 0,
\end{equation}
which shows that the map in \eqref{correspondenceA} is well defined. The map is clearly injective.

We construct an inverse map as follows: If $u\in \oplus_{k\ge 0} \Omega_k$ is a first integral, then its Fourier modes $u_k$ are easily seen to satisfy  $\Vert u_k \Vert_{C^m(SM)} = O(k^{-\infty})$ ($m\in \N_0$), such that  \eqref{powerseries} defines a function $h\in C^\infty(SM\times \DD)$. We compute that
\begin{equation} \label{powerseries2} \V h(x,v,\omega) = \sum_{k\ge 0} \omega^k (V u_k(x,v) - i k u_k(x,v))=0, 
\end{equation}
 which means that it descends to a function $f=p_*h \in \Omega^0(Z)$ with $f\vert_{SM} = u$. It remains to show that $f$ is holomorphic, or equivalently that
\begin{equation}
g:=(\omega^2 \eta_+ + \eta_-) h = 0 \quad \text{ and } \quad \partial_{\bar \omega}h =0.
\end{equation}
The latter equation is satisfied in view of the expansion \eqref{powerseries} and we know that $g(x,v,1)=0$ for all $(x,v)\in SM$, as $Xu=\eta_+ u + \eta_- u=0$. To see that $g$ indeed vanishes for {\it all} $\omega \in \DD$ note that
\begin{equation}
\partial_{\bar \omega}g=0 \text{ for } \vert \omega \vert \le 1\quad \text {and } \quad g = 0 \text{ for } \vert \omega \vert =1,
\end{equation} 
which follows in view of the previous observations from $[\omega^2\eta_++\eta_-,\partial_{\bar \omega}]=0$ and the $S^1$-invariance of $g$, respectively. Thus $g$ vanishes on all of $SM\times \DD$ by the maximum modulus principle  on $\DD$. This completes the proof.
\end{proof}

\begin{remark}\label{invariantforms}
By the same method of proof, we can associate to any  $u\in \oplus_{k\ge k_0}\Omega_k$ ($k_0\in \Z$) a function $h\in C^\infty(SM\times \DD)$  with $\partial_{\bar \omega}h=0$ by  setting
\begin{equation}\label{invariantforms2}
h(x,v,\omega)=\sum_{k\ge k_0} \omega^{k-k_0} u_k(x,v).
\end{equation}
This is easily checked to satisfy $(\V - ik_0)h=0$ such that for $k_0=0,-1,-2$ we can generate elements of $\Omega^0(Z)$, $\Omega^{0,1}(Z)$ and $\Omega^{0,2}(Z)$, respectively. Vice versa, if $h\in C^\infty(SM\times \DD)$ satisfies $\partial_{\bar \omega}h=0$ and  $(\V - ik_0)h=0$, then $u(x,v)=h(x,v,1)$ defines an element in $\oplus_{k\ge k_0} \Omega_k$.
\end{remark}

For the next result we consider the embedding $\iota_0\colon M\rightarrow Z$, $\iota_0(x)=(x,0)$ as zero section. If we equip $M$ with the complex structure induced by $g$ and the orientation, then $\iota_0$ becomes a holomorphic map.

\begin{lemma} The embedding $\iota_0\colon M\rightarrow Z$ as zero section is holomorphic.
\end{lemma}

\begin{proof} We have to show that for all $(x,v)\in SM$,
$$
(\iota_0)_*(T^{0,1}_{x}M) \subset \D_{(x,0)},\quad \text{ where }  T_x^{0,1}M = \spn_\C\{v+iv^\perp\}.
$$ 
To see this, pick a neighbourhood $U\subset M$ of $x$ and let $\tilde \iota_0\colon U\rightarrow SU\times \DD$ be a local lift of $\iota_0$ with $\tilde \iota_0(x)=(x,v,0)$. Then modulo $\spn_\C \V(x,v,0)$ we have
\begin{equation*}
(d\tilde \iota_0)_x(v + iv^\perp) \equiv X(x,v) - i X_\perp  \equiv \eta_-(x,v) \in 
\widetilde \D_{(x,0)},
\end{equation*}
which yields the desired inclusion after push-forward by $p$. 
\end{proof}

As a consequence there is a well-defined map
\begin{equation}
\iota_0^*:\O(Z)\rightarrow \O(M),
\end{equation}
where $\mathcal O(M)$ denotes the space of holomorphic functions on $M$ that are smooth up to the boundary. Under the identification $C^\infty(M)\cong \Omega_0$ this is also given as
\begin{equation}
\O(M)=\{g\in \Omega_0:\eta_-g=0\}.
\end{equation}
The following result is then a consequence of the characterisation in Proposition \ref{tcA} and a classical result of Pestov and Uhlmann on the surjectivity of the adjoint X-ray transform $I_0^*$.

\begin{corollary}[Cartan extension - transport version] \label{cartanex} Suppose $Z$ is the twistor space of a simple surface $(M,g)$. Then the map $\iota_0^*:\O(Z)\rightarrow \O(M)$ is onto. 
\end{corollary}

\begin{proof} As explained above, any $g\in \O(M)$ may be viewed as element in $\Omega_0$ with $\eta_-g =0$. By \cite{PeUh05} (see also Theorem 8.2.2 in \cite{PSU20}) there exists a solution $w\in C^\infty(SM)$ to $Xw = 0$ with Fourier mode $w_0=g$. Then $u = w_0+w_2+\dots$ is smooth, fibrewise holomorphic and satisfies $X u=0$. By the preceding proposition (and equation \eqref{powerseries} in the proof) $u$ gives rise to an element $f\in \O(Z)$  with
\begin{equation}
p^*f(x,v,\omega)=w_0(x) + \omega^2 w_2(x,v) + \dots,
\end{equation}
in particular, $\iota_0^*f(x)=p^*f(x,v,0)=w_0(x)=g(x)$, as desired.
\end{proof}

The preceding result can be viewed as `Cartan extension theorem' and implies in particular that the twistor  space $Z$ of a simple surface admits an abundance of holomorphic functions. This is first evidence for $Z$ behaving like a Stein surface, as claimed in \textsection \ref{twistorintro}.
Further evidence is provided by Theorem \ref{tog} and its corollaries.

%This hints at the general theme that
%{\it twistor spaces of simple surfaces behave like Stein surfaces} mentioned in the introduction. This is further substantiated by Theorem \ref{tog} and its corollaries.

\subsection{Coordinates and Euclidean case} \label{coordinates} It is instructive to express the twis\-tor space $Z$ in terms of isothermal coordinates on $(M,g)$. Suppose that $(x_1,x_2)$ are coordinates on an open subset $O\subset M$, such that $g\vert_O = e^{2\lambda}\d x^2$, with $\lambda\in C^\infty(O,\R)$. Viewing $O$ as subset of $\C$ with complex coordinate $z=x_1+ix_2$, we define
\begin{equation}\label{coo0}
Z_O= O\times \DD,\quad \D_O=\spn \{\Xi , \partial_{\bar \mu} \}\subset T_\C Z_O,
\end{equation}
where $\mu$ is the coordinate of the $\DD$-factor and the vector field $\Xi$ is defined by 
\begin{equation}
\Xi= e^{-\lambda}\left[\mu^2\partial_{z} + \partial_{\bar z} + \left(\mu^2 \partial_{z} \lambda - \partial_{\bar z} \lambda \right)(\bar \mu \partial_{\bar \mu} - \mu \partial_{\mu})\right].
\end{equation}
On $SM\vert_O$ we have coordinates $(x_1,x_2,\theta)$, where $\theta\in \R/(2\pi \Z)$ is the (oriented) angle of a unit vector with $\partial_{x_1}$ and there is an isomorphism
\begin{equation}
\varrho_O: SM\vert_O \xrightarrow{\sim}  O \times S^1,\quad (x_1,x_2,\theta) \mapsto (x_1+ix_2,e^{i\theta}),
\end{equation}
which is made implicit below.

The next lemma shows that $(Z_O,\D_O)$ is a  (degenerately) complex surface -- its proof is independent from the analogous  Lemma \ref{cxstructure} and the two constructions are seen to be equivalent below. We use the following notation:
\begin{equation}\label{defb}
\Lambda  := e^{-\lambda} \left(\mu^2 \partial_{z} \lambda - \partial_{\bar z} \lambda \right) \in C^\infty(Z_O,\C)
\end{equation}

\begin{lemma}[Complex structure in coordinates]\label{lemcoo1} \,
\begin{enumerate}[label=(\roman*)]
\item\label{coo1} $[\Xi,\partial_{\bar \mu}]=-\Lambda\partial_{\bar \mu}$, hence $\D_O$ is involutive.
\item \label{coo2} $\D_O\cap \bar \D_O = 0$ on $Z_O\backslash (O\times S^1)$.
\item \label{coo3} On $O\times S^1\cong SM\vert_O$ we have $\Xi = \mu X$.
\end{enumerate}
\end{lemma}
\begin{proof} Parts \ref{coo1} and \ref{coo2} follow from simple computations that we omit. For part \ref{coo3} note that
$
e^{i\theta}\partial_z = \frac12 \left(\cos \theta \partial_{x_1} + \sin \theta \partial_{x_2}\right) + \frac i2 \left(\sin \theta \partial_{x_1} - \cos \theta \partial_{x_2}\right),
$
hence the coordinate description of $X$ from \eqref{defX} is equivalent to
\begin{equation}\label{xcoo}
X= e^{-\lambda} \left(e^{i\theta} \partial_z + e^{-i\theta} \partial_{\bar z} + \left(e^{i\theta}\partial_z\lambda - e^{-i\theta} \partial_{\bar z}\lambda\right)\left(i\partial_{\theta}\right)\right).
\end{equation}
Under the isomorphism $\rho_O$ we have $\mu =e^{i\theta}$ such that $\bar \mu \partial_{\bar \mu} - \mu \partial_\mu = i \partial_{\theta}$ and hence $X= \mu^{-1} \Xi$ for $\vert \mu \vert =1$, as desired.
%\begin{equation}
%X=e^{-\lambda}\left(\cos \theta \partial_{x_1} + \sin \theta \partial_{x_2} + \left(-\partial_{x_1}\lambda \sin \theta + \partial_{x_2}\lambda \cos \theta \right) \partial_\theta\right).
%\end{equation} 
\end{proof}

Define a map $\kappa_O:Z_O\rightarrow Z$ by $\kappa_O(x,\mu)=p(x,0,\mu)$ (where $0$ stands for the angle $\theta=0$) and note that  $\kappa_O(Z_O)=\{(x,v)\in Z:x\in O\}$.

%Note that $Z$ is a fibre bundle over $M$ with projection $\pi_Z([(x,v,\omega)])=x$. In particular the set $O\subset Z$ corresponds to a subset $\pi_Z^{-1}(O)\subset Z$ and the next lemma shows that this is isomorphic to $Z_O$ as (degenerately) complex surface. 

\begin{lemma}[Comparison with invariant twistor space] \label{lemcoo2} \,
\begin{enumerate}[label=(\roman*)]
	\item \label{coo4} The map $\kappa_O$ is a diffeomorphism onto its image and $(\kappa_O)_*(\D_O)=\D$.
\item \label{coo5} Let $U\subset \kappa_O(Z_O)$ be open, then  pullback by $\kappa_O$ induces isomorphisms that fit into the commutative diagram
\begin{equation}\label{diagram}
\begin{tikzcd}
\Omega^0(U) \arrow{d}{\sim} \arrow{r}{\bar \partial} & \Omega^{0,1}(U) \arrow{d}{\sim} \arrow{r}{\bar \partial} & \Omega^{0,2}(U)\arrow{d}{\sim}\\
C^\infty(\kappa_O^{-1}(U)) \arrow{r}{\bar \partial}& C^\infty(\kappa_O^{-1}(U))^2 \arrow{r}{\bar \partial} & C^\infty(\kappa_O^{-1}(U)),
\end{tikzcd}
\end{equation}
where the $\bar \partial$-operators on the bottom are given (with $\Lambda$ as in \eqref{defb}) by
\begin{equation}
\bar \partial f = (\Xi f, \partial_{\bar \mu} f)\quad \text{ and } \quad \bar \partial (f_1,f_2) = (\Xi + \Lambda) f_2 - \partial_{\bar \mu} f_1.
\end{equation}
\end{enumerate} 
\end{lemma}

\begin{proof} For \ref{coo4} define $q:SM\vert_O\times \DD\rightarrow Z_O$ by  $q(x,\theta,\omega)=(x,e^{i\theta}\omega)$; then $q$ is smooth, $S^1$-invariant and satisfies $\kappa_O(q(x,\theta,
\omega))=[(x,0,e^{i\theta}\omega)]=p(x,\theta,\omega)$. In particular, $q$ descends to an inverse of $\kappa_O$, which is consequently a diffeomorphism. We claim that $q_*(\widetilde \D) \subset \D_O$ -- this will complete the proof of part \ref{coo4} by comparing ranks. Indeed, one derives, similarly to \eqref{xcoo}, the  coordinate expression
\begin{equation}\label{etacoo}
(\omega^2\eta_+ + \eta_-) = e^{-i\theta}e^{-\lambda}\left(\omega^2 e^{2i\theta} \partial_{z} + \partial_{\bar z} + (\omega^2 e^{2i\theta} \partial_z \lambda - \partial_{\bar z} \lambda) (i\partial_\theta)\right)
\end{equation}
and employs this to compute the push forwards
\begin{eqnarray*}
\d q_{(x,\theta,\omega)}(\partial_\theta) &=& (\partial_t\vert_{t=0}) q(x,\theta+t,\omega) = (\partial_t\vert_{t=0}) q(x,\theta,e^{it} \omega)\\
&  =& i e^{i\theta}\omega \partial_{\mu} - i e^{-i\theta}\bar \omega \partial_{\bar \mu} = -i(\bar \mu \partial_{\bar\mu} - \mu\partial_\mu)\\
\d q_{(x,\theta,\omega)}(\omega^2 \eta_++\eta_-) &=& e^{-i\theta} \Xi(q(x,\theta,\omega))\\
\d q_{(x,\theta,\omega)}(\partial_{\bar \omega})&=& e^{-i\theta} \partial_{\bar \mu},
\end{eqnarray*}
from which the claim follows.

For part \ref{coo5} we first note that the vertical arrows in \eqref{diagram} are defined as `pull-backs' by $\kappa_O$, understood as follows: Given a function $h\in C^\infty(p^{-1}(U))$ (representing an element in $\Omega^0(U), \Omega^{0,1}(U)$ or $\Omega^{0,2}(U)$), we write
$\kappa_O^* h(x,\mu) = h(x,0,\mu)$. Now consider $h\in \Omega^0(U)$, then as $\V h = 0$,
\begin{equation}\label{coo6}
\partial_\theta h(x,0,\mu) = Vh(x,0,\mu)= -i(\bar \omega \partial_{\bar \omega} - \omega \partial_{\omega})h(x,0,\mu)
\end{equation}
and together with  \eqref{etacoo} we obtain
\begin{equation}
\kappa_O^*\left((\omega^2\eta_++\eta_-) h\right)(x,\mu) = \Xi(\kappa_O^*h)(x,\mu).
\end{equation}
Similarly, $\kappa_O^*(\partial_{\bar \omega})h(x,\mu)=\partial_{\bar \mu}(\kappa_O^*h)(x,\mu)$ and thus the left square in \eqref{diagram} commutes. Next, if $(h_1,h_2)\in \Omega^{0,1}(U)$, then $\V h_j = -ih_j$ ($j=1,2$) and similarly to \eqref{coo6} we have
\begin{equation}
\partial_\theta h_2(x,0,\mu) = Vh_2(x,0,\mu)= -i(\bar \omega \partial_{\bar \omega} - \omega \partial_{\omega})h_2(x,0,\mu) - i h_2(x,0,\mu),
\end{equation}
such that
$
\kappa_O^*\left((\omega^2 \eta_+ + \eta_-)h_2\right)(x,\mu) = (\Xi+ \Lambda)(\kappa_O^* h_2)(x,\mu)
$.  The computation for $\partial_{\bar \omega}$ remains unchanged and thus also the right square in \eqref{diagram} commutes.
\end{proof}

We can gain more insight into the (degenerate) complex surface $Z$ in the case that  $(M,g)$ is a Euclidean domain. First suppose that $M=\R^2$, such that $Z=\C\times \DD$, with Cauchy-Riemann equations given in terms of
\begin{equation}
\Xi = \mu^2 \partial_{z} + \partial_{\bar z}\quad \text{ and } \quad \partial_{\bar \mu}.
\end{equation} 
Let $W=\C\times \DD$ be equipped with the standard complex structure, given in terms of $\partial_{\bar w}$ and $\partial_{\bar \mu}$ for coordinates $(w,\mu)\in \C\times \DD$. Then the map
\begin{equation}\label{beta1}
\beta: Z\rightarrow W,\quad (z,\mu) \mapsto (z-\mu^2\bar z,\mu)
\end{equation}
is holomorphic (in the sense that $\beta_*(\D)\subset \spn \{\partial_{\bar w},\partial_{\bar \mu}\}$) and maps the interior of $Z$ diffeomorphically onto the interior of $W$, with inverse given by
\begin{equation}
\beta^{-1}(w,\mu)=\left(\frac{w}{1+\vert \mu \vert^2} + \frac{2\mu\Re(\bar \mu w)}{1-\vert \mu \vert^4} ,\mu\right),\quad (w,\mu)\in W^{\interior}.
\end{equation}
Thus $Z^{\interior}$ is biholomorphically equivalent to a polydisk in $\C^2$ and the degeneracy of the complex structure is encoded in the `blow down' map $\beta$. More generally:

\begin{lemma}\label{flatstein} Suppose $M\subset \R^2$ is a Euclidean domain. Then the interior of its twistor space  $Z$ is a Stein surface that is biholomorphic to a domain in $\C^2$.
\end{lemma}

\begin{proof} The restriction of $\beta$ from \eqref{beta1} to  $Z^\interior=M\times \DD^{\interior}$ 
gives the desired embedding as domain in $\C^2$. In particular, the global holomorphic functions $\beta_1,\beta_2\in \mathcal O(Z^\interior)$  give a global coordinate system and separate points. To show that $Z^\interior$ is a Stein surface, 
 it thus remains to establish holomorphic convexity. To this end, define for $p=(z_*,\mu_*)\in \R^2\times \DD\backslash Z^\interior$ a function $f_p\in \O(Z^\interior)$ by
\begin{equation}
f_p(z,\mu)=
\begin{cases}
(\mu-\mu_*)^{-1}, & \vert \mu_*\vert =1,\\
\left ( (z-z_*)-\mu^2(\bar z -\bar {z_*})\right)^{-1}, & \vert \mu_* \vert<1, z_*\in \R^2\backslash M.
\end{cases}
\end{equation}
Let $K\subset Z^\interior$ be compact and consider the holomorphic hull $\hat K =\{(z,\mu)\in Z^{\interior}: \vert f(z,\mu)\vert \le \sup_K \vert f \vert \text{ for all } f\in \O(Z^{\interior})\}$. If $\hat K$ was {\it not} compact, it would contain a sequence $(z_n,\mu_n)$ with limit point $p$ as above, which leads to a contradiction, as $f_p$ is unbounded along that sequence. Thus $\hat K$ is compact and, as $K$ was arbitrary, the complex surface $Z^\interior$ is holomorphically convex.
\end{proof}

In fact, also the twistor space of a simple surface admits a natural, albeit less tractable,  holomorphic map $\beta:Z\rightarrow \C^2$ as follows: Passing to global isothermal coordinates and with $\Lambda$ as in \eqref{defb}, we may find a solution $u\in C^\infty(Z,\C)$ to \begin{equation}
\Xi u = \Lambda \quad \text{ and } \quad \partial_{\bar \mu}u =0 \quad \text{ on } Z.
\end{equation}
This follows from the existence of {\it scalar} holomorphic integrating factors on simple surfaces and is also a consequence of the vanishing result $H^1(Z,\mathcal{O})\equiv H^1_{\bar \partial}(Z,[0])=0$ from Corollary \ref{corco} below. Further, by Corollary \ref{cartanex} there exists a function $\beta_1\in \O(Z)$ with $\beta_1(z,0)=z$ for all $z\in M$ and one checks that
\begin{equation}
\beta(z,\mu) = \left(\beta_1(z,\mu),e^{u(z,\mu)}\mu \right) \in \C^2
\end{equation}
indeed defines a holomorphic map of similar form as \eqref{beta1} in the Euclidean case. While it would be interesting to know more about the behaviour of $\beta$  (e.g., is it also diffeomorphism on the interior of $Z$?), our approach to the transport Oka-Grauert principle does not require any such blow-down.

\subsection{Transport Oka-Grauert principle} We now define a `moduli space  of holomorphic vector bundles' over an open set $U\subset Z$. Noting that there are natural $\C^n$ and $\C^{n\times n}$-valued versions of the $\bar \partial$-complex \eqref{dbar}, we consider {\it partial connections} $A^{0,1}\in \Omega^{0,1}(U,\C^{n\times n})$ and maps
\begin{equation}\label{twisteddbar}
\Omega^0(U,\C^n) \xrightarrow{\bar \partial + A^{0,1}} \Omega^{0,1}(U,\C^n) \xrightarrow{\bar \partial + A^{0,1}} \Omega^{0,2}(Z,\C^n),
\end{equation}
defined in the obvious way. If  $A^{0,1}=(a_1,a_2)\in C^\infty(p^{-1}(U),\C^n)^2$ with $(\V + i)a_j=0$ $(j=1,2)$, a computation shows that the {\it curvature} of \eqref{twisteddbar} equals
\begin{equation}\label{curvature}
(\bar \partial + A^{0,1})^2 = (\omega^2 \eta_++\eta_-) a_2 - \partial_{\bar \omega} a_1 + [a_1,a_2] \in \Omega^{0,2}(U,\C^{n\times n}).
\end{equation}

%defined as follows: Suppose $A^{0,1}=(a_1,a_2)\in C^\infty(SM\times \DD,\C^n)^2$ with $(\V + i)a_j=0$ $(j=1,2)$, then for $h\in \Omega^0(U,\C^n)$ and $(h_1,h_2)\in \Omega^{0,1}(Z,\C^n)$,
%\begin{eqnarray*}
%(\bar \partial + A^{0,1}) h &=& \bar \partial h  + (a_1h,a_2 h) \\ 
% (\bar \partial + A^{0,1})(h_1,h_2)&=&\bar \partial(h_1,h_2) + (a_1 h_2 - a_2 h_1).
%\end{eqnarray*}
%In general \eqref{twisteddbar} might not be a complex and the requirement $(\bar \partial + A^{0,1})^2=0$ can be viewed as integrability condition. On a classical complex manifold (e.g.,~away from the degenerate boundary $\partial^0Z$) such an integrability condtion gives 

\begin{definition} \label{modulispace} For $U\subset Z$ open we define the {\it moduli space}
\begin{equation}
\M(U)=\M_n(U)= \{A^{0,1}\in \Omega^{0,1}(U,\C^{n\times n}): (\bar\partial+A^{0,1})^2=0 \}/\sim,
\end{equation}
where $A^{0,1}\sim B^{0,1}$  if and only if there exists $\varphi\in C^\infty(U,GL(n,\C))$ with
\begin{equation}
 B^{0,1}=\varphi^{-1} \bar \partial \varphi + \varphi^{-1} A^{0,1} \varphi.
\end{equation}
\end{definition}

If $U\cap \partial Z=\emptyset$, such that $U$ is a classical complex surface, then
\begin{equation}\label{dokro}
\M_n(U) \cong \left\{\begin{array}{l}\text{(topologically trivial) holomorphic vector bundles}\\
\text{of rank $n$ over $U$ up to isomorphism}
\end{array}\right. . 
\end{equation}
Indeed, a representative $A^{0,1}$ of a class in $\M_n(U)$ equips $U\times \C^n$ with the structure of a holomorphic vector bundle by declaring a local section $f:V\rightarrow \C^n$ (for $V\subset U$ open) to be holomorphic, if $(\bar \partial + A^{0,1})f=0$; equivalent representatives yield isomorphic vector bundles (cf.\,Chapter 2.1.5 in \cite{DoKr90}).

\smallskip

Recall from \eqref{defmho} that $\mho$ consists of attenuations $\A\in C^\infty(SM,\C^{n\times n})$ with Fourier coefficients $\A_k=0$ for $k< -1$. The group $\G$ from \eqref{defG} acts on $\mho$ by \eqref{rule} and we now establish a correspondence between the orbits of $\G$ and elements in $\M\equiv \M_n(Z)$. Define a map $\mho \rightarrow 
\M$ as follows: For $\A\in \mho$ let
 %we further use the notation $\mho_0$ for attenuations that admit a holomorphic integrating factor (HIF) and consider the `sequence'
%\begin{equation}\label{ses}
%0 \rightarrow \mho_0\hookrightarrow \mho \rightarrow %\mathfrak{M}\rightarrow 0,
%\end{equation}
%where the map $\mho\rightarrow \M\equiv \M_n(Z)$ is %defined as follows: For $\A\in \mho $, we let
\begin{equation}\label{littlea}
A^{0,1}(x,v,\omega) := (a,0)\equiv \left( \sum_{k\ge -1} \omega^{k+1} \A_k(x,v), 0 \right) \in \Omega^{0,1}(Z,\C^{n\times n}),
\end{equation}
noting that $A^{0,1}$ lies in $\Omega^{0,1}(Z,\C^{n\times n})$ and satisfies $(\bar \partial + A^{0,1})^2 = 0$ in view of \eqref{curvature} and Remark \ref{invariantforms}. We then map $\A$ to the equivalence class $[A^{0,1}]\in \M$. 

\begin{proposition}[Twistor correspondence B]\label{tcB} The map $\mho \rightarrow \M$, $\A\mapsto [A^{0,1}]$ is $\G$-invariant and descends to an injective map $\mho /\G \rightarrow \M$. If $(M,g)$ is diffeomorphic to a disk, the induced map is also surjective, such that \begin{equation} \label{moduliiso}
\mho/\G\cong \M.
\end{equation}
\begin{comment}
Suppose $Z$ is the twistor space of a surface $(M,g)$ which is diffeomorphic to the disk $\DD$.  The map $\mho \rightarrow \M$ defined in \eqref{littlea} is invariant under the action of $\G$ and the induced map $\mho /\G \rightarrow \M$ is a bijection.
Then the maps in \eqref{ses} define a short exact sequence (of pointed sets). This means:
\begin{enumerate}[label=(\roman*)]
\item \label{tcB1} An attenuation $\A\in \mho$ admits a HIF if and only if $[A^{0,1}]=[0]$.
\item \label{tcB2} The map $\A\mapsto [A^{0,1}]$ is surjective from $\mho$ onto $\mathfrak{M}$.
\end{enumerate}
Moreover, the map $\mho \rightarrow \M$ in \eqref{ses} induces an isomorphism $\mho/\G \rightarrow \M$, where $\mho/G$ is the quotient by the action of $\G$ in \eqref{rule}.
\end{comment}
\end{proposition}

In fact, the isomorphism \eqref{moduliiso} holds true for any orientable surface $(M,g)$. For a proof of this more general statement we refer to Remark 4.4.13 in \cite{Bohr-thesis}.

% whenever $(M,g)$ is globally conformal to a smooth domain in $\C$. This follows immediately from the proof.

\begin{proof} Suppose that $\A,\B\in \mho$, let $A^{0,1}=(a,0)$ as in \eqref{littlea} and define $B^{0,1}=(b,0)$ analogously. To demonstrate $\G$-invariance, we assume that $\A \act F=\B$ for some $F\in \G$ and consider the function
\begin{equation} \label{tcB3}
\phi(x,v,\omega) = \sum_{k\ge 0} \omega^k F_k(x,v),
\end{equation}
which is smooth and $S^1$-invariant by Remark \ref{invariantforms}; further it satisfies $\det \phi(x,v,\omega) \neq 0$ for all $\omega\in \DD$. To see this,  we define $\psi$ as in \eqref{tcB3}, but with $F$ replaced by $F^{-1}$. Then $h=\phi \psi - \Id$ satisfies $h(x,v,1)=0$ by construction, and, due to $S^1$-invariance,
\begin{equation}
h = 0 \text{ on } \{\vert \omega \vert = 1\} \quad \text{ and } \quad \partial_{\bar \omega} h = 0 \text { on } \{\vert \omega \vert\le 1\}.
\end{equation} 
By the maximum principle, $h\equiv 0$, which means that $\psi$ is an inverse for $\phi$.

 We claim that $A^{0,1}$ is equivalent to $B^{0,1}$ via the gauge $\varphi=p_*\phi \in C^\infty(Z,GL(n,\C))$, which is is to say that
\begin{equation}\label{tcB4}
\varphi^{-1} (\bar \partial + A^{0,1})\varphi \equiv \left(\phi^{-1}(\omega^2 \eta_+ + \eta_- + a)\phi, \phi^{-1}\partial_{\bar \omega}\phi\right)=(b,0)\in \Omega^{0,1}(Z,\C^{n\times n}).
\end{equation}
Evidently $\partial_{\bar \omega} \phi = 0$, so it remains to show that the function $g =  \phi^{-1}(\omega^2 \eta_+ + \eta_- + a)\phi - b$ vanishes identically. To see this, note that
$g(x,v,1)=\A\act F(x,v)-\B(x,v)=0$. Moreover $(\V + i)g=0$ which means that $g$ only changes phase along the flow of $\V$ and thus
\begin{equation}
g = 0 \text{ on } \{\vert \omega \vert  = 1\} \quad \text{ and } \quad \partial_{\bar \omega}g = 0 \text{ on } \{\vert \omega \vert < 1\},
\end{equation}
where holomorphicity in $\omega$ is easily checked. In particular the maximum modulus principle on $\DD$ applies to yield $g\equiv 0$.

To show that the induced map $\mho/\G\rightarrow \M$ is injective, we assume that $[A^{0,1}]=[B^{0,1}]$. This means that \eqref{tcB4} holds true, where now $\phi$ is defined as $p^*\varphi$ for an appropriate gauge $\varphi\in C^\infty(Z,GL(n,\C))$. In particular $\phi$ is holomorphic in $\omega$ and thus admits a series expansion as in \eqref{tcB3} with coefficients $F_k(x,v)\in \C^{n\times n}$ ($(x,v)\in SM$). Similar to the proof of Proposition \ref{tcA} one checks that $F_k\in \Omega_k$ such that 
\begin{equation}
F(x,v) = \phi(x,v,1)
\end{equation}
defines a smooth, $GL(n,\C)$-valued map on $SM$ with both $F$ and $F^{-1}$ being fibrewise holomorphic. Further, evaluating \eqref{tcB4} at $\omega=1$ yields $\A\act F=\B$, as desired.

To establish surjectivity of $\mho/\G\rightarrow \M$ we need to show that each class in $\M$ admits a representative $A^{0,1}=(a_1,a_2)\in \Omega^{0,1}(Z,\C^{n\times n})$ with $a_2\equiv 0$. Indeed, in that case $a_1\in C^\infty(SM\times \DD,\C^{n\times n})$ satisfies $\partial_{\bar \omega}a_1=0$ by the curvature condition, hence
\begin{equation}
a_1(x,v,\omega) = \sum_{k\ge -1} \omega^{k+1} \A_k(x,v)
\end{equation}
for coefficients $\A_k(x,v)$, which can be seen to lie in $\Omega_k$ as in the proof of Proposition \ref{tcA}; in particular  $\A(x,v)=a_1(x,v,1)$ is a preimage of $[A^{0,1}]$. 

We now make use of the fact that $M$ is diffeomorphic to the disk $\DD$, such that global isothermal coordinates become available. Using the description from \textsection \ref{coordinates}, the twistor space is then given by $Z=\DD\times \DD$ and a  representative of a class in $\M$ is a tuple $(b_1,b_2)\in C^\infty(Z,\C^{n\times n})^2$ obeying the curvature condition $(\Xi+\Lambda) b_2 - \partial_{\bar \mu} b_1 +[b_1,b_2]= 0 $, where $\Xi,\partial_{\bar \mu}$ are as in \eqref{coo0} and $\Lambda$ is as in \eqref{defb}.
Then by the Oka-Grauert principle on the $\mu$-disk (Lemma \ref{ogdisk}) there exists a solution $\varphi\in C^\infty(Z,GL(n,
\C))$ of
\begin{equation}
\partial_{\bar \mu} \varphi - \varphi b_2 = 0,
\end{equation}
which means that
$
(a_1,a_2):=\varphi(\Xi + b_1 , \partial_{\bar \mu} + b_2) \varphi^{-1}\equiv \left(\varphi \Xi \varphi^{-1} + \varphi b_1 \varphi^{-1},0\right)
$
defines an equivalent representative with $a_2\equiv 0$, as desired.
\end{proof}

In view of the preceding correspondence principle, Theorem \ref{mainhif} can be reformulated as:% `Oka-Grauert principle', which classically holds true for Stein surfaces: 

\begin{theorem}[Transport Oka-Grauert principle]\label{tog} Suppose $Z$ is the twistor space of  a simple surface $(M,g)$. Then $\M=\M_n(Z)=0$ for all $n\in \N$. \qed
\end{theorem}

%In classical complex geometry the Oka-Grauert principle asserts that on a Stein surface every topologically trivial holomorphic vector bundle is also holomorphically trivial. 
%It is an interesting question, whether the twistor space of a simple surface has a genuine Stein surface as interior -- thus making available the standard Oka-Grauert principle on $Z^{\interior}$. However, we emphasise that this would not imply Theorem \ref{tog}, as it fails to address the `degeneracy' of the complex structure at $\partial^0Z$.

Note that our proof does {\it not} rely on an integrability theorem as in \cite[Theorem 2.1.53]{DoKr90} -- that is, we do not first establish the existence of local holomorphic frames, which are then `glued' by means of a Cartan lemma as in the proof of the standard Oka-Grauert principle. A local integrability theorem -- for general twistor spaces -- follows a posteriori:

\begin{corollary}[Local integrability]\label{localint} Let $Z$ be the twistor space of an arbitrary oriented Riemannian surface. Consider a class $[A^{0,1}]\in \M_n(Z)$ and a point $p=(x,v)\in Z$ with $x\in M^\interior$. Then there exists an open neighbourhood $U$ of $p$ and a gauge $\varphi\in C^\infty(U,GL(n,\C))$ with $(\bar \partial + A^{0,1})\varphi=0$.
\end{corollary}

\begin{proof}
There exists a simple surface $M_1\subset M^{\interior}$ containing $x$ in its interior -- its twistor space $Z_1$ is then a subset of $Z$. By Theorem \ref{tog}, we have $[A^{0,1}\vert_{Z^1}]=0\in \M_n(Z_1)$ and thus the corollary follows with $U=Z_1\vert_{M_1^\interior}$.
\end{proof}

A further consequence is the following `vanishing theorem' in the spirit of Cartan's Theorem B. In fact, this is a reformulation of the linear result in Proposition \ref{xaonto} (modulo tame estimates) and thus does not require an inverse function theorem.

\begin{corollary}\label{corco} Suppose $Z$ is the twistor space of a simple surface $(M,g)$. Then for $[A^{0,1}]\in \M$ the $1$st cohomology of the twisted $\bar \partial$-complex \eqref{twisteddbar} vanishes, i.e.
\begin{equation}\label{cohomology}
H^1_{\bar \partial}(Z,[A^{0,1}])\equiv\frac{\ker \left((\bar \partial + A^{0,1})\vert_{\Omega^{0,1}(Z)}\right)}{ \operatorname{im}  \left((\bar \partial + A^{0,1})\vert_{\Omega^0(Z)}\right)} = 0.
\end{equation}
\end{corollary}

\begin{proof} We give a brief sketch: It is straightforward to see that \eqref{cohomology} is gauge-invariant, so by Proposition \ref{tcB} we may assume that $A^{0,1}=(a,0)$, where $a$ is as in \eqref{littlea} for an attenuation $\A \in \mho$. Next, using solvability of the $\partial_{\bar \omega}$-equation, any cohomology class may be represented by a tuple $(h_1,0)\in \Omega^{0,1}(SM,\C^n)$. Via Remark \ref{invariantforms} the function $h_1$ gives rise to a element in $\oplus_{k\ge -1}\Omega_k$ and we are in the setting of Proposition \ref{xaonto}. From this, one deduces that there is a solution $h\in \Omega^0(Z,\C^n)$ to $(\bar \partial + A^{0,1})h=(h_1,0)$ -- first for $\omega=1$, then for all $\omega \in \DD$ by invariance and the maximum principle.
\end{proof}

%First note that the cohomology group in \eqref{cohomology} is invariant under a change of gauge, so by Proposition \ref{tcB} we may assume that $A^{0,1}=(a,0)$, where $a$ is as in \eqref{littlea} for an attenuation $\A \in \mho$. Next, we claim that every cohomology class has a representative of the form $(h_1,0)\in \Omega^{0,1}(Z,\C^n)$ with $\partial_{\bar \omega}h_1=0$. Indeed, starting with an arbitrary representative $(\tilde h_1,\tilde h_2)$, we can solve the equation $\partial_{\bar \omega}\tilde h = - \tilde h_2$ with some $\tilde h\in \Omega^0(Z)$, such that 
%\begin{equation}
%[(\tilde h_1,\tilde h_2)] = [(\tilde h_1,\tilde h_2)+(\bar \partial + A^{0,1})\tilde h]=[(*, \tilde h_2 + \partial_{\bar \omega}\tilde h)]=[(*,0)]\in H^1_{\bar \partial}(Z,[A^{0,1}]),
%\end{equation}
%as desired. It remains to prove that any given $(h_1,0)\in \Omega^{0,1}(Z,\C^n)$ with $\partial_{\bar \omega}h_1=0$ is null-cohomologous. This is shown by means of Proposition \ref{?}, using Remark \ref{invariantforms} to translate between transport equations and $\bar \partial$-equations: Setting $u(x,v) = h_1(x,v,1)$, we obtain a function $u\in \oplus_{k\ge -1} \Omega_k$ on $SM$ and by the proposition there is a solution $f\in \oplus_{k \ge 0}
%\Omega_k$ to  $(X+\A)f=u$. Then $h(x,v,\omega)=\sum_{k\ge 0} \omega^k f_k(x,v)$ defines a function $h\in \Omega^0(Z,\C^n)$ with $(\omega^2 \eta_++\eta_- + a)h=0$ on  $SM\times \{\omega=1\}$. Using the same arguments for Proposition \ref{tcB} one shows that in fact $(\bar \partial + A^{0,1})h=(h_1,0)$, which completes the proof.

\subsection{Discussion of related work}\label{relatedwork} 

The twistor space $Z$ considered in this article is closely related to the more classical twistor notion from \cite{Dub83,OR85}, used recently, e.g., in the context of projective structures \cite{Met21,MePa20}. To explain this relation, we first note that the constructions from \textsection \ref{cxsurface} can be carried out in greater generality by substituting the vector fields $\V$ and $\xi=\omega^2\eta_++\eta_-$ by
\begin{equation}
\V_n = V + i n (\bar \omega \partial_{\bar \omega} - \omega \partial_\omega),\quad \text{ and} \quad \xi_k =\begin{cases}
	\omega^k \eta_+ + \eta_- & k \ge 0\\
	\bar \omega^{-k} \eta_+ + \eta_- & k<0
\end{cases},
\end{equation} 
for $n,k\in \Z$, respectively. A computation similar to Lemma \ref{cxstructure} shows that $[\xi_k,\V_n]=i\xi_k$ if and only if $nk=2$, such that we obtain four twistor spaces
\begin{equation}
Z(n) = \big((SM\times \DD)/S^1,~ \D_n=p_*\spn \{\xi_{2/n},\partial_{\bar \omega} \}  \big),\quad n\in \{\pm 1, \pm 2\},
\end{equation}
where the quotient is taken with respect to the flow of $\V_n$. Then $(x,v,\omega)\mapsto (x,v,\omega^2)$ induces holomorphic maps $Z(\pm 1)\rightarrow Z(\pm 2)$ and, in particular, $Z\equiv Z(1)$ may be viewed as branched double cover of the space $Z(2)$ (branched double covers were also found useful in \cite{LM10}).

We claim that the interior of $Z(2)$ is precisely the twistor space considered in the articles mentioned above. Using the description in \cite{MePa20}, this means that there is a biholomorphic map
\begin{equation} \label{defF}
\mathcal{F}: Z(2)^\interior \xrightarrow{\sim} P/CO(2),
\end{equation}
where $P$ is the oriented frame bundle and $CO(2)$ is the group of dilations and rotations of $\R^2$. Here we assume that the projective class $\mathfrak{p}$ that is used to define the complex structure on $P/CO(2)$, as explained  in  \cite[\textsection 4.1]{MePa20}, is given by $\mathfrak{p}=[\nabla^g]$ for the Levi-Civita connection $\nabla^g$ of $(M,g)$. To construct $\mathcal{F}$, consider
\begin{equation} 
SM\times \DD^\interior \xrightarrow{\mathcal{R}} P \times \DD^\interior \xrightarrow{\Upsilon} P,
\end{equation}
with $\Upsilon$ as in \cite[\textsection 4.1]{MePa20} and $\mathcal{R}(x,v,\omega)=(x,f_v,-\omega)$, where $f_v=(v,v^\perp)\in P_x$. Here $v^\perp$ is the rotation of $v$ by $\pi/2$, counterclockwise with respect to the orientation of $M$. %, i.e., $\mathcal{R}$ {\it reverses} orientation in the fibres. 
As below equation (4.6) in \cite{MePa20} one checks that
\begin{equation}
\mathcal{R}\left((x,v,\omega)\act e^{it}\right) = \left((x,f_v)\act R_t, R_t^{-1}\lact (-\omega)\right) = \mathcal{R}(x,v,\omega)\act R_t,
\end{equation}
where $R_t\in GL^+(2,\R)$ is the rotation matrix corresponding to $e^{it}$. This shows that $\mathcal{R}$ induces a smooth map between the quotient spaces  $Z(2)^\interior \equiv (SM\times \DD^{\interior})/S^1$ and $P\times_{GL^+(2,\R)}\DD^{\interior}$. Also $\Upsilon$ descends to quotient spaces and thus $\mathcal{F}([(x,v,\omega])=[\Upsilon\circ\mathcal{R}(x,v,\omega)]$ defines a smooth map $\mathcal{F}$ as in \eqref{defF}.  To check that $\mathcal{F}$ is also holomorphic one can use the description of $(1,0)$-forms on $P$ from \cite{MePa20}, in particular the computation of their pull-backs by $\Upsilon$ in $(4.3)$ and $(4.5)$. Pulling these back by $\mathcal{R}$ one obtains (nonzero multiples) of the following $1$-forms on $SM\times \DD^\interior$:
\begin{equation}
\eta_+^\vee - \omega \eta_-^\vee \quad \text{ and } \quad \d \omega + 2i \omega V^\vee%\quad \text{ on } SM\times \DD^\interior.
\end{equation}
After complex conjugation these equal precisely the $1$-forms in Lemma \ref{lemisotoast} (with an additional factor $2$ which is due to the choice of $n=2$ here). This shows that $\mathcal{F}$ is a holomorphic immersion. It is easily checked that $\mathcal{F}$ is bijective, so overall we obtain an isomorphism as in \eqref{defF}.

%The real $4$-manifold $P^\ominus/CO(2)$ can be equipped with a complex structure depending only on the choice of a projective class $\mathfrak{p}$ of connections on $TM$. It is an interesting question whether also the twistor space $Z(1)$ has a description in terms of connections (rather than Riemannian metrics) or their projective classes.

\medskip

Next, we briefly discuss the work of Eskin and Ralston \cite{EsRa04}, who proved a version of Theorem \ref{tog} in a Euclidean setting. They establish the existence of gauges $\varphi$ -- that is, $GL(n,\C)$-valued solutions to $(\bar \partial +A^{0,1})\varphi=0$ -- that are smooth in $Z^{\interior}$ and have a {\it continuous} extension to $\partial Z$. We give a brief outline of their argument in the language developed above. Recall from \textsection\ref{coordinates} that the twistor space of $\R^2$ admits a `blow down' map $\beta:Z\rightarrow W$ into a (closed) polydisk. The punchline of \cite{EsRa04} is that pull back by $\beta$ gives a surjective map 
\begin{equation}\label{famoustwelve}
\beta^*: \widetilde\M(W)\rightarrow \widetilde \M(Z)
\end{equation}
\begin{comment}
\begin{equation}
\beta^*: \left\{
\begin{array}{l}
\text{Holomorphic vector}\\
\text{vector bundles on } W
\end{array}
\right\}
\rightarrow 
\left\{
\begin{array}{l}
\text{Holomorphic vector}\\
\text{vector bundles on } Z
\end{array}
\right\}
\end{equation}
\end{comment}
between appropriate moduli spaces containing `holomorphic vector bundles' which are trivial away from a compact set and have a {\it continuous} extension to the boundary. The result then follows from  the classical Oka-Grauert principle on $W$ -- in a version with continuous boundary values (cf.\,\cite[Theorem~10.1]{Lei86}) -- which implies $\widetilde \M(W)=0$. Their approach thus parallels the desingularisation by means of a blow down in \cite{LeMa02}.

In order to establish surjectivity of $\beta^*$ as in \eqref{famoustwelve}, the authors prove a local integrability result as in Corollary \ref{localint} {\it a priori} (using the inverse function theorem in a H{\"o}lder space, where no loss of derivatives occurs) and then glue local solutions to $(\bar \partial +A^{0,1})\varphi=0$ by means of an appropriate Cartan lemma. The crucial step lies in showing that by such a gluing procedure one can arrange all transition functions to be of the form $h=\beta^* g$ for locally defined functions $g$ on $W$ (cf.\,equation (12) in \cite{EsRa04}) -- this is quite delicate and  encompasses removing singularites at $\beta(\partial Z)$ using methods from complex analysis.

%on $W$, which involves solving a Riemann-Hilbert problem with respect to $\beta(\partial Z)\subset \partial W$.

%for any pair $(n,k)\in \Z\times \N$, noting that so far we had chosen $n=-1$ and $k=2$. Similar to the proof of Lemma \ref{cxstructure} one shows that the line bundle spanned by $\xi(k)$ is invariant under the flow of $\V(n)$ iff $nk+2=0$.

\section{Range characterisations}\label{range}

We now turn to the range characterisation for the non-Abelian X-ray transform, starting with some general considerations that hold on any non-trapping surface $(M,g)$ with strictly convex boundary. Define
\begin{equation}\label{defB}
B:C^\infty(\partial SM,GL(n,\C))\rightarrow C^\infty(\partial_+SM,GL(n,\C)),\quad f\mapsto f(f^{-1}\circ\alpha)\vert_{\partial_+SM},
\end{equation}
where $\alpha:\partial SM\rightarrow \partial SM$ is the scattering relation of $(M,g)$ (see \textsection \ref{introrange}). To motivate the range characterisations in this section, consider the following diagram:
\begin{equation}\label{magicdiagram}
\begin{tikzcd}
  C_\Id^\infty(SM,GL(n,\C))  \arrow[two heads]{d}{(\cdot)\vert_{\partial SM}} \arrow{dr}\arrow[two heads]{r}{\I^*} & C^\infty(SM,\C^{n\times n}) \arrow{d}{\A\mapsto C_\A}\\
C_\Id^\infty(\partial SM,GL(n,\C)) \arrow{r}{B} & C^\infty(\partial_+SM,GL(n,\C)).
\end{tikzcd}
\end{equation}
Here and below, double-headed arrows stand for surjections; further $C_\Id^\infty(\cdot,GL(n,\C))$ is the space of maps which are homotopic to $\Id$ and we define \begin{equation}
\I^*(R)=-(XR)R^{-1}.
\end{equation}
The map $\I^*$, while not being an adjoint in any natural way, serves a similar purpose as $I^*$ in the linear theory. This is illustrated by the following lemma and further substantiated in \textsection \ref{rangeuattenuations}, where surjectivity results for $\I^*$ in different settings are derived using Theorem \ref{mainhif} on simple surfaces.

\begin{lemma} The diagram \eqref{magicdiagram} commutes and the map $\I^*: C_\Id^\infty(SM,GL(n,\C))\rightarrow C^\infty(SM,\C^{n\times n})$ is surjective.
\end{lemma}

\begin{proof}
We assume that the diagonal arrow is $B(\cdot\vert_{\partial SM}$) such that the lower triangle commutes. To check that the upper triangle commutes, let $F\in   C_\Id^\infty(SM,GL(n,\C))$ and denote $\A = \I^*(F)$. Let $G:SM\rightarrow GL(n,\C)$ be the unique continuous solution (differentiable along the geodesic flow) of the transport problem
\begin{equation}
XG = 0 \quad \text{ and } \quad G = F^{-1} \text{ on } \partial_-SM.
\end{equation}
Then $R=FG$ satisfies $(X+\A)R=0$ and $R=\Id$ on $
\partial_-SM$. In particular, using that $G\vert_{\partial SM}$ is $\alpha$-invariant, we have
\begin{equation}
C_\A = R \vert_{\partial_+SM} = B(R\vert_{\partial SM}) = F G (G^{-1}\circ \alpha)(F^{-1}\circ\alpha)\vert_{\partial_+SM} = B(F\vert_{\partial SM}).
\end{equation}
To check that $\I^*$ is onto, we have to show that any $\A\in C^\infty(SM,\C^{n\times n})$ admits a smooth, contractible integrating factor. To this end, embed $(M,g)$ into a closed manifold $(N,g)$ and extend $\A$ smoothly to $N$. Then there is a smooth cocycle $C:SN\times \R\rightarrow GL(n,\C)$ associated to $\A$, uniquely defined by
\begin{equation}
\partial_t C(x,v,t) + \A C(x,v,t) = 0 \text{ on } SN\times \R \quad \text{ and } \quad C(x,v,0)=\Id \text{ on } SN.
\end{equation}
Let $M_0\subset N$ be a non-trapping surface with strictly convex boundary, containing $M$ in its interior. Let $\tau_0$ be the exit time of $M_0$ (which is smooth on $SM$) and define
\begin{equation}
R_s(x,v) = [C(x,v,s\tau_0(x,v))]^{-1},\quad 0\le s \le 1, (x,v)\in SM.
\end{equation}
Then $R_1\in C^\infty(SM,GL(n,\C))$ is a smooth integrating factor for $\A$ (cf.\,Lemma 5.3.2 in \cite{PSU20}) and sending $s\rightarrow 0$ provides a homotopy with $\Id$, as desired.
\end{proof}

Using the preceding lemma, a simple diagram chase in \eqref{magicdiagram} reveals  a first range characterisation:

\begin{proposition}
An element $q\in C^\infty(\partial_+SM,GL(n,\C))$ is given as scattering data $q=C_\A$ of a general attenuation $\A\in C^\infty(SM,\C^{n\times n})$ if and only if $q=Bf$ for some $f\in C_\Id^\infty(\partial SM,GL(n,\C))$.\qed
\end{proposition}

%This is a characterisation in the spirit of \cite{PU} as the range of $C^\infty(SM,\C^{n\times n})\ni \A\mapsto C_\A$ is parametrised by  `boundary data'  $f\in C^\infty(\partial SM,GL(n,\C))$ by means of a `boundary operator', here given by $B$.
In the remaining section we give similar characterisations, when $\A$ is restricted to certain subclasses $\mathcal{A}\subset C^\infty(SM,\C^{n\times n})$ of attenuations. This involves finding appropriate domains $\mathcal{D}$ and {\it boundary spaces} $\mathcal{B}$ for which there is a diagram
\begin{equation}
\begin{tikzcd}
	\mathcal{D} \arrow[two heads]{r}{\I^*} \arrow{dr} \arrow[two heads]{d} &\mathcal{A}\arrow{d}{\A\mapsto C_\A}\\
	\mathcal{B} \arrow{r}{P} & C^\infty(\partial_+SM,GL(n,\C)),
\end{tikzcd}
\end{equation}
where $P$ is an appropriate {\it boundary operator} and the arrows emerging from $\mathcal D$ are surjective. If such a diagram commutes (up to gauge), then the range of $\mathcal{A}\ni \A\mapsto C_\A$ equals that of $P$ (up to gauge).

\subsection{Nonlinear Hilbert transforms} \label{hilbertsection} As building block for the boundary operators considered below we introduce here a nonlinear operator
\begin{equation}\label{hilbert1}
\H:C_\extend^\infty(\partial SM,GL(n,\C)) \rightarrow C_\extend^\infty(\partial SM,U(n))
\end{equation}
which is based on the factorisation theorems discussed in \textsection \ref{factorisation} (see Remark \ref{boundaryfac} for the $\extend$-notation) and serves as analogue of the Hilbert transform in the linear theory. Upon choosing a section ${\1}: M \rightarrow SM$ (or equivalently, fixing a trivialisation of $SM$) we define $\H\equiv\H_{\1}$ by
\begin{equation}
\H (r) = u^*,
\end{equation}
where $r=uf$ is a decomposition as in \eqref{bfac2}, normalised such that $u(x,{\1}(x))=\Id$. In reference to $\H$ we also define
\begin{align}
\hspace{-1em} \H^*:C_\extend^\infty(\partial SM,GL(n,\C)) \rightarrow C_\extend^\infty(\partial SM,U(n)),& \quad \H^*(r) =\H(r^{-1})\\
\hspace{-1em} \H^+:C^\infty(\partial SM,\he^+_n) \rightarrow C_\Id^\infty(\partial SM,GL(n,\C)),&\quad  \H^+(r) =\H(r^{1/2})r^{1/2}. \label{hilbert+}
\end{align}
Both transforms can be described in terms of suitable decompositions: Indeed, $\H^*(r)=u$, where $r=fu$ as in \eqref{bfac3}, normalised  such that $u(x,\1(x))=\Id$ and $\H^+(r)=f$, where $r=f^*f$ is a Birkhoff factorisation as in \eqref{bfac1}, with normalisation inherited from $\H$.
%is a symmetric Birkhoff factorisation (with normalisation inherited from $\H$).
%further, the latter transform performs a symmetric Birkhoff factorisation $R=F^*F$ (which is possible, as $R$ takes values $\he^+_n$) and maps $R\mapsto \H^+(R)=F\in \G$, with normalisation inherited from $\H$.

We introduce two further types of `nonlinear Hilbert transforms', which do not depend on a choice of $\1$, but are only available if $r$ admits a special decomposition. To this end, define spaces
\begin{equation}\label{zerospace}
C_0^\infty(\partial SM,GL(n,\C))\quad \text{ and } \quad C_0^\infty(\partial SM,\he^+_n),
\end{equation}
as follows: An element $r\in C_\extend^\infty(\partial SM,GL(n,\C))$ lies in the left space, if it admits a (necessarily unique) decomposition $r=uf$ as in \eqref{bfac2} with $f_0=\Id$ --  we then write $\H^0(r)=u^*$. Further, 
$r\in C^\infty(\partial SM,\he^+_n)$ lies in the right space in \eqref{zerospace}, if $r^{1/2}\in C_0^\infty(\partial SM,GL(n,\C))$ and we set $\H^{+,0}(r) =\H^0(r^{1/2})r^{1/2}$. We obtain transforms:
\begin{eqnarray}\label{hilbert0}
\H^0%,\H^{*,0}
&:&C_0^\infty(\partial SM,GL(n,\C))\rightarrow C_\extend^\infty(\partial SM,U(n)),\\
 \H^{+,0}&:&C_0^\infty(\partial SM,\he^+_n)\rightarrow C_\Id^\infty(\partial SM,GL(n,\C)). \label{hilbert+0}
\end{eqnarray}
Next, we consider the space $C_\Delta^\infty(\partial SM,GL(n,\C))$ consisting of those maps $r\in C_\Id^\infty(\partial SM,GL(n,\C))$ which factor uniquely as
$
r=gf$, where $f,g^*\in \HH$ (with $\HH$ as in Remark \ref{boundaryfac}) and $g_0=\Id$. 
With respect to this factorisation we define
\begin{equation}\label{hilbertjumping}
\H^\Delta: C_\Delta^\infty(\partial SM,GL(n,\C)) \rightarrow \G, \quad \H^\Delta(r)=f.
\end{equation}
We discuss the relation of $\H^\Delta$ to Birkhoff factorisations below Theorem \ref{rangeglpairs}.

\begin{example} We consider the `nonlinear Hilbert transforms' from above for $n=1$. A general element in $C_\extend^\infty(\partial SM,GL(1,\C))$ has the form 
\begin{equation}
r = e^{ik\theta} e^{\psi + i\sigma}, \text{ where } k\in \Z,\psi,\sigma\in C^\infty(\partial SM,\R).
\end{equation}
Let $\psi=\psi_{<0}+ \psi_0+\psi_{>0}$ be the decomposition into negative, zero and positive Fourier modes. Then the standard, linear Hilbert transform of $\psi$ is 
$
H\psi = (\psi_{>0}-\psi_{<0})/i
$
and thus
$
\psi = -iH\psi + \psi_0+2\psi_{>0},
$
which implies that
\begin{equation}
r = \left(e^{ik \theta} e^{-iH\psi + i\sigma} \right) \times \left(e^{\psi_0+2\psi_{>0}}\right) =: uf
\end{equation}
is a decomposition as in \eqref{bfac2} (not necessarily normalised). Then
\begin{equation}
\H(r) = w e^{iH\psi} e^{-i\sigma - ik\theta}\quad \text{ and } \quad \H^*(r) = w e^{-iH\psi}e^{i\sigma + ik\theta},
\end{equation}
where $w=w_\1\in C^\infty(M,U(n))$ is chosen to achieve the correct normalisation.
If $r$ takes values in $\R_{>0}\equiv \he^+_1$ (such that $k=0$ and $\sigma=0$) we see that $\H$ and $\H^*$ are exponentiated linear Hilbert transforms; further
\begin{equation}
\H^+(r) = w e^{\frac{1}{2}(\psi + i H\psi)}.
\end{equation}
Finally, $r$ is in the domain of $\H^0$ and $\H^{+,0}$ iff $\psi_0=0$ and it is in the domain of $\H^\Delta$ iff $k=0$, in which case $\H^\Delta(r)=e^{\psi_0+ \psi_{>0}+i\sigma_0 + i \sigma_{>0}}$.
\end{example}

%\begin{equation}
%A_+w(x,v) = \begin{cases}
%w(x,v)& (x,v)\in \partial_+SM\\
%w(\alpha(x,v))& (x,v)\in \partial_-SM.
%\end{cases}
%\end{equation}

\subsection{Range for $\u(n)$-attenuations} \label{rangeuattenuations}
It is instructive to first consider the non-Abelian X-ray transform on the space $\mho$ from \eqref{defmho}. In terms of the right action of $\G$ on $\mho$ (defined in \eqref{rule}), we have $\I^*(F)=0\act F^{-1}$ and hence $\I^*$ fits into an exact sequence (of pointed sets):
\begin{equation} \label{mses}
0\rightarrow \G_0 \hookrightarrow \G \xrightarrow{\I^*} \mho \rightarrow \M \rightarrow 0.
\end{equation}
Here $\G_0$ is the stabiliser of $0\in \mho$ and -- for the purpose of this section -- we think of $\M$ as quotient space $\mho/\G$, such that exactness in \eqref{mses} is evident. The identification $\mho/\G=\M$ is justified by  Proposition \ref{tcB}.

Let us assume now that $\M$ is trival -- by Theorem \ref{mainhif} this holds in particular if $(M,g)$ is simple.
 Then $\I^*:\G\rightarrow \mho$ is surjective and we have a commutative diagram
\begin{equation}\label{magicdiagram2}
\begin{tikzcd}
\G \arrow[two heads]{r}{\I^*} \arrow{dr}\arrow[swap, two heads]{d}{(\cdot)\vert_{\partial SM}} & \mho \arrow{d}{\A\mapsto C_\A} \\
\mathbb{H} \arrow{r}{B} & C^\infty(\partial_+SM,GL(n,\C)),
\end{tikzcd}
\end{equation}
where $\mathbb{H}=\{f = F\vert_{\partial SM}: F\in \G\}$ (see also Remark \ref{boundaryfac}). Note that $\mathbb{H}$ is a genuine boundary space in that it has the intrinsic characterisation
\begin{equation*}
\mathbb{H}=\{f\in C_\Id^\infty(\partial SM,GL(n,\C)):  f \text{ is fibrewise holomorphic} \}.
\end{equation*}
We thus obtain the following range characterisation:

\begin{proposition}
Suppose $\M=0$. Then an element $q\in C^\infty(\partial_+SM,GL(n,\C))$ lies in the range of $\mho\ni \A\mapsto C_\A$ if and only if $q = B f$ for some $f\in \mathbb{H}$.\qed
\end{proposition}

Let us now consider attenuations in one of the three classes
\begin{equation}\label{threeclasses}
\mho(\u(n))=\{\u(n)\text{-pairs}\},\quad C^\infty(M,\u(n))\quad \text{ and } \quad  \Omega^1(M,\u(n)),
\end{equation}
all considered as subsets of $\mho$.
Note that $\mho(\u(n)) = \mho \cap C^\infty(SM,\u(n))$, due to the identity $\A_{-k}=-\A_k^*$ ($k\in \Z$) for skew-Hermitian attenuations. The second and third space in \eqref{threeclasses} consist of skew-Hermitian matrix fields  $\Phi$  and connections $A$, respectively.

Define
\begin{equation}
\sqrt{\G_0} := \{F\in \G: X(F^*F)=0\}\subset C^\infty(SM,\C^{n\times n}),
\end{equation}
and recall that a function $F$ on $SM$ (or $\partial SM$) is even, if it only has even Fourier modes or equivalently if it obeys the symmetry condition $F(x,v)=F(x,-v)$.

\begin{proposition}\label{ionto} Suppose that $(M,g)$ is simple. Then $\I^*$ is well defined and surjective in the following settings:
\begin{enumerate}[label=(\roman*)]
\item \label{ionto1} $\I^*:\sqrt{\G_0} \rightarrow \mho(\u(n))$
\item \label{ionto3} $\I^*:\{F\in \sqrt{\G_0}: F \text{ {\rm even}}\} \rightarrow \Omega^1(M,\u(n))$
\item \label{ionto2} $\I^*:\{F\in \sqrt{\G_0}: F_0=\Id\} \rightarrow C^\infty(M,\u(n))$
\end{enumerate}
\end{proposition}

Part \ref{ionto1} holds in greater generality and is in fact equivalent to the assertion that $\M=0$ (assuming that $(M,g)$ is non-trapping and has a strictly convex boundary).

\begin{proof}
For \ref{ionto1} let $F\in \sqrt{\G_0}$, then 
\begin{equation}
0 = X(F^*F)=(X F^*) F + F^* (XF) = F^*\left((F^{-1})^* (XF^*) - \I^* (F) \right) F
\end{equation}
and hence 
$
\left( \I^* (F)\right)^* = \left( (F^{-1})^* (XF^*) \right)^* = -\I^*(F)
$, which shows that $\I^*$ indeed maps $\sqrt{\G_0}$ into $\mho(\u(n))$. To see that it is onto, let $\A\in \mho(\u(n))$ and let $F\in \G$ be an arbitrary HIF for $\A$. Then, as $\A$ is skew-Hermitian,
\begin{equation}
X (F^*F) = (-\A F)^* F + F^* XF= F^* (\A F+XF)=0
\end{equation}
which means that $F\in \sqrt{\G_0}$.

For \ref{ionto3} let $F\in \sqrt{\G_0}$ be even. Then $\I^*(F)$ is odd and, being skew-Hermitian, only has Fourier modes in degree $\pm 1$; hence $\I^*(F)\in \Omega^1(M,\u(n))$. Conversely, if $A$ is a $\u(n)$-connection, then by Proposition \ref{hifeven} there exists an even HIF $F\in \G$ and we must have $F\in \sqrt{\G_0}$, as above.

Finally, for \ref{ionto2}, let $F\in \sqrt{\G_0}$ with $F_0=\Id$. Then $\A = \I^*(F)$ satisfies
\begin{equation}\label{ionto4}
\A_{-1} = - (\eta_-F_0) (F^{-1})_0 = 0 
\end{equation}
and, $\A$ being skew-Hermitian, also $\A_1 = - \A_{-1}^*=0$. Hence $\A\in \Omega_0(SM,\u(n))\equiv C^\infty(M,\u(n))$. Conversely, given $\Phi \in C^\infty(M,\u(n))$ let $\tilde F\in \sqrt{\G_0}$ be any HIF. Then, similar to \eqref{ionto4} we see that $\eta_- \tilde F_0 = 0$ and as $\tilde F_0$ is invertible, also $\eta_-  \tilde F_0^{-1} = 0$. By Theorem 13.11.6 in \cite{PSU20} there exists $G\in \G_0$ with $G_0=\tilde F_0^{-1}$, hence $F=\tilde F G \in \sqrt{\G_0}$ is a new HIF for $\Phi$ with $F_0 = \tilde F_0 G_0 = \Id$, as desired.
\end{proof}

The domains in the preceding proposition can be projected onto the following boundary spaces of $\he^+_n$-valued functions:
\begin{align}
 C_\alpha^\infty(\partial_+SM,\he^+_n)&= \Big \{w\in C^\infty(\partial_+SM,\he^+_n):
 \begin{array}{l}
 A_+w \text{ smooth on } \partial SM
\end{array}
\Big \} \label{bspace1} \\
 C_{\alpha,1}^\infty(\partial_+SM,\he^+_n) &= \Big\{w\in C_\alpha^\infty(\partial_+SM,\he^+_n): ~\,w\circ\alpha_a = w \Big\} \label{bspace3}\\
 C_{\alpha,0}^\infty(\partial_+SM,\he^+_n)&=\left\{ w\in C_\alpha^\infty(\partial_+SM,\he^+_n):
 \begin{array}{l}
 w^\sharp =F^*F \text{ for some } \\  F\in \G \text{ with } F_0 =\Id 
 \end{array}
 \right\}
 %\{w = F^*F\vert_{\partial_+SM}:  F\in \sqrt{\G_0}\} 
\label{bspace2}
\end{align}
Here
\begin{equation}\label{someas}
%A_+:C(\partial_+SM)\rightarrow C(\partial SM),\quad 
A_+w(x,v) =\begin{cases}
w(x,v)& (x,v)\in \partial_+SM\\
w\circ \alpha(x,v) & (x,v) \in \partial_-SM
\end{cases}
\quad \text{ and } \quad \alpha_a(x,v) = \alpha(x,-v),
\end{equation}
and $w^\sharp$ is the unique solution to $Xw^\sharp =0$ on $SM$ and $w^\sharp = w$ on $\partial_+ SM$. By a classical result of Pestov and Uhlmann (see \cite{PeUh04}  or Theorem 5.1.1 in \cite{PSU20}), the first integral $w^\sharp$ is smooth on $SM$ for $w\in C_\alpha^\infty(SM,\he_n^+)$.

Note that \eqref{bspace1} and \eqref{bspace3} define genuine boundary spaces in the sense that membership can be checked on $\partial SM$ only in terms of the scattering relation $\alpha$. To check whether a function $w$ belongs to $C_{\alpha,0}^\infty$ one first has to find the first integral $w^\sharp$. 

\smallskip

The following result is a consequence of Birkhoff's factorisation theorem for Hermitian matrices (Theorem \ref{thmfac}) and does {\it not} require $(M,g)$ to be simple.

\begin{proposition} \label{sigma} Let $\sigma(F)= F^*F\vert_{\partial_+SM}$, then: %is well defined and surjective in the following settings:
\begin{enumerate}[label=(\roman*)]
\item  \label{sigma1} $\sigma:\sqrt{\G_0} \rightarrow C_\alpha^\infty(\partial_+SM,\he^+_n)$ is surjective% and injective up to left multiplcation by $C^\infty(M,U(n))$.
\item \label{sigma3} $\sigma:\{F\in \sqrt{\G_0}: F~{\rm even}\} \rightarrow C_{\alpha,1}^\infty(\partial_+SM,\he^+_n)$ is surjective
\item \label{sigma2} $\sigma:\{F\in \sqrt{\G_0}: F_0=\Id\}\rightarrow C_{\alpha,0}^\infty(\partial_+SM,\he^+_n)$ is bijective
\end{enumerate}
\end{proposition}

\begin{proof}
For \ref{sigma1} let $w\in C_\alpha^\infty(\partial_+SM,\he^+_n)$,  then $w^\sharp\in C^\infty(SM,\he^+_n)$ (it takes values in $\he^+_n$, as it is constant along the geodesic flow)
and by Theorem \ref{thmfac} there exists $F\in \G$ with $w^\sharp =F^*F$. We then automatically have $F\in \sqrt{\G_0}$. 

For \ref{sigma3} note that if $F\in \sqrt{\G_0}$ is even, then $F^*F\vert_{\partial SM}$ is $\alpha$-invariant and thus  $w=\sigma(F)$ satisfies
\begin{equation}
w(\alpha(x,-v)) = F^*F(x,-v)  = F^*F(x,v)  = w(x,v).
\end{equation}
Conversely if $w\in C_\alpha^\infty(\partial_+SM,\he^+_n)$ satisfies $w(\alpha(x,-v))=w(x,v)$, then $w^\sharp$ is even by Lemma 9.4.9 in \cite{PSU20}. Using Theorem \ref{thmfac}, this factors as $w^\sharp =F^*F$ for an even $F\in\G$. As above, we must have $F\in \sqrt{\G_0}$ and the proof is complete.

For \ref{sigma2}, surjectivity is clear. Here $\sigma$ is also injective, because $w^\sharp = F^*F = \tilde F^*\tilde F$ for two different $F,\tilde F\in \sqrt{\G_0}$ only if $F =  u \tilde F $ for $u\in C^\infty(M,U(n))$, while the requirement $F_0=\tilde F_0=\Id$ implies $u=\Id$ and thus $F=\tilde F$.
\end{proof}

By the preceding propositions, and if $(M,g)$ is simple, we obtain three diagrams which are analoguous to \eqref{magicdiagram} and \eqref{magicdiagram2}. The first of these is
\begin{equation}\label{magicdiagram3}
\begin{tikzcd}
\sqrt{\G_0} \arrow[two heads]{r}{\I^*} \arrow{dr}{B(\cdot\vert_{\partial SM})}\arrow[swap, two heads]{d}{\sigma} & \mho(\u(n)) \arrow{d}{\A\mapsto C_\A} \\
C_\alpha^\infty(\partial_+SM,\he^+_n) \arrow[dashed]{r}{P} & C^\infty(\partial_+SM,U(n)),
\end{tikzcd}
\end{equation}
with boundary operator $P$ yet to be defined. Here the upper triangle commutes as it arises from restricting \eqref{magicdiagram}; moreover, any choice of $P$ making the lower triangle commute would have the same range as $\mho(\u(n))\ni \A\mapsto C_\A$. However, commutativity can only be achieved up to unitary gauge, as we now explain. 

\smallskip

 Note that $C^\infty(M,U(n))$ acts on $\sqrt{\G_0}$ by left-multiplication and $C_\Id^\infty(\partial M,U(n))$ (subscript $\Id$ means: homotopic to $\Id$) acts on $C^\infty(\partial_+SM,GL(n,\C))$ by 
\begin{equation}\label{uaction}
h\lact q = h q (h^{-1}\circ\alpha).
\end{equation}
Then the diagonal arrow in \eqref{magicdiagram3} is equivariant with respect to these group actions, while $\sigma$ is invariant and in fact injective up the action of $C^\infty(M,U(n))$, i.e.
$
\sigma(F) = \sigma(F')$ if and only if $F'=UF$ for some $U\in C^\infty(M,U(n))$ (see Theorem \ref{thmfac}\ref{fac1}).

In terms of $\H^+$ from \eqref{hilbert+} we now define a boundary operator $P$ by
\begin{equation}\label{defP}
Pw := B\mathcal{H}^+A_+w,\quad w\in C_\alpha^\infty(\partial_+SM,\he^+_n).
\end{equation}

\begin{lemma}\label{gaugetriangle} With $P$ as in \eqref{defP} the lower triangle in \eqref{magicdiagram3} commutes up to gauge. That is, for any $F\in \sqrt{\G_0}$ there exists $h\in C_\Id^\infty(\partial M,U(n))$ such that
$
B(F\vert_{\partial SM}) = h\lact P(\sigma(F))).
$
\end{lemma}

\begin{proof}
Let $F\in \sqrt{\G_0}$ and put $w=\sigma(F)$. Then, by definition of $\H^+$, we have  $A_+w = f^* f$, where $f=\H^+(A_+w)$. On the other hand, $A_+w = F^*F\vert_{\partial SM}$ is another such decomposition and thus $f=hF$ for some $h\in C_\Id^\infty(\partial M,U(n))$. Hence $Pw=B(hF\vert_{
\partial SM})=h \lact B(F\vert_{\partial SM})$, as desired.
\end{proof}

\begin{theorem}[Range for $\u(n)$-pairs]\label{rangepairs}
Suppose that $(M,g)$ is simple (or more generally, that $\M=0$). Then an element $q\in C^\infty(\partial_+SM,U(n))$ lies in the range of $\mho(\u(n))\ni (A,\Phi)\rightarrow C_{A,\Phi}$ if and only if 
\begin{equation}\label{rangeua}
\hspace{3em} q=h\lact Pw, \quad \text{for some } (w,h)\in C_\alpha^\infty(\partial_+SM,\he^+_n) \times C^\infty_\Id(\partial M,U(n)).
\end{equation}
\end{theorem}

\begin{proof}
The proof is essentially a diagram chase in \eqref{magicdiagram3}. First suppose that $q=C_\A$ for some $\A\in\mho(\u(n))$. By Proposition \ref{ionto} (valid also if $\M=0$) there exists  $F\in \sqrt{\G_0}$ with  $\I^*(F)=\A$ and consequently, using Lemma \ref{gaugetriangle}, we have $q = B(F\vert_{\partial SM})= h\lact P(w)$ for $w= \sigma(F)$ and some $h\in C^\infty_\Id(\partial M,U(n))$. For the other direction suppose that $q=h\lact P(w)$ for $(w,h)$ as in \eqref{rangeua}. By Proposition \ref{sigma} we have $w=\sigma (F)$ for some $F\in \sqrt{\G_0}$ and by Lemma \ref{gaugetriangle}, we have $B(F\vert_{\partial SM}) = h_1\lact Pw$ for some $h_1\in C^\infty_\Id(\partial M,U(n))$. We may extend both $h$ and $h_1$ to functions in $C^\infty(M,U(n))$ (denoted by the same symbol) and set $
\A =  \I^*(hh_1^{-1} F),
$
such that
\begin{equation}
C_\A = (hh_1^{-1})\lact B(F\vert_{\partial SM}) = h\lact Pw= q.
\end{equation}
This completes the proof.
\end{proof}

\begin{remark}\label{equivalenceoftheorems}
In fact, on simple surfaces the range characterisation in the preceding theorem is {\it equivalent} to the assertion that $\M=0$ in the following sense: If the scattering data of a  $\u(n)$-pair $\A=(A,\Phi)$ is of the form \eqref{rangeua}, then $\A$ automatically admits holomorphic integrating factors.  To see this, let $(w,h)$ be as in \eqref{rangeua}, extend $h$ to a function in $C^\infty(M,U(n))$ and factor $w^\sharp = G^*G$ for some $G\in \G$. Then also $F:=hG$ lies in $\G$ and $\B:=-(XF)F^{-1}$ is a $\u(n)$-pair with the same scattering data as $\A$. By Theorem \ref{PSU}, the attenuation $\A$ is gauge equivalent to $\B$ and thus it admits holomorphic integrating factors.
Granted a characterisation as in the theorem, one thus obtains HIF's for $\u(n)$-pairs and as every orbit in $\M\equiv \mho/\G$ contains such a pair \cite[Lemma 5.2]{PaSa20}, it holds that $\M=0$.
\end{remark}

Similarly, with $\alpha_a$ as in \eqref{someas}, we obtain:

\begin{theorem}[Range for $\u(n)$-connections]\label{rangeconnection}
Suppose that $(M,g)$ is simple. Then an element $q\in C^\infty(\partial_+SM,U(n))$ lies in the range of $\Omega^1(M,\u(n))\ni A \mapsto C_A$ if and only if one of the following equivalent conditions is satisfied:
\begin{enumerate}[label=(\roman*)]
\item \label{rangeconn1} $q$ satisfies \eqref{rangeua} and additionally $q\circ\alpha_a=q^{-1}$. 
\item \label{rangeconn2} $q$ satisfies \eqref{rangeua} with $w\in C_{\alpha,1}^\infty(\partial_+SM,\he^+_n)$ (i.e.~$w\circ\alpha_a=w$).
\end{enumerate}
\end{theorem}

\begin{proof}
To prove characterisation \ref{rangeconn1}, we first consider an arbitrary attenuation $\A \in C^\infty(SM,\C^{n\times n})$ with integrating factor $R\in C^\infty(SM,\C^{n \times n})$. Then  $S(x,v) = R(x,-v)$ defines an integrating factor for $\B(x,v)=-\A(x,-v)$ and we have
\begin{equation}
C_\A\circ \alpha_a = (R\circ \alpha_a) (R^{-1}\circ\alpha\circ\alpha_a) = \left[S(S^{-1} \circ \alpha) \right]^{-1} = C_\B^{-1}\quad \text{ on } \partial_+SM.
\end{equation}
Now, if $q=C_\A$ for $\A$ equal to a connection $A\in \Omega^1(M,\u(n))$, then also $\B=A$ and thus $q$ has the desired symmetry. Conversely, if $q=C_\A$  for a $\u(n)$-pair $\A=(A,\Phi)$ and additionally  $q\circ\alpha_a = q^{-1}$, then the previous display implies $C_\A = C_\B$ and by  Theorem \ref{PSU} there is a 
gauge $\varphi \in C^\infty(M,U(n))$ with $\varphi=\Id$
on $\partial M$ such that
\begin{equation}
\varphi \Phi + \Phi\varphi =0\quad \text{and} \quad \d \varphi + [A,\varphi] =0 \quad \text{ on } M.
\end{equation}
By the second equation $\varphi$ solves an ODE along every curve in $M$ and thus it is determined by its boundary values. It follows that $\varphi \equiv \Id$ and hence $\Phi=0$.

The characterisation in \ref{rangeconn2} follows by the same arguments that lead to Theorem \ref{rangepairs}, replacing diagram \eqref{magicdiagram3} with the obvious analogue containing the spaces $\{F\in\sqrt{\G_0}: F\text{ even}\}$ and $C^\infty_{\alpha,1}(\partial_+SM,\he^+_n)$. This completes the proof.
\end{proof}

Next, we consider the range of  $\Phi\mapsto C_\Phi$ for $\u(n)$-valued matrix fields. In this case one defines a boundary operator in terms of the transform 
$\H^{+,0}$ from \eqref{hilbert+0}:
\begin{equation}
P_0 w : = B\H^{+,0} A_+ w,\quad w\in C_{\alpha,0}^\infty(\partial_+SM,\he^+_n).
\end{equation}
Using Propositions \ref{ionto} and \ref{sigma}, one obtains a similar diagram as in \eqref{magicdiagram3}, this time commutative, as $\sigma$ is bijective in this setting. We obtain the following result, omitting the proof:

\begin{theorem}[Range for $\u(n)$-matrix fields] \label{rangefields} Suppose that $(M,g)$ is simple. Then an element $q\in C^\infty(\partial_+SM,U(n))$ lies in the range of $C^\infty(M,\u(n))\ni \Phi \mapsto C_\Phi$ if and only if
$
q = P_0w$ for some  $w\in C_{\alpha,0}^\infty(\partial_+SM,\he^+_n)$. \qed
\end{theorem}

\begin{remark} \label{alternative1} An alternative characterisation for $\u(n)$-pairs is obtained as follows: Define
$
\G_U = \{(U,F)\in C^\infty(SM,U(n))\times \G: X(UF)=0\}
$ and consider the diagram
\begin{equation}\label{magicdiagram4}
\begin{tikzcd}
\G_U\arrow[two heads]{r}{(U,F)\mapsto \I^*(F)} \arrow{dr}{(U,F)\mapsto B(F\vert_{\partial SM})}\arrow[swap, two heads]{d}{(U,F)\mapsto UF\vert_{\partial_+SM}} & \mho(\u(n)) \arrow{d}{\A\mapsto C_\A} \\
C_\alpha^\infty(\partial_+SM,GL(n,\C)) \arrow{r}{B\H A_+ } & C^\infty(\partial_+SM,U(n)),
\end{tikzcd}
\end{equation}
where $\H$ is as in \eqref{hilbert1}.
As above, one proves commutativity up to gauge and -- assuming that $(M,g)$ is simple -- one establishes surjectivity results as indicated by the double headed arrows. This shows that condition \eqref{rangeua} in Theorem \ref{rangepairs} can be replaced by
\begin{equation*}
\hspace{-1em} q=h\lact (B\H A_+w) \quad \text{for some } (w,h)\in C_\alpha^\infty(\partial_+SM,GL(n,\C)) \times C^\infty_\Id(\partial M,U(n)).
\end{equation*}
In order to isolate connections and matrix fields one can use the same ideas that lead to Theorems \ref{rangeconnection} and \ref{rangefields}; we leave the details to the reader.
\end{remark}

\subsection{Range for $\gl(n,\C)$-attenuations}
 To obtain range characterisations in the  $\gl(n,\C)$-case, define $\G_G=\{(G,F)\in \G^*\times \G: G_0=\Id\}$ and consider the diagram 
\begin{equation}\label{magicdiagram5}
\begin{tikzcd}
\G_G\arrow[two heads]{r}{(G,F)\mapsto \I^*(F)} \arrow{dr}\arrow[swap, two heads]{d}{(G,F)\mapsto GF\vert_{\partial_+SM}} & \{\gl(n,\C)\text{-pairs}\} \arrow{d}{\A\mapsto C_\A} \\
C_{\alpha,\Delta}^\infty \arrow{r}{P^\Delta} & C^\infty(\partial_+SM,GL(n)),
\end{tikzcd}
\end{equation}
where $P^\Delta = B\H^\Delta A_+$ with transform $\H^\Delta$ as in \eqref{hilbertjumping}
and with diagonal arrow equal to $(G,F)\mapsto B(F\vert_{\partial SM})$. Here $C_{\alpha,\Delta}^\infty\subset C_{\alpha}^\infty(\partial_+SM,GL(n,\C))$ is defined -- somewhat tautologically -- as range of the left vertical map and commutativity holds up to a $GL(n,\C)$-valued gauge. If $(M,g)$ is simple, the top arrow is surjective and similar to Theorem \ref{rangepairs}, the following result holds true (the proof is omitted):
\begin{theorem}[Range for $\gl(n,\C)$-pairs]\label{rangeglpairs} Suppose that $(M,g)$ is simple (or more generally, that $\M =0$). Then an element $q\in C^\infty(\partial_+SM,GL(n,\C))$ lies in the range of $\{\gl(n,\C)\text{-pairs}\}\ni (A,\Phi)\mapsto C_{A,\Phi}$ iff
\begin{equation*} \pushQED{\qed} 
q=h\lact P^\Delta w\quad
\text{ for some } (w,h)\in C_{\alpha,\Delta}^\infty \times C^\infty_\Id(\partial M,GL(n,\C)). \qedhere \popQED
\end{equation*}
\end{theorem}
Similar to above one can isolate connections and matrix fields; we omit the details.

\smallskip

The space $C_{\alpha,\Delta}^\infty$ can also be described as follows: Any $w\in C^\infty_\alpha(\partial SM,GL(n,\C))$ extends to a first integral $w^\sharp:SM\rightarrow GL(n,\C)$ which, by virtue of \cite[Theorem 8.1.2]{PS86},  admits a Birkhoff factorisation  as
\begin{equation}
w^\sharp(x,\cdot) = G(x,\cdot)\Delta(x,\cdot)F(x,\cdot).
\end{equation}
Here $F(x,\cdot)$ and $G(x,\cdot)^*$ are fibrewise holomorphic and $\Delta(x,\theta)=\mathrm{diag}\left(e^{ia_1(x)\theta}, \dots,\right.$ $\left. e^{ia_n(x)\theta} \right)$ for not necessarily continuous maps $a_i:M\rightarrow \Z$ ($i=1,\dots,n$). We then have $w\in C_{\alpha,\Delta}^\infty$ if and only if $\Delta \equiv \Id$, in which case $F$ and $G$ are automatically smooth on $SM$. To check membership of $w$ in $C_{\alpha,\Delta}^\infty$ it is thus necessary that $\Delta\vert_{\partial_+SM}\equiv \Id$; it is an interesting question whether this also sufficient or more generally, whether discontinuities of $\Delta$ -- also called {\it jumping lines} -- can be detected at the boundary.

\subsection{Nontrivial $\mathfrak{M}$} Finally, we give a range characterisation if $\M\neq 0$. In this case the range of $\{\u(n)\text{-pairs}\} \ni (A,\Phi) \mapsto C_{A,\Phi}$ is parametrised in terms of {\it both} solitonic and radiative/dispersive degrees of freedom -- as discussed below Theorem \ref{mainrange} -- in the following sense: We construct 
 a `boundary space' $\BB(\partial_+SM)$ which fits into a short exact sequence (of pointed sets)
 %\footnote{We say that a sequence of pointed sets $0\rightarrow (X,x_0)\xrightarrow {f} (Y,y_0) \xrightarrow{g} (Z,z_0)\rightarrow 0$ is exact, if $f$ is injective, $g$ is surjective and $f(X)=g^{-1}(\{z_0\})$. The base points in \eqref{ses} are $\Id$, $[(0,\Id)]$ and $[0]$, respectively.}
\begin{equation}\label{ses}
0\rightarrow C_\alpha^\infty(\partial_+SM,GL(n,\C)) \rightarrow \BB(\partial_+SM) \rightarrow \M \rightarrow 0,
\end{equation}
and is the natural domain of a `boundary operator'
\begin{equation}
\mathcal{P}:\BB(\partial_+SM) \rightarrow  \U\backslash C^\infty(\partial_+ SM,U(n)).
\end{equation}
Here the right hand side denotes the quotient by $\U:=C_\Id^\infty(\partial M,U(n))$ under the action defined in \eqref{uaction}. 
\smallskip

In order to define the boundary space $\BB(\partial_+SM)$, denote
$\S_\A^\infty(\partial_+SM,GL(n,\C))$ the space  of functions $w:\partial_+SM\rightarrow GL(n,
C)$ for which the extension
\begin{equation}
E_\A w(x,v) = \begin{cases}
w(x,v) & (x,v)\in \partial_+SM\\
C_\A^{-1} w(\alpha(x,v)) & (x,v)\in \partial_-SM
\end{cases}
\end{equation}
defines a smooth function on $\partial SM$. Next,  denote with $\S(\partial_+SM)$ the subset of $\mho \times C^\infty(\partial_+SM,GL(n,\C))$ consisting of pairs $(\A,w)$ with  $w\in \S_\A^\infty(\partial_+SM,\C^{n\times n})$. Then $\G$ acts on $\S(\partial_+SM)$ via
$
(\A,w)\act F=(\A\act F,(F^{-1}\vert_{\partial_+SM})w)
$
%To see that this is well defined, note that if $w\in \S_\A^\infty(\partial_+SM,\C^{n\times n})$, then $w=R\vert_{\partial_+SM}$ for some solution $R\in C^\infty(SM,\C^{n\times n})$ to $(X+\A)R=0$. Then
%$
%X(F^{-1} R) = - F^{-1} (X F) R - F^{-1} \A R = - (\A.F) (F^{-1}R)
%$ 
%and one checks that $F^{-1} R\vert_{\partial_+SM}$ lies in $\S^\infty_{\A.F}(\partial_+SM,\C^{n\times n})$. 
and we define
\begin{equation}
\BB(\partial_+SM):=\S(\partial_+SM)/\G.
\end{equation}
The arrows in \eqref{ses} are given by $w\mapsto [(0,w)]$ and $[(\A,w)]\mapsto [\A]$ respectively (exactness is obvious). 
The boundary operator $P$ is defined as concatenation
\begin{equation*}
\hspace{-.5em} \BB(\partial_+SM) \rightarrow \G \backslash C^\infty(\partial SM,GL(n,\C)) \xrightarrow{\H^*} \U\backslash C^\infty(\partial S M,U(n)) \xrightarrow{B} \U\backslash C^\infty(\partial_+SM,U(n)),
\end{equation*}
where the first arrow is $[(\A,w)]\mapsto [E_\A w]$ and 
$\H^*$ is the transform defined in \eqref{hilbert+}. Both $\H^*$ and $B$ are easily seen to descend to quotient spaces as indicated and we keep denoting them by the same symbols.

\begin{theorem} An element $q \in C^\infty(\partial_+SM,U(n))$ lies in the range of $\{\u(n)\text{-pairs}\} \ni (A,\Phi) \mapsto C_{A,\Phi}$ if and only if
\begin{equation}
[q] = \mathcal{P}(\mathfrak{b}) \quad \text{ for some } \mathfrak{b}\in \BB(\partial_+SM).
\end{equation}
\end{theorem}

\begin{proof} First assume that $q=C_\A$ for a $\u(n)$-pair $\A=(A,\Phi)$. Let $U\in C^\infty(SM,U(n))$ be a solution to $(X+\A)U=0$ on $SM$. Setting $w=U\vert_{\partial_+SM}$, we have $\mathfrak{b}:=[(\A,w)]\in \BB(\partial_+SM)$ and  $E_\A w = U\vert_{\partial SM}$, which means that $\mathcal{P}(\mathfrak{b})=[B(U\vert_{\partial SM})]=[C_\A]=[q]$.\\ 
Conversely, suppose that $[q]=P(\mathfrak{b})$ for some $\mathfrak{b}=[(\A,w)]\in \BB(\partial_+SM)$. 
Then $w=R\vert_{\partial_+SM}$ for a solution $R\in C^\infty(SM,GL(n,\C))$ to $(X+\A)R=0$. By Theorem \ref{thmfac} this may be factored as $R=FU$ for $F\in \G$ and $U\in C^\infty(SM,U(n))$. Then
\begin{equation}
\mathcal{P}(\mathfrak{b})=[B(U\vert_{\partial SM})] = [C_{\A\act F}],
\end{equation}
and by Lemma 5.2 in \cite{PaSa20}, the attenuation $\A\act F$ is given by a $\u(n)$-pair, as desired.
\end{proof}

\section{Appendix}\label{appendix}

\subsection{A tame setting}\label{tamesetting}

We discuss the Fr{\'e}chet structure and tameness of the spaces, Lie groups and actions used in the proofs of Theorem \ref{mainhif} and Proposition \ref{hifeven}. Throughout $(M,g)$ is a compact, oriented Riemannian surface with smooth and possibly empty boundary $\partial M$.

\smallskip

First recall that $C^\infty(SM,\C^n)$ has a standard Fr{\'e}chet topology, which can be generated by norms $\Vert \cdot \Vert_{H^s}$ of the Sobolev-spaces $H^s(SM,\C^n)$ ($s\in \R$). We view $C^\infty(SM,\C^n)$ as graded Fr{\'e}chet space, with grading given by
\begin{equation}\label{sobolevgrading}
\Vert \cdot \Vert_{L^2} = \Vert \cdot \Vert_{H^0} \le \Vert \cdot \Vert_{H^1} \le \dots.
\end{equation}
Note that while there are several ways to define these norms,  a different choice will result in a tamely equivalent grading. As in \textsection \ref{prelim} we will tacitly apply the considerations in this section also to $\gl(n,\C)$-valued functions.
%The spaces $C^\infty(SM,\C^n)$ and $C^\infty(SM,\C^{n\times n})$ are graded analogously. All tameness statements refer to the grading given in \eqref{sobolevgrading}.

\subsubsection{Tame spaces} The space $\oplus_{k\in I}\Omega_k$ ($I\subset \Z$) from \eqref{defOmega} lies closed in the ambient $C^\infty$-space and thus inherits a Fr{\'e}chet topology and a grading. The next lemma implies that $\oplus_{k\in I}\Omega_k$ is a tame direct summand and as $C^\infty(SM,\C^n)$ is tame \cite[Corollary~1.3.7,\textsection II]{Ham82} the space $\oplus_{k\in I} \Omega_k$ must be tame itself  \cite[Lemma~1.3.3,\textsection II]{Ham82}.

In particular, both $\mho$ from \eqref{defmho} and $\mho_{\text{odd}}$ from Proposition \ref{hifeven} are tame Fr{\'e}chet spaces.

\begin{lemma}\label{tameproj} For all $I\subset \Z$,  the $L^2$-orthogonal projection
$P_I:C^\infty(SM,\C^n)\rightarrow \oplus_{k\in I} \Omega_k$ satisfies the tame estimate
\begin{equation} 
\Vert P_I u \Vert_{H^s} \lesssim \Vert u \Vert_{H^s},\quad u\in C^\infty(SM,\C^n), s\ge 0,
\end{equation}
where $\lesssim$ means up to a constant that may depend on $I$ and $s$.
\end{lemma}

\begin{proof}
First note that $P_I$ extends to a bounded map $L^2(SM,\C^n)\rightarrow L^2(SM,\C^n)$, simply because it is a projection. This gives a tame estimate for $s=0$ and we may proceed by induction. To this end note the Sobolev scale on $SM$ is generated by the operators $\eta_\pm$ from \eqref{defeta} together with the vertical derivative $V$ in the sense that 
\begin{equation}
\Vert \cdot \Vert_{H^{s+1}} \approx \Vert \eta_+ \cdot \Vert_{H^s} + \Vert \eta_- \cdot \Vert_{H^s} + \Vert V \cdot \Vert_{H^s}+ \Vert \cdot \Vert_{H^s},\quad s\ge 0
\end{equation}
is an equivalence of norms. Let $I^\pm=I\pm1\subset \Z$, then $\eta_\pm P_{I}=P_{I^\pm}\eta_\pm$ and $[V,P_I]=0$, which means that for all $u\in C^\infty(SM,\C^n)$ 
\begin{equation}
\begin{split}
\Vert P_I u \Vert_{H^{s+1}} &\lesssim \Vert \eta_+ P_I u \Vert_{H^s} + \Vert \eta_- P_I u \Vert_{H^s} + \Vert V P_I u\Vert_{H^s} + \Vert P_I u \Vert_{H^s} \\
& = \Vert  P_{I^+} (\eta_+ u) \Vert_{H^s} + \Vert P_{I^-} (\eta_-  u) \Vert_{H^s} + \Vert  P_I  (Vu)\Vert_{H^s}+ \Vert P_I u \Vert_{H^s} \\
&\lesssim \Vert \eta_+ u \Vert_{H^s} + \Vert \eta_- u \Vert_{H^s} + \Vert V u \Vert_{H^s} + \Vert  u \Vert_{H^s} 
\lesssim \Vert u \Vert_{H^{s+1}},
\end{split}
\end{equation}
where we have used the induction hypothesis.
\end{proof}

\begin{comment}
Now consider $I=\{m\}$ ($m\in \Z$) with associated projection $P_mu=u_m$. Then for $s\ge 0$ and $u\in C^\infty(SM)$ it holds that
\begin{equation}
\begin{split}
\Vert P_m u \Vert_{H^{s+1}} &\lesssim \Vert \eta_+ u_m \Vert_{H^s} +  \Vert \eta_- u_m \Vert_{H^s} + \Vert V u_m \Vert_{H^s}\\
&= \Vert P_{m+1}(\eta_+ u) \Vert_{H^s} +  \Vert P_{m-1}(\eta_- u) \Vert_{H^s} +m \Vert  P_mu \Vert_{H^s}\\
& \lesssim \Vert \eta_+ u \Vert_{H^s} + \Vert \eta_- u \Vert_{H^s} + \Vert  u \Vert_{H^s} \
\lesssim \Vert u \Vert_{H^{s+1}},
\end{split}
\end{equation}
where we have used the induction hypothesis.  This proves the result for projections onto a single Fourier mode.
Next consider $I=\{k\in \Z: k\ge m\}$ for some $m\in \Z$ with associated projection $P_{\ge m}$ and note that the following commutator formulas hold true:
\begin{equation}
[\eta_+,P_{\ge m}]=- \eta_+ P_{m-1},\quad [\eta_-, P_{\ge m}] =  \eta_- P_m\quad \text{ and } \quad [V,P_{\ge m}]=0
\end{equation}
In particular, the induction step can be performed by using tame estimates for $P_{m-1}$ and $P_m$, which are already proved. This proves the result for the operators $P_{\ge m}$ $(m\in \Z)$. The general case poses only combinatorial difficulties and we omit the proof, as only the operators $P_m$ and $P_{\ge m}$ are used below.
\end{comment}

\subsubsection{Tame Lie groups} The group $\hat \G= C^\infty(SM,GL(n,\C))$ lies open in the ambient space $C^\infty(SM,\C^{n\times n})$ and thus is a Fr{\'e}chet manifold. We claim that $\hat \G$ is a tame Lie Group, which means that additionally the maps 
\begin{equation}\label{tameliegroup}
\mathfrak{m}:\hat\G\times \hat \G\rightarrow \hat\G\quad \text{ and } \quad \mathfrak{i}:\hat \G\rightarrow \hat\G,
\end{equation}
given by multiplication and taking inverses, respectively, are {\it smooth tame}. Similarly, the subgroups $\G\subset \hat \G$ from \eqref{defG} and $\G_\even\subset \G$ from Theorem \ref{hifeven} are tame Lie groups.

To prove tameness of $\mathfrak{m}$ and $\mathfrak{i}$ one may invoke the high-level Theorem 2.2.6 in \cite[\textsection II]{Ham82}, which states that so called `nonlinear vector bundle operators' are tame.  %Indeed, we may recast $\mathfrak{m}$ and $\mathfrak{i}$ as such operators: 
To this end let $E$ and $F$ be trivial vector bundles over $SM$, with fibres given by $\C^{n\times n}$ and $\C^{n\times n}\times \C^{n\times n}$ respectively.  Next, let $U\subset E$ and $V\subset F$ be the open subsets consisting of tuples $(x,v,A)$ and $(x,v,A,B)$ respectively, where $(x,v)\in SM$ and $A,B\in GL(n,\C)$. Then
\begin{equation}
\begin{split}
p:V\rightarrow E,\quad &(x,v,A,B)\mapsto (x,v,AB)\\
q:U\rightarrow E,\quad &(x,v,A)\mapsto (x,v,A^{-1})
\end{split}
\end{equation}
are `nonlinear vector bundle maps' in Hamilton's sense.  Let $\mathcal{V}\subset C^\infty(SM,F)$ be the set of sections $f$ with image in $V$ and denote $Pf = p\circ f$, then the just cited theorem implies that $P:\mathcal{V}\rightarrow C^\infty(SM,E)$ is a tame map.  Similarly, $q$ gives rise to a tame map $Q:\mathcal{U}\rightarrow C^\infty(SM,E)$. Identifying $\mathcal{U}$ with $\hat \G$ and $\mathcal{V}$ with $\hat \G\times \hat \G$,  we see that $P$ and $Q$ correspond precisely to ${\mathfrak{m}}$ and ${\mathfrak{i}}$, such that we have established tameness of multiplication and inversion on $\hat \G$.

For $  {\mathfrak{m}}$ and ${\mathfrak{i}}$ to be smooth tame it is required that they be smooth (which is clear) and that all derivatives are tame. However, this is a consequence of the already obtained tameness, for all derivatives are again given in terms of multiplication and inversion.
\subsubsection{Tame actions}  Finally, the  various Lie group actions defined in the paper are smooth tame. As each of the actions is given in terms of multiplication, inversion and taking adjoints, this can be proved similar to above, by recasting the action map as nonlinear partial differential operator and applying \cite[Corollary~2.2.7]{Ham82}; we omit the details.

\subsection{Oka-Grauert principle on compact disks}

Let $\DD=\{z\in \C: \vert z \vert \le 1\}\subset \C$.

\begin{lemma}\label{ogdisk} Let $a\in C^\infty(\DD,\C^{n\times n})$ $(n\in \N)$. Then there exist $GL(n,\C)$-valued solutions $f,g\in C^\infty(\DD,GL(n,\C))$ to the equations
\begin{equation}\label{ogdisk1}
\partial_{\bar z} f + a f = 0 \quad \text{ and }  \quad \partial_{\bar z}g - g a = 0 \quad \text{ on } \DD.
\end{equation}
Moreover, if $a=a(p,\cdot)$ smoothly depends on a parameter $p$ in some manifold $\mathcal{P}$ -- that is, $a\in C^\infty( \mathcal{P}\times \DD,\C^{n\times n})$ -- then there are corresponding solutions $f=f(p,\cdot)$ and $g=g(p,\cdot)$ in $C^\infty(\mathcal{P} \times \DD,GL(n,\C))$.
\end{lemma}

This classical result is dicussed e.g.~in \cite{Mal84} -- we include a brief sketch of its proof and refer to the just cited monograph for further background and details.

\begin{proof}[Proof]
It suffices to solve the second equation in \eqref{ogdisk1}, the first one is then solved by $f=g^{-1}$. We can extend $a$ to a function $a\in C^\infty(\mathcal{P}\times \C,\C^{n\times n})$ and cover $\DD$ by translates of the box $[0,\epsilon]^2\subset \C$. For $\epsilon>0$ sufficiently small, $GL(n,\C)$-valued solutions $g_1,\dots,g_{m(\epsilon)}$, defined in neighbourhoods of the boxes, can be constructed by means of a scaling argument and a Neumann series. By Cartan's lemma, these local solutions can be patched together and, when restricted to $\DD$, yield the desired solution $g$. For more details -- including smooth parameter dependence in case that $\mathcal{P}$ is an open subset of $\R^n$ -- see \cite[Theorem~1, pp.66]{Mal84}. The passage to $\mathcal{P}$ being a manifold follows from standard arguments.  
\end{proof}

\bibliographystyle{plain}
\bibliography{ref.bib}

\begin{thebibliography}{10}

\bibitem{AiAs15}
G.~Ainsworth and Y.~M. Assylbekov.
\newblock On the range of the attenuated magnetic ray transform for connections
  and {H}iggs fields.
\newblock {\em Inverse Probl. Imaging}, 9(2):317--335, 2015.

\bibitem{AMU18}
Y.~M. Assylbekov, F.~Monard, and G.~Uhlmann.
\newblock Inversion formulas and range characterizations for the attenuated
  geodesic ray transform.
\newblock {\em J. Math. Pures Appl. (9)}, 111:161--190, 2018.

\bibitem{AHS78}
M.~F. Atiyah, N.~J. Hitchin, and I.~M. Singer.
\newblock Self-duality in four-dimensional {R}iemannian geometry.
\newblock {\em Proc. Roy. Soc. London Ser. A}, 362(1711):425--461, 1978.

\bibitem{Bohr-thesis}
J.~Bohr.
\newblock Stability, {R}ange and {S}tatistical {A}spects of non-{A}belian
  {X}-ray tomography.
\newblock {\it PhD thesis, University of Cambridge, 2022}.

\bibitem{BLP23}
J.~Bohr, T.~Lefeuvre, and G.~P. Paternain.
\newblock Invariant distributions and the transport twistor space of closed
  surfaces, 2023. arXiv:2301.07391.

\bibitem{BoNi21}
J.~Bohr and R.~Nickl.
\newblock On log-concave approximations of high-dimensional posterior measures
  and stability properties in non-linear inverse problems, 2021.
\newblock arXiv:2105.07835.

\bibitem{DL_20}
N.~M. Desai, W.~R.~B. Lionheart, M.~Sales, M.~Strobl, and S.~Schmidt.
\newblock Polarimetric neutron tomography of magnetic fields: uniqueness of
  solution and reconstruction.
\newblock {\em Inverse Problems}, 36(4):045001, 17, 2020.

\bibitem{DoKr90}
S.~K. Donaldson and P.~B. Kronheimer.
\newblock {\em The geometry of four-manifolds}.
\newblock Oxford Mathematical Monographs. The Clarendon Press, Oxford
  University Press, New York, 1990.
\newblock Oxford Science Publications.

\bibitem{Dub83}
M.~Dubois-Violette.
\newblock Structures complexes au-dessus des vari\'{e}t\'{e}s, applications.
\newblock In {\em Mathematics and physics ({P}aris, 1979/1982)}, volume~37 of
  {\em Progr. Math.}, pages 1--42. Birkh\"{a}user Boston, Boston, MA, 1983.

\bibitem{Esk04}
G.~Eskin.
\newblock On non-abelian {R}adon transform.
\newblock {\em Russ. J. Math. Phys.}, 11(4):391--408, 2004.

\bibitem{EsRa04}
G.~Eskin and J.~Ralston.
\newblock On the inverse boundary value problem for linear isotropic elasticity
  and {C}auchy-{R}iemann systems.
\newblock In {\em Inverse problems and spectral theory}, volume 348 of {\em
  Contemp. Math.}, pages 53--69. Amer. Math. Soc., Providence, RI, 2004.

\bibitem{FiUh01}
D.~Finch and G.~Uhlmann.
\newblock The x-ray transform for a non-abelian connection in two dimensions.
\newblock volume~17, pages 695--701. 2001.
\newblock Special issue to celebrate Pierre Sabatier's 65th birthday
  (Montpellier, 2000).

\bibitem{For11}
F.~Forstneri\v{c}.
\newblock {\em Stein manifolds and holomorphic mappings}, volume~56 of {\em
  Ergebnisse der Mathematik und ihrer Grenzgebiete. 3. Folge. A Series of
  Modern Surveys in Mathematics [Results in Mathematics and Related Areas. 3rd
  Series. A Series of Modern Surveys in Mathematics]}.
\newblock Springer, Heidelberg, 2011.
\newblock The homotopy principle in complex analysis.

\bibitem{Gra58}
H.~Grauert.
\newblock Analytische {F}aserungen \"{u}ber holomorph-vollst\"{a}ndigen
  {R}\"{a}umen.
\newblock {\em Math. Ann.}, 135:263--273, 1958.

\bibitem{GK}
V.~Guillemin and D.~Kazhdan.
\newblock Some inverse spectral results for negatively curved {$2$}-manifolds.
\newblock {\em Topology}, 19(3):301--312, 1980.

\bibitem{Ham82}
R.~S. Hamilton.
\newblock The inverse function theorem of {N}ash and {M}oser.
\newblock {\em Bull. Amer. Math. Soc. (N.S.)}, 7(1):65--222, 1982.

\bibitem{nature2}
A.~Hilger, I.~Manke, and N.~et~al. Kardjilov.
\newblock Tensorial neutron tomography of three-dimensional magnetic vector
  fields in bulk materials.
\newblock {\em Nat Commun 9, 4023}, 2018.

\bibitem{Hit80}
N.~J. Hitchin.
\newblock Complex manifolds and {E}instein's equations.
\newblock In {\em Twistor geometry and nonlinear systems ({P}rimorsko, 1980)},
  volume 970 of {\em Lecture Notes in Math.}, pages 73--99. Springer,
  Berlin-New York, 1982.

\bibitem{LeB80}
C.~LeBrun.
\newblock Spaces of complex geodesics and related structures.
\newblock {\it PhD thesis, University of Oxford, 1980}.

\bibitem{LeMa02}
C.~LeBrun and L.~J. Mason.
\newblock Zoll manifolds and complex surfaces.
\newblock {\em J. Differential Geom.}, 61(3):453--535, 2002.

\bibitem{LM10}
C.~LeBrun and L.~J. Mason.
\newblock Zoll metrics, branched covers, and holomorphic disks.
\newblock {\em Comm. Anal. Geom.}, 18(3):475--502, 2010.

\bibitem{Lei86}
J.~Leiterer.
\newblock Holomorphic vector bundles and the {O}ka-{G}rauert principle.
\newblock In {\em Current problems in mathematics. {F}undamental directions,
  {V}ol. 10 ({R}ussian)}, Itogi Nauki i Tekhniki, pages 75--121, 283. Akad.
  Nauk SSSR, Vsesoyuz. Inst. Nauchn. i Tekhn. Inform., Moscow, 1986.
\newblock Translated by D. N. Akhiezer.

\bibitem{Mal84}
B.~Malgrange.
\newblock {\em Lectures on the theory of functions of several complex
  variables}, volume~13 of {\em Tata Institute of Fundamental Research Lectures
  on Mathematics and Physics}.
\newblock Distributed for the Tata Institute of Fundamental Research, Bombay;
  by Springer-Verlag, Berlin, 1984.
\newblock Reprint of the 1958 edition, Notes by Raghavan Narasimhan.

\bibitem{Mas06}
L.~J. Mason.
\newblock Global anti-self-dual {Y}ang-{M}ills fields in split signature and
  their scattering.
\newblock {\em J. Reine Angew. Math.}, 597:105--133, 2006.

\bibitem{MaWo96}
L.~J. Mason and N.~M.~J. Woodhouse.
\newblock {\em Integrability, self-duality, and twistor theory}, volume~15 of
  {\em London Mathematical Society Monographs. New Series}.
\newblock The Clarendon Press, Oxford University Press, New York, 1996.
\newblock Oxford Science Publications.

\bibitem{Met21}
T.~Mettler.
\newblock Metrisability of projective surfaces and pseudo-holomorphic curves.
\newblock {\em Math. Z.}, 298(1-2):69--78, 2021.

\bibitem{MePa20}
T.~Mettler and G.~P. Paternain.
\newblock Convex projective surfaces with compatible {W}eyl connection are
  hyperbolic.
\newblock {\em Anal. PDE}, 13(4):1073--1097, 2020.

\bibitem{MNP19}
F.~Monard, R.~Nickl, and G.~P. Paternain.
\newblock Consistent inversion of noisy non-{A}belian {X}-ray transforms.
\newblock {\em Comm. Pure Appl. Math.}, 74(5):1045--1099, 2021.

\bibitem{MNP20}
F.~{Monard}, R.~{Nickl}, and G.~P. {Paternain}.
\newblock Statistical guarantees for {B}ayesian uncertainty quantification in
  nonlinear inverse problems with {G}aussian process priors.
\newblock {\em Ann. Statist.}, 49(6):3255--3298, 2021.

\bibitem{Nov02}
R.~G. Novikov.
\newblock On determination of a gauge field on {$\Bbb R^d$} from its
  non-abelian {R}adon transform along oriented straight lines.
\newblock {\em J. Inst. Math. Jussieu}, 1(4):559--629, 2002.

\bibitem{Nov19}
R.~G. Novikov.
\newblock Non-abelian radon transform and its applications.
\newblock {\em R. Ramlau, O. Scherzer. The Radon Transform: The First 100 Years
  and Beyond}, pages 115--128, 2019.

\bibitem{OR85}
N.~R. O'Brian and J.~H. Rawnsley.
\newblock Twistor spaces.
\newblock {\em Ann. Global Anal. Geom.}, 3(1):29--58, 1985.

\bibitem{PaSa20}
G.~P. Paternain and M.~Salo.
\newblock The non-{A}belian {X}-ray transform on surfaces, 2020.
\newblock arXiv:2006.02257, {\it to appear in J. Differential Geom.}

\bibitem{PSU12}
G.~P. Paternain, M.~Salo, and G.~Uhlmann.
\newblock The attenuated ray transform for connections and {H}iggs fields.
\newblock {\em Geom. Funct. Anal.}, 22(5):1460--1489, 2012.

\bibitem{PSU15}
G.~P. Paternain, M.~Salo, and G.~Uhlmann.
\newblock On the range of the attenuated ray transform for unitary connections.
\newblock {\em Int. Math. Res. Not. IMRN}, (4):873--897, 2015.

\bibitem{PSU20}
G.~P. Paternain, M.~Salo, and G.~Uhlmann.
\newblock {\em Geometric inverse problems---with emphasis on two dimensions},
  volume 204 of {\em Cambridge Studies in Advanced Mathematics}.
\newblock Cambridge University Press, Cambridge, 2023.
\newblock With a foreword by Andr\'{a}s Vasy.

\bibitem{PSUZ19}
G.~P. Paternain, M.~Salo, G.~Uhlmann, and H.~Zhou.
\newblock The geodesic {X}-ray transform with matrix weights.
\newblock {\em Amer. J. Math.}, 141(6):1707--1750, 2019.

\bibitem{Pen77}
R.~Penrose.
\newblock The twistor programme.
\newblock {\em Rep. Mathematical Phys.}, 12(1):65--76, 1977.

\bibitem{PeUh04}
L.~Pestov and G.~Uhlmann.
\newblock On characterization of the range and inversion formulas for the
  geodesic {X}-ray transform.
\newblock {\em Int. Math. Res. Not.}, (80):4331--4347, 2004.

\bibitem{PeUh05}
L.~Pestov and G.~Uhlmann.
\newblock Two dimensional compact simple {R}iemannian manifolds are boundary
  distance rigid.
\newblock {\em Ann. of Math. (2)}, 161(2):1093--1110, 2005.

\bibitem{PS86}
A.~Pressley and G.~Segal.
\newblock {\em Loop groups}.
\newblock Oxford Mathematical Monographs. The Clarendon Press, Oxford
  University Press, New York, 1986.
\newblock Oxford Science Publications.

\bibitem{nature1}
M.~Sales, M.~Strobl, and T.~et~al. Shinohara.
\newblock Three dimensional polarimetric neutron tomography of magnetic fields.
\newblock {\em Nat Sci Rep 8, 2214}, 2018.

\bibitem{SaUh11}
M.~Salo and G.~Uhlmann.
\newblock The attenuated ray transform on simple surfaces.
\newblock {\em J. Differential Geom.}, 88(1):161--187, 2011.

\bibitem{See64}
R.~T. Seeley.
\newblock Extension of {$C^{\infty }$} functions defined in a half space.
\newblock {\em Proc. Amer. Math. Soc.}, 15:625--626, 1964.

\bibitem{Sha94}
V.~A. Sharafutdinov.
\newblock {\em Integral geometry of tensor fields}.
\newblock Inverse and Ill-posed Problems Series. VSP, Utrecht, 1994.

\bibitem{Sha00}
V.~A. Sharafutdinov.
\newblock On the inverse problem of determining a connection on a vector
  bundle.
\newblock {\em J. Inverse Ill-Posed Probl.}, 8(1):51--88, 2000.

\bibitem{SiTh76}
I.~M. Singer and J.~A. Thorpe.
\newblock {\em Lecture notes on elementary topology and geometry}.
\newblock Undergraduate Texts in Mathematics. Springer-Verlag, New
  York-Heidelberg, 1976.
\newblock Reprint of the 1967 edition.

\bibitem{SUV21}
P.~Stefanov, G.~Uhlmann, and A.~Vasy.
\newblock Local and global boundary rigidity and the geodesic {X}-ray transform
  in the normal gauge.
\newblock {\em Ann. of Math. (2)}, 194(1):1--95, 2021.

\bibitem{UhVa16}
G.~Uhlmann and A.~Vasy.
\newblock The inverse problem for the local geodesic ray transform.
\newblock {\em Invent. Math.}, 205(1):83--120, 2016.

\bibitem{Ver91}
L.~B. Vertgeim.
\newblock Integral geometry with a matrix weight and a nonlinear problem of the
  reconstruction of matrices.
\newblock {\em Dokl. Akad. Nauk SSSR}, 319(3):531--534, 1991.

\end{thebibliography}

\end{document}